\documentclass{amsart}
\usepackage[english]{babel}
\usepackage[hidelinks]{hyperref}
\usepackage{amsmath,amsfonts,amssymb}
\usepackage{tikz}
\usepackage{tkz-euclide} 
\usetikzlibrary{calc}
\usetikzlibrary{arrows} 
\usetikzlibrary{decorations.markings}
\usetikzlibrary{shapes.geometric}
\usetikzlibrary{fit,positioning,
	calc,arrows,
	backgrounds, 
	shapes.misc,
	snakes,
	backgrounds,
	decorations.pathmorphing,
	decorations.pathreplacing,
	calligraphy}
\usepackage{pgfplots}
\usepackage{amsthm}
\usepackage[all]{xy}
\usepackage{ stmaryrd }
\usepackage{ mathrsfs }
\theoremstyle{plain} 

\usepackage{epsfig}
\usepackage{lmodern}
\usepackage{eurosym}
\usepackage{graphicx}

\definecolor{MidnightBlue}{rgb}{0.1, 0.1, 0.44}
\definecolor{Plum}{rgb}{0.56, 0.27, 0.52}

\hypersetup{
	colorlinks=true,
	linkcolor=MidnightBlue,
	citecolor=Plum,
	filecolor=magenta,      
	urlcolor=cyan,
	pdfpagemode=FullScreen,
}

\newcommand{\ccong}{\mathbin{\rotatebox[origin=c]{90}{$\cong$}}}

\usepackage{stmaryrd}
\usepackage{mathrsfs}
\author[B.~Dequêne]{Benjamin Dequêne}
\address[B.~Dequêne]{Département de mathématiques, LaCIM, Université du Québec à Montréal}
\email{dequene.benjamin@courrier.uqam.ca}

\title[Jordan recoverability of some subcategories]{Jordan Recoverability of some subcategories of modules over gentle algebras}

\newtheorem{theorem}{Theorem}[section]
\newtheorem{lemma}[theorem]{Lemma}
\newtheorem{prop}[theorem]{Proposition}
\newtheorem{cor}[theorem]{Corollary}
\newtheorem{conj}[theorem]{Conjecture}

\theoremstyle{definition}
\newtheorem{definition}[theorem]{Definition}
\newtheorem{example}[theorem]{Example}
\newenvironment{ex}    
{%
	\pushQED{\qed}\begin{example}}
	{\popQED\end{example}}

\theoremstyle{remark}
\newtheorem{remark}[theorem]{Remark}

\newcommand{\Hom}{\operatorname{Hom}}

\newcommand{\vdim}{\mathbf{dim}}

\usepackage{todonotes}

\newcommand{\new}[1]{\textit{\color{Plum}{#1}}}
\begin{document}
	
\begin{abstract}
	Gentle algebras form a class of finite-dimensional algebras introduced by I. Assem and A. Skowro\'{n}ski in the 1980s. Modules over such an algebra can be described by string and band combinatorics in the associated gentle quiver from the work of M.C.R. Butler and C.M. Ringel. Any module can be naturally associated to a quiver representation. A nilpotent endomorphism of a quiver representation induces linear transformations over vector spaces at each vertex. Generically among all nilpotent endomorphisms, a well-defined Jordan form exists for these representations. If $\mathcal{Q} = (Q, I)$ is a finite connected gentle quiver, we show a characterization of the vertices $m$ such that representations consisting of direct sums of indecomposable representations all including $m$ in their support, the subcategory of which we denote by $\mathscr{C}_{\mathcal{Q},m}$, are determined up to isomorphism by this invariant. 
\end{abstract}
	
	\maketitle

	\tableofcontents
	
	\section{Introduction}
	\label{s:Intro1}
	Let $\mathcal{Q} = (Q, I)$ be a finite connected gentle quiver, i.e. a finite connected quiver $Q$ provided with an admissible ideal $I$ that satisfies some additional conditions that are defined later in the next section. They were introduced by I. Assem and A. Skowro\'{n}ski \cite{AS87}. A lot of recent studies are still being made nowadays of this interesting class of algebras: see for example \cite{APS19}, \cite{OPS18}, \cite{BS18}, \cite{PPP18}, \cite{FGLZ21}, \cite{BDMT19}. As a good reference to learn some basics about what is a gentle algebra and some of the underlying combinatoric behaviors that we will also see throughout the article, we can suggest to the reader the following reference: \cite{PPP182}.
	
	Let $X$ be a finite-dimensional representation of $\mathcal{Q}$ over an algebraically closed field $\mathbb{K}$. Let $N$ be a nilpotent endomorphism of $X$. At each vertex $q \in Q_0$, $N$ induces a nilpotent endomorphism $N_q$ of $X_q$. Thus we can extract from $N$ a sequence of integer partitions $\lambda^q \vdash \dim(X_q)$ which correspond to the Jordan blocks size of the Jordan form of each $N_q$. Thanks to a result from \cite{GPT19}, $\underline{\lambda} = (\lambda^q)_{q \in Q_0}$ is well-defined for a generic choice of $N$. We denote $\mathsf{GenJF}(X)$ this tuple that we will refer to the \new{generic Jordan form data} of $X$.
	
	An interesting question that we could ask ourselves is:
	\begin{center}
		\textbf{ ``For which subcategories $\mathscr{C}$ of the category $\mathsf{rep}(\mathcal{Q})$ could we recover up to isomorphism $X \in \mathscr{C}$ from $\mathsf{GenJF}(X)$?"}
	\end{center}
	Such a subcategory is called \new{Jordan recoverable}. Throughout the article, we mean a full subcategory closed under direct sums and direct summands by a subcategory. 
	
	Let us see the following toy example.
	\begin{ex} \label{A2ex1}
		Let $\mathcal{Q}$ be the following $A_2$-type quiver.
		\begin{center}
			$\mathcal{Q} = \xymatrix{1 \ar[r]^a & 2 }$
		\end{center}
		The indecomposable representations of $\mathcal{Q}$ are 
		\begin{center}
			$S_1 = \xymatrix{\mathbb{K} \ar[r] & 0}$ \qquad $S_2 = \xymatrix{0 \ar[r] & \mathbb{K}}$  \qquad $P_1 = \xymatrix{\mathbb{K} \ar[r]^1 & \mathbb{K}}$ 
		\end{center}
		We want to know which of the subcategories generated by those indecomposable representations are Jordan recoverable and which are not:
		\begin{itemize}
			\item[•] Let $X = S_1^n$ and let $N = (N_1,N_2)$ be a nilpotent endomorphism of $X$. We can see we have no restriction to define $N_1$. This implies that if we choose $N_1$ generically then $N_1$ is similar to a Jordan block of size $n$. So $\mathsf{GenJF}(X) = ((n),(0))$ and clearly $\mathsf{add}(S_1)$ is Jordan recoverable;
			
			\item[•] For the same reason $\mathsf{add}(S_2)$ is Jordan recoverable too;
			
			\item[•] Let $X = P_1^m$ and $N = (N_1,N_2)$ be a nilpotent endomorphism of $X$. By the commutative square relation that $N$ has to satisfy ($X_a N_1 = N_2 X_a$), and as $X_a$ is a bijective map, $N_1$ and $N_2$ can be identified. Moreover we note there is no restriction to define $N_1$. So $\mathsf{GenJF}(X) = ((m),(m))$ and $\mathsf{add}(P_1)$ is Jordan recoverable;
			
			\item[•] Let $X = S_1^n \oplus S_2^k$ and $N = (N_1,N_2)$ be a nilpotent endomorphism of $X$. We can see we have no restriction to define $N_1$ and $N_2$, and they can be defined independently of one another. Thus $\mathsf{GenJF}(X) = ((n),(k))$ and $\mathsf{add}(S_1,S_2)$ is Jordan recoverable;
			
			\item[•] Let $X = S_2^k \oplus P_1^m$ and $N = (N_1,N_2)$ be a nilpotent endomorphism of $X$. We can see thanks to commutative square relations that if we define $N_2$, we have defined $N_1$. Moreover if we adapt a choice of $N_1$ such that it corresponds to a choice of $N_2$, we can see we have no restriction to define $N_2$. Hence $\mathsf{GenJF}(X) = ((m),(m+k))$ and $\mathsf{add}(S_2, P_1)$ is Jordan recoverable;
			
			\item[•] By an analogous argument we have $\mathsf{add}(S_1,P_1)$ Jordan recoverable;
			
			\item[•] However, $\mathsf{add}(S_1,S_2,P_1) = \mathsf{rep}(\mathcal{Q})$ is not Jordan recoverable. Indeed this comes from the fact that $\mathsf{GenJF}(S_1 \oplus S_2) = ((1),(1)) = \mathsf{GenJF}(P_1)$ by the previous results. \qedhere
		\end{itemize}
	\end{ex}
	Obviously a subcategory where the dimension vectors of its indecomposable elements form a linearly independent family is Jordan recoverable. Indeed that comes from the fact that any representation of this kind of subcategory is entirely determined by its dimension vectors. However, in general, it is not an easy question to determine which subcategories of $\mathsf{rep}(\mathcal{Q})$ are Jordan recoverable.
	
	For this purpose, we can ask for something a bit more restrictive. There is a procedure that, at least in some cases, allows one to reconstruct $X$ from $\mathsf{GenJF}(X)$: we can see if there exists a generic choice of a representation $Y$ in $\mathsf{rep}(\mathcal{Q})$ among representations that admit a nilpotent endomorphism with Jordan forms encoded by $\mathsf{GenJF}(X)$, and then we can ask if $Y$ is isomorphic to $X$. 
	
	In more concrete words, for a fixed subcategory $\mathscr{C}$ of $\mathsf{rep}(\mathcal{Q})$, we can ask ourselves if for any $X \in \mathscr{C}$ there is a (Zariski) dense open set $O$ in the variety of representations which admit a nilpotent endomorphism with a Jordan form equal to the generic Jordan form data of $X$, such that any $Y \in O$ is isomorphic to $X$. Such a subcategory $\mathscr{C}$ is said to be \new{canonically Jordan recoverable}. 
	
	Note that a subcategory that is Jordan recoverable is not necessarily canonically Jordan recoverable.
	\begin{ex} \label{JRnotCJRex}
		As in the Example \ref{A2ex1}, let $\mathcal{Q} = \xymatrix{1 \ar[r] & 2}$. Let $\mathscr{C} = \mathsf{add}(S_1, S_2)$. We saw that $\mathscr{C}$ is Jordan recoverable, but it is not canonically Jordan recoverable. If we take $X = S_1 \oplus S_2$ then we get $\mathsf{GenJF}(X) = ((1),(1))$, and we already saw that there is an other representation that has the same Jordan form data: it is $P_1$! For $Y \in\mathsf{rep}(\mathcal{Q})$ admitting a nilpotent endomorphism with a Jordan form given by $((1),(1))$, we have the following diagram to satisfy
		\begin{center}
			$\displaystyle \xymatrix@R=1em @C=1em{
				\mathbb{K}\ar[rr]^\alpha \ar@{.>}[dd]^0 & & \mathbb{K} \ar@{.>}[dd]^0 \\
				& & \\
				\mathbb{K} \ar[rr]^\alpha  & & \mathbb{K}  }$
		\end{center}
		and so a generic choice that we have for the linear morphism $\alpha$ is $k\ \mathsf{id}$ for $0 \neq k \in \mathbb{K}$. Hence $Y \cong P_1 \ncong S_1 \oplus S_2$.\qedhere
	\end{ex}
	Let us see another interesting example.
	\begin{ex} \label{1stGentlex}
		Let us consider the following gentle quiver. 
		\begin{center}
			$\displaystyle \mathcal{Q} = \xymatrix@R=1em @C=1em{1 \ar[rr]_a & \ar@{--}@/^1pc/[rr] & 2 \ar[rr]_b & & 3}$
		\end{center}
		Let $\displaystyle X = \xymatrix@R=1em @C=1em{\mathbb{K} \ar[rr]_1 & \ar@{--}@/^1pc/[rr] & \mathbb{K} \ar[rr] & & 0} \oplus \xymatrix@R=1em @C=1em{0 \ar[rr] & \ar@{--}@/^1pc/[rr] & \mathbb{K} \ar[rr] & & 0} \oplus \xymatrix@R=1em @C=1em{0 \ar[rr] & \ar@{--} @/^1pc/[rr] & \mathbb{K} \ar[rr]_1 & & \mathbb{K}}$. Let $N = (N_1, N_2, N_3)$ be a nilpotent endomorphism of $X$. It is effortless to check that $N_1 = N_3 = 0$. Let $(x_1,x_2,x_3)$ be a basis of $\mathbb{K}^3$, the vector space at the vertex $2$, adapted to the decomposition of $X$ into indecomposable representations above. That means for example that we identify the vector space on the vertex $1$ to the subspace on the vertex $2$ generated by $x_1$.
		
		From the commutative square relations, we need that $\langle x_1 \rangle \subseteq \mathsf{Ker}(N_2)$, $\mathsf{Im}(N_2) \subseteq \langle x_1, x_2 \rangle$ and of course $N_2^3 = 0$. So:
		\begin{itemize}
			\item[•] $N_2(x_1) = 0$;
			
			\item[•] $N_2(x_2) = \alpha x_1 + \beta x_2$ for some $\alpha, \beta \in \mathbb{K}$;
			
			\item[•] $N_2(x_3) = \gamma x_1 + \delta x_2$ for some $\gamma, \delta \in \mathbb{K}$.
		\end{itemize}
		In addition, we can remark that $N_2^2(x_2) = \beta N_2(x_2)$. This implies that $\beta = 0$.
		
		Hence all the allowed choices of $N_2$ can be characterized as matrices of the form
		\begin{center}
			$N(\alpha, \gamma,\delta) = \left( \begin{matrix}
				0 & \alpha & \gamma \\
				0 & 0 & \delta \\
				0 & 0 & 0
			\end{matrix} \right)$
		\end{center}
		Thus, among all the choices we can make to define $N_2$ (and so $N$), there is a generic choice (a dense open set) which comes naturally from generic restrictions $\alpha, \delta \neq 0$ that we can impose. So the Jordan form of that generic choice of $N$ is $((1),(3),(1))$. Then $\mathsf{GenJF}(X) = ((1),(3),(1))$.
		
		Suppose now we have $Y$ such that $\mathsf{GenJF}(Y) = ((1),(3),(1))$. We know $Y$ is of the form:
		\begin{center}
			$\displaystyle \xymatrix@R=1em @C=1em{\mathbb{K} \ar[rr]_{Y_a} & \ar@{--}@/^1pc/[rr] & \mathbb{K}^3 \ar[rr]_{Y_b} & & \mathbb{K}}$
		\end{center} 
		Let $(x)$, $(y_1, y_2,y_3)$ and $(z)$ be respective bases of $Y_1, Y_2$ and $Y_3$ adapted to each nilpotent endomorphism. The following diagram summarizes the situation.
		\begin{center}
			$\displaystyle \xymatrix @R=1.5em @C=4em{
				Y_1 & Y_2 & Y_3  \\
				& y_1 \ar[d]^{N_2} \ar@{-->}[r]^{Y_b}& *++[o][F.]{z} \ar[d]^{N_3} \\
				& y_2 \ar[d]^{N_2} & 0 \\
				x \ar[d]^{N_1} \ar@{-->}[r]^{Y_a} & *++[o][F.]{y_3} \ar[d]^{N_2} &  \\
				0 & 0 & }$
		\end{center}
		By commutative square relations ($Y_a N_1 = N_2 Y_a$ and $Y_b N_2 = N_3 Y_b$), we have:
		\begin{itemize}
			\item[•] $Y_a(x) = \alpha' y_3$ for some $\alpha' \in \mathbb{K}$;
			
			\item[•] $Y_b(y_2) = Y_b(y_3) = 0$ and $Y_b(y_1) = \beta' z$ for some $\beta' \in \mathbb{K}$.
		\end{itemize}
		Thus among all the choices we can make to define $Y_a$ and $Y_b$ (and so $Y$), there exists a generic choice that comes from generic restrictions $\alpha', \beta' \neq 0$. Hence we get $Y \cong X$. We recovered $X$.
	\end{ex}
	Some remarks:
	\begin{itemize}
		\item[•] We get for free $Y_b Y_a = 0$, which is something that seems not guaranteed in the first instance;
		
		\item[•] From what we have done above, if we started with $Z = S_1 \oplus S_2^3 \oplus S_3$, which has the same generic Jordan form data as $X$, we would not recover $Z$. Hence the category $\mathsf{add}(S_1, S_2, S_3)$ is not canonically Jordan recoverable, even though it is Jordan recoverable;
		
		\item[•] If we try to get a representation $Y$ such that its generic Jordan form data is $((1),(1),(1))$, we will have no restriction from the nilpotent endomorphism (as $N=0$) other than $Y_q \cong \mathbb{K}$ for each vertex $q$. But we need to respect the relation $Y_b Y_a = 0$. So we get two isomorphism classes of representations that could be good candidates for a generic choice:
		\begin{center}
			$\displaystyle \xymatrix @R=1em @C=1em{ 
				\mathbb{K} \ar[rr]_{1} & \ar@{--} @/^1pc/ [rr] & \mathbb{K} \ar[rr]_{0}  & & \mathbb{K}  }$ \qquad  $\displaystyle \xymatrix @R=1em @C=1em{ 
				\mathbb{K} \ar[rr]_{0} & \ar@{--} @/^1pc/ [rr] & \mathbb{K} \ar[rr]_{1}  & & \mathbb{K}  }$
		\end{center}
		However, there is no reason to privilege one to the other: the open sets in play are not dense in the set of representations that have $((1),(1),(1))$ as generic Jordan form. Hence there is no generic choice of a representation $Y$ with that generic Jordan form;
		
		\item[•] We give in Example \ref{CJRex} a proof that $\mathsf{add}(I_2, S_2, P_2)$ is canonically Jordan recoverable. The crucial point in this subcategory is we can note the dimension vector $(d_1, d_2, d_3)$ of any representation $X$ satisfies the relation $d_2 \geqslant d_1 + d_3$. This is something that the representations in the previous bullet point do not satisfy which explains why the previous bullet point is not a counter-example to this claim.
	\end{itemize}
	We will study Jordan recoverability and canonical Jordan recoverability of a special kind of subcategories. For $m$ a vertex of $\mathcal{Q}$, let $\mathscr{C}_{\mathcal{Q},m}$ be the subcategory generated by the indecomposable representations that have a nonzero vector space at the vertex $m$.
	
	One of the engrossing notions around the vertices of a quiver that seems important to achieve our goal is the minusculeness of a vertex. A vertex $m$ is said to be \new{minuscule} if for any indecomposable representation $X$ of $\mathcal{Q}$, we have $\dim X_m \leqslant 1$. This notion is a crucial point to get $\mathscr{C}_{\mathcal{Q},m}$ canonically Jordan recoverable by the following main result of \cite{GPT19}.
	\begin{theorem}[\cite{GPT19}] If $Q$ is a Dynkin quiver, and $m$ is a minuscule vertex, then $\mathscr{C}_{\mathcal{Q},m}$ is canonically Jordan recoverable.
	\end{theorem}
	\begin{remark}
		The result presents the minuscule property as a sufficient condition, but not a necessary one. For type $A$ quivers, note that all the vertices are minuscule.
	\end{remark}
	In this article, we give an analogous result (which extends the $A_n-$type quiver case) in the case of a finite connected gentle quiver $\mathcal{Q}$. We note that the condition of minusculeness is not enough to get Jordan recoverability or canonical Jordan recoverability. We need to add some combinatorial conditions, linked to strings of $\mathcal{Q}$.
	
	A \new{string} of $\mathcal{Q}$ can be understood as a walk in the quiver. These walks (satisfying some additional conditions as we shall recall below) allow us to describe all the indecomposable representations of $\mathcal{Q}$. By knowing what is going on with strings, we can give a complete description of all the subcategories $\mathscr{C}_{\mathcal{Q},m}$ that are Jordan recoverable or canonically Jordan recoverable. 
	
	In this article, we first recall some main definitions and tools that will then help us to prove the following statements. The first main result is a characterization of the subcategories $\mathscr{C}_{\mathcal{Q},m}$ that are canonically Jordan recoverable.
	\begin{theorem} \label{main}
		Let $\mathcal{Q} = (Q,I)$ be a finite connected gentle quiver and $m \in Q_0$. The subcategory $\mathscr{C}_{\mathcal{Q},m}$ is canonically Jordan recoverable if and only if $m$ satisfies the following conditions:
		\begin{itemize}
			\item[$(i)$] For any pair of strings $\rho$ and $\nu$ passing through $m$, there is no arrow $\alpha \in Q_1$ which is not an arrow of either string such that $\rho$ passes through $s(\alpha)$ and $\nu$ passes through $t(\alpha)$;
			
			\item[$(ii)$] At least one of the two following conditions :
			\begin{itemize}
				\item[$(a)$] There is at most one arrow $\alpha \in Q_1$ such that $s(\alpha) = m$, and there is at most one arrow $\beta \in Q_1$ such that $t(\beta) = m$ ; if $\alpha$ and $\beta$ both exist then $\alpha \beta \in I$;
				
				\item[$(b)$] There is at most one maximal string of $\mathcal{Q}$ passing through $m$. 
			\end{itemize}		
		\end{itemize}
	\end{theorem}
	Following the proof given later (See Section \ref{sub:proof1}), we can rephrase the previous theorem in the following way: $\mathscr{C}_{\mathcal{Q},m}$ is canonically Jordan recoverable if and only if the underlying graph of the full subquiver of $Q$ given by the vertices $q$ attained by strings passing through $m$ is a tree $(i)$, and furthermore, if this tree is not a line $(ii)(a)$, either $m$ is a leaf of that tree, or all the incident arrows to $m$ are in relation $(ii)(b)$.
	
	The second one is some kind of consequence of the first, but it needs a few additional works. It allows us to get a characterization of the subcategories $\mathscr{C}_{\mathcal{Q},m}$ that are Jordan recoverable.
	\begin{theorem}\label{2ndmain}
		Let $\mathcal{Q} = (Q,I)$ be a finite connected gentle quiver and $m \in Q_0$. The subcategory $\mathscr{C}_{\mathcal{Q},m}$ is Jordan recoverable if and only if $m$ satisfies the point $(ii)$ of Theorem \ref{main} and the following condition:
		\begin{itemize}
			\item[$(i*)$] Suppose that there exist two strings $\rho$ and $\nu$ passing through $m$, and an arrow $\alpha \in Q_1$ which belongs to neither $\rho$ nor $\nu$, but such that $\rho$ passes through $s(\alpha)$ and $\nu$ passes through $t(\alpha)$. Then there is no string passing through $m$ supported on $\alpha$.
		\end{itemize}
	\end{theorem}
	Following the proof given later (See Section \ref{sub:proof2}), we can rephrase the previous theorem in the following way: $\mathscr{C}_{\mathcal{Q},m}$ is Jordan recoverable if and only if the underlying graph of the subquiver of $Q$ given by arrows and  vertices that we go through by strings passing through $m$ is a tree $(i*)$, that satisfies the same property than in the previous result. Note that it is a weaker condition than in Theorem \ref{main}, as cyclic walks are allowed in the full subquiver of $Q$ mentioned previously. If so, then there must be an arrow $\alpha_w$ in each walk $w$ such that $s(\alpha_w) \neq t(\alpha_w)$ and there are no strings passing through $m$ supported at $\alpha_w$.
	
	\section{Background}
	
	\subsection{Gentle algebras and representations}
	
	A \new{quiver} $Q$ is a $4$-tuple of the form $(Q_0,Q_1,s,t)$ where $Q_0$ is \new{the set of vertices}, $Q_1$ is \new{the set of arrows} and $s,t : Q_1 \longrightarrow Q_0$ are respectively \new{source} and \new{target functions}. A quiver is called \new{finite} if $\#Q_0,\#Q_1 < \infty$. Throughout this article we will suppose that all the quivers we will work with are finite and connected. A \new{path} of $Q$ is a finite sequence of arrows $\alpha_1, \alpha_2, \ldots, \alpha_n \in Q_1$ such that $s(\alpha_{i+1}) = t(\alpha_i)$. We will write $p = \alpha_n \ldots \alpha_1$ for a path, and for $q \in Q_0$, $e_q$ for the \new{lazy path} at the vertex $q$, that is the path staying at $q$ without following any arrow. The \new{length} of the path $p$ is the number of arrows in it. 
	
	Let $\mathbb{K}$ be a field. Throughout this article, we will suppose that $\mathbb{K}$ is algebraically closed. This is a restriction we need to use the main result of \cite{GPT19}. They need it because some of their arguments rely on algebraic geometry. The \new{path algebra} denoted $\mathbb{K}Q$ is the $\mathbb{K}$-vector space having as a basis all the paths in the quiver $Q$, together with multiplication given by concatenation of paths. For $l \geqslant 0$, let denote $\mathbb{K}Q_{\geqslant l}$ the ideal of $\mathbb{K}Q$ generated by paths of $Q$ of length at least $l$. An \new{admissible ideal} $I$ of $\mathbb{K}Q$ is an ideal such that: $\exists N > 0/\ \mathbb{K}Q_{\geqslant N} \subseteq I \subseteq \mathbb{K}Q_{\geqslant 2}$.
	\begin{definition} \label{gentledef}
		A \new{gentle algebra} is a quotient algebra $\mathbb{K}Q/I$ where:
		\begin{itemize}
			\item[•] $Q$ is a quiver such that there are at most two incoming arrows and at most two outgoing arrows at each vertex.
			
			\item[•] $I$ is an admissible ideal generated by some paths of length $2$ such that for every arrow $\alpha \in Q_1$ there exists:
			\begin{itemize}
				\item[$*$] at most one arrow $\beta$ such that $s(\alpha) = t(\beta)$ and $\alpha \beta \in I$;
				
				\item[$*$] at most one arrow $\gamma$ such that $s(\gamma) = t(\alpha)$ and $\gamma \alpha \in I$;
				
				\item[$*$] at most one arrow $\beta'$ such that $s(\alpha) = t(\beta')$ and $\alpha \beta' \notin I$;
				
				\item[$*$] at most one arrow $\gamma'$ such that $s(\gamma') = t(\alpha)$ and $\gamma' \alpha \notin I$.
			\end{itemize}
		\end{itemize}
		We will denote a \new{gentle quiver} by $\mathcal{Q} = (Q,I)$ a quiver together with an ideal such that the algebra $\mathbb{K}\mathcal{Q} := \mathbb{K}Q/I$  is gentle.
	\end{definition}
	\begin{ex} \label{gentleex} Here is a gentle quiver. We will use the convention to represent elements generating the admissible ideal by dashed lines in a gentle quiver. They are all the paths of length two that are in $I$.  
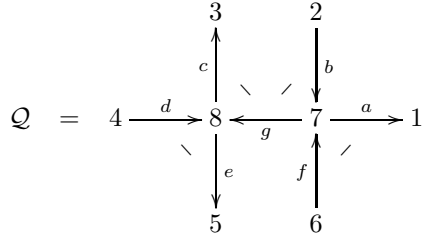
		\begin{center}
			$\displaystyle  \xymatrix @R=0.5em @C=0.5em{ 
				& & & & & 3 & & &  2 \ar[ddd]^b & & & \\
				& & & & & & & & & & \\
				& & & & & & & & & & \\
				\mathcal{Q} & = & 4 \ar[rrr]^d & & \ar@{--} @/_/[rd]& 8 \ar[uuu]^c \ar[ddd]^e & \ar@{--} @/_/[lu]  & \ar@{--} @/^/[ru] & 7 \ar[lll]^g \ar[rrr]^a & \ar@{--} @/^/[ld] & & 1 \\
				& & & & & & & & & & & \\
				& & & & & & & & & & \\
				& & & & & 5 & & & 6 \ar[uuu]^f & & &}$
		\end{center}
		In this example $I = \langle af, gb, cg, ed \rangle$ and so $\mathbb{K}\mathcal{Q}$ is the gentle algebra associated to $\mathcal{Q}$ generated as basis by:
		\begin{itemize}
			\item[•] lazy paths: $e_1$, $e_2$, $\ldots$, $e_8$;
			
			\item[•] arrows: $a$, $b$, $\ldots$, $g$; 
			
			\item[•] paths of length $2$: $cd$, $eg$, $gf$, $ab$;
			
			\item[•] paths of length $3$: $egf$. \qedhere
		\end{itemize}
	\end{ex}
	\begin{remark}
		We can note that $\mathbb{K}Q$ (and $\mathbb{K}\mathcal{Q}$) is graded by the length of the paths.
	\end{remark}
	Let $\mathcal{Q} = (Q,I)$ be a gentle quiver. A \new{representation of} $\mathcal{Q}$ (\new{over} $\mathbb{K}$) is a pair $X = ((X_q)_{q \in Q_0}, (X_\alpha)_{\alpha \in Q_1})$ where $X_q$ is a finite-dimensional $\mathbb{K}$-vector space for each $q \in Q_0$ and $X_\alpha : X_{s(\alpha)} \longrightarrow X_{t(\alpha)}$ is a $\mathbb{K}$-linear function  for each $\alpha \in Q_1$ such that if $\alpha \beta \in I$ then $X_\alpha \cdot X_\beta = 0$. We can understand this construction as placing a $\mathbb{K}$-vector space on each vertex $q$ and a $\mathbb{K}$-linear transformation on each arrow $\alpha$ in $\mathcal{Q}$. We write $\underline{\dim}(X) = (\dim(X_q))_{q\in Q_0}$ the \new{dimension vector of} $X$.
	
	A \new{morphism} $\phi$ \new{between two representations} $X$ \new{and} $Y$ is a tuple of $\mathbb{K}$-linear transformations $(\phi_q)_{q \in Q_0}$ where $\phi_q : X_q \longrightarrow Y_q$ for each $q \in Q_0$ such that, for all $\alpha \in Q_1$, $Y_\alpha \phi_{s(\alpha)} = \phi_{t(\alpha)} X_\alpha$. We can have in mind the following commutative square to understand the previous condition.
	\begin{center}
		$\displaystyle \xymatrix@R=1em @C=1em{
			X_{s(\alpha)} \ar[rr]^{X_\alpha} \ar@{.>}[dd]^{\phi_{s(\alpha)}} & & X_{t(\alpha)} \ar@{.>}[dd]^{\phi_{t(\alpha)}} \\
			& \circlearrowright & \\
			Y_{s(\alpha)} \ar[rr]^{Y_{\alpha}} & & Y_{t(\alpha)}
		}$
	\end{center}
	Let $X$ and $Y$ be two representations of $\mathcal{Q}$. Write $X \cong Y$ whenever $X$ and $Y$ are \new{isomorphic}. Denote $X \oplus Y$ the \new{direct sum of representations} $X$ \new{and} $Y$. A representation $M \neq 0$ is called \new{indecomposable} if  $X \cong 0$ or $Y \cong 0$, whenever $M \cong X \oplus Y$.
	
	Let $\mathsf{rep}(\mathcal{Q})$ denote the category of (finite-dimensional) representations of $\mathcal{Q}$.
	\begin{theorem} There is an equivalence of categories between $\mathsf{rep}(\mathcal{Q})$ and the category of modules over $\mathbb{K}\mathcal{Q}$.
	\end{theorem}
	For more details, see for example \cite[Theorem 1.6 p72]{ASS06}.
	
	From now to the end of this subsection, we will recall the combinatorial description of the indecomposable representations of a gentle algebra. With this in mind, we need to introduce some more tools.
	\begin{definition}\label{stringdef}
		A \new{string} of $\mathcal{Q}$ is a finite sequence $\rho = \alpha_k^{\varepsilon_k} \ldots \alpha_1^{\varepsilon_1}$ with $\alpha_i \in Q_1$ for all $i \in \{ 1, \ldots, k \}$ and $\varepsilon_i \in \{\pm 1\}$  such that:
		\begin{itemize}
			\item[•] by defining $s(\alpha^{-1}) = t(\alpha)$ and $t(\alpha^{-1}) = s(\alpha)$ for any $\alpha \in Q_1$, we have, for all $i \in \{ 1,\ldots, k-1 \}$, $s(\alpha_{i+1}^{\varepsilon_{i+1}}) = t(\alpha_{i}^{\varepsilon_{i}})$;
			
			\item[•] if $\varepsilon_i = \varepsilon_{i+1} = 1$, then $\alpha_{i+1} \alpha_i \notin I$; if $\varepsilon_i = \varepsilon_{i+1} = -1$, then $\alpha_{i} \alpha_{i+1} \notin I$;
			
			\item[•] $\rho$ is reduced: for $i \in \{ 1,\ldots , k-1 \}$, if $\varepsilon_i \varepsilon_{i+1} = -1$, then $\alpha_i \neq \alpha_{i+1}$.
		\end{itemize}
	\end{definition}
	In particular, we can notice that a path is a string.
	
	We write $\ell(\rho) := k$ for the \new{length} of the string, $s(\rho):= s(\alpha_1^{\varepsilon_1})$ and $t(\rho) := t(\alpha_k^{\varepsilon_k})$. We denote $\mathsf{Supp}_0(\rho)$ the \new{vertex support} of $\rho$ which is the set of vertices $\{s(\alpha_1^{\varepsilon_1}), t(\alpha_1^{\varepsilon_1}), t(\alpha_2^{\varepsilon_2}), \ldots, t(\alpha_k^{\varepsilon_k})\}$, and $\mathsf{Supp}_1(\rho)$ the \new{arrow support} of $\rho$ which is the set of arrows $\{\alpha_1, \ldots, \alpha_k\}$.
	
	A \new{substring of} $\rho$ is a string of the form $\alpha_j^{\varepsilon_j} \ldots \alpha_i^{\varepsilon_i}$ with $1 \leqslant i \leqslant j \leqslant k$, or a lazy path $e_q$ for $q \in \mathsf{Supp}_0(\rho)$. For all strings $\rho$, we identify $\rho$ and $\rho^{-1}$. Note that this identification is on purpose of Theorem \ref{BR}. A string $\rho$ of $\mathcal{Q}$ is said to be \new{maximal} if there is no arrow $\alpha \in Q_1$ such that any of the following $\alpha \rho$, $\alpha^{-1} \rho$, $\rho \alpha$ or $\rho \alpha^{-1}$ is a string.
	
	We can understand a string as a (reduced) walk in $\mathcal{Q}$ such that we can follow a sequence of arrows (of the form $\alpha \in Q_1$) and inverse arrows (of the form $\alpha^{-1}$ with $\alpha \in Q_1$) avoiding consecutive steps that belong to the ideal $I$ and successive steps $\alpha \alpha^{-1}$ or $\alpha^{-1} \alpha$.
	\begin{ex} \label{stringex} Let us take the following (gentle) quiver $\mathcal{Q}=(Q,I)$.
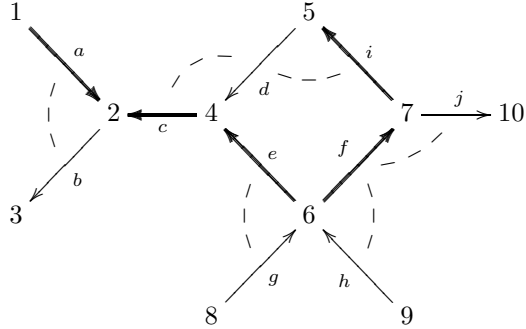
		\begin{center}
			$\displaystyle \xymatrix@R=1em @C=1em{
				1 \ar@<-0.1ex>[rrdd] \ar[rrdd] \ar@<+0.1ex>[rrdd]^a& & & & & &  5 \ar[lldd]^(.6)d & & & &  \\
				&  \ar@{--} @/_/[dd] & & & & \ar@{--} @/_/ [rr] & & & & & \\
				& & \ar[lldd]^b 2 & \ar@{--} @/^1pc/ [rru] &  4 \ar@<-0.1ex>[ll] \ar[ll] \ar@<+0.1ex>[ll]^c & &  & & 7 \ar@<-0.1ex>[lluu] \ar[lluu] \ar@<+0.1ex>[lluu]_i \ar[rr]^j& \ar@{--} @/^/[lld] & 10  \\
				& & & & & & & & & &  \\
				3 & & & & & &  6 \ar@<-0.1ex>[lluu] \ar[lluu] \ar@<+0.1ex>[lluu]_e \ar@<-0.1ex>[rruu] \ar[rruu] \ar@<+0.1ex>[rruu]^f   & & & &  \\
				& & & & & \ar@{--} @/^/ [uu] & & \ar@{--} @/_/ [uu] &  & &   \\
				& & & & 8 \ar[rruu]_g &  & & &  9 \ar[lluu]^h &  &   \\ }$
		\end{center}
		By following the bold arrows, we read the string $\rho = ife^{-1}c^{-1}a$ or $a^{-1}cef^{-1}i^{-1} = \rho^{-1}$ and we can check that it satisfies the previous conditions of the definition. We have $l(\rho) = 5$, $\mathsf{Supp}_0(\rho) = \{1,2,4,5,6,7\} $ and $\mathsf{Supp}_1(\rho) = \{a,c,e,f,i\}$. The substrings of $\rho$ are:
		\begin{itemize}
			\item[•] lazy paths: $e_q$ for $q \in \mathsf{Supp}_0(\rho)$;
			
			\item[•] arrows: $\mathsf{Supp}_1(\rho)$;
			
			\item[•] strings of length 2: $if, fe^{-1},\ ce,\ a^{-1}c$;
			
			\item[•] strings of length 3: $ife^{-1},\ c e f^{-1},\ a^{-1}ce$;
			
			\item[•] strings of length 4: $ife^{-1}c^{-1},\ f e^{-1} c^{-1} a$;
			
			\item[•] strings of length 5: $\rho$.
		\end{itemize}
		In addition, we can see that $\rho$ is a maximal string of $\mathcal{Q}$: we can neither concatenate any arrow of $\mathcal{Q}$ to $\rho$ at the vertex 1 (as $a$ is the only one arrow incident to $1$) nor at the vertex $5$ (as $di \in I$). 
		
		More visually, we can draw the same string as follows.
		\vspace*{-1cm}
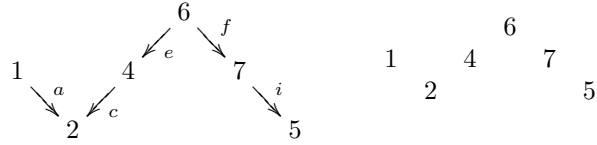
		\begin{center}
			$\displaystyle \xymatrix@R=1em @C=1em{
				& & & 6 \ar[ld]^e \ar[rd]^f & & \\
				1 \ar[rd]^a& & 4 \ar[ld]^c& & 7 \ar[rd]^i & \\
				& 2 & & & & 5 }$ \qquad \begin{tabular}{c c c c c c}
				& & & & &  \\
				& & & & &  \\
				& & & & &  \\
				& & & $6$ & & \\
				$1$ & & $4$ & & $7$ &  \\
				& $2$ & & & & $5$
			\end{tabular}
		\end{center}
		Note that by convention arrows always point downwards.
		
		We can choose not to represent the arrows and only keep the sequence of vertices that the string is passing through if there is no ambiguity about the arrows that we have to keep track of, like in this example.
	\end{ex}
	\begin{definition} \label{stringrepdef}
		A \new{string representation} $M(\rho)$ \new{of} $\mathcal{Q}$ is a representation in $\mathsf{rep}(\mathcal{Q})$ defined as follows:
		\begin{itemize}
			\item[•] Let $v_0 = s(\alpha_1^{\varepsilon_1})$ and for all $ i \in \{ 1, \ldots, k\}$, $v_i = t(\alpha_i^{\varepsilon_i})$;
			
			\item[•] For $q \in Q_0$, $M(\rho)_q$ is the vector space having as basis $\{ x_i \mid v_i = q \}$, where the $x_i$ are formal elements;
			
			\item[•] For $\beta \in Q_1$, $M(\rho)_\beta : M(\rho)_{s(\beta)} \longrightarrow M(\rho)_{t(\beta)}$ is the linear transformation such that: 
			\begin{center}
				$\displaystyle M(\rho)_\beta(x_i) = \left\{ \begin{matrix}
					x_{i-1} & \text{ if } \alpha_i = \beta \text{ and } \varepsilon_i = -1 \hfill\\
					x_{i+1} & \text{ if } \alpha_{i+1} = \beta \text{ and } \varepsilon_{i+1} = 1 \hfill\\
					0 & \text{ otherwise } \hfill
				\end{matrix} \right.$
			\end{center}
		\end{itemize}
	\end{definition}
	\begin{ex} \label{stringrepex} Consider the following gentle quiver.
		\vspace*{0.2cm}
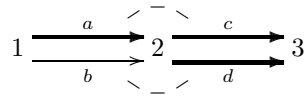
		\begin{center}
			$\displaystyle \xymatrix @R=1em @C=1em{1 \ar@<1ex> [rrr]  \ar@<1.2ex>[rrr] \ar@<1.1ex> [rrr]^a \ar@<-1ex>[rrr]_b & & \ar@{--} @/^1pc/ @<1ex> [rr] \ar@{--} @/_1pc/ @<-1ex>[rr] & 2 \ar@<1.2ex>[rrr] \ar@<1.1ex>[rrr]^c  \ar@<1ex> [rrr] \ar@<-1.2ex>[rrr] \ar@<-1.1ex>[rrr] \ar@<-1ex>[rrr]_d & & & 3}$
		\end{center}
		\vspace*{0.2cm}
		The word $\rho = c^{-1}da$ is a string of $\mathcal{Q}$ and the representation associated to it is:
		\begin{center}
			$\displaystyle \xymatrix @R=1em @C=1em{\mathbb{K} \ar@<1ex> [rrr]^{\begin{tiny}
						\left[ \begin{matrix}
							1 \\
							0
						\end{matrix} \right]
				\end{tiny}} \ar@<-1ex>[rrr]_{\begin{tiny}
						\left[ \begin{matrix}
							0 \\
							0
						\end{matrix} \right]
				\end{tiny}} & & \ar@{--} @/^1pc/ @<1ex> [rr] \ar@{--} @/_1pc/ @<-1ex>[rr] & \mathbb{K}^2 \ar@<1ex> [rrr]^{\begin{tiny} [0 \ 1] \end{tiny}} \ar@<-1ex>[rrr]_{\begin{tiny} [1 \ 0] \end{tiny}} & & & \mathbb{K}}$
		\end{center}
	\end{ex}
	\begin{remark} \label{stringrepconstruct}
		We can understand the construction in a more combinatorial way. To get the string representation $M(\rho)$ associated to a string $\rho$, we draw the string as Example \ref{stringex}, then we can replace each vertex of the string by $x_i$ and finally each arrow defines the direction of the linear map that we will get.
		\begin{center}
			$\displaystyle 
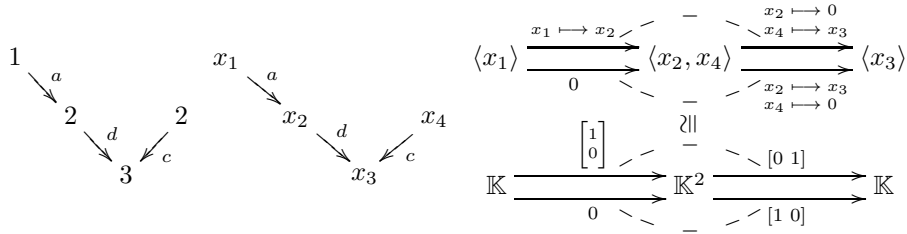
\xymatrix@R=1em @C=1em {
				1 \ar[rd]^a & & &  \\
				& 2 \ar[rd]^d & & 2 \ar[ld]^c \\
				& & 3 & }$  $\displaystyle \xymatrix@R=1em @C=1em {
				x_1 \ar[rd]^a & & &  \\
				& x_2 \ar[rd]^d & & x_4 \ar[ld]^c \\
				& & x_3 & }$  $\displaystyle \xymatrix @R=1em @C=1em{\langle x_1 \rangle \ar@<1ex> [rrr]^(.4){\begin{tiny}
						{\begin{tiny}
								\begin{matrix}
									x_1 \longmapsto x_2
								\end{matrix}
						\end{tiny} }
				\end{tiny}} \ar@<-1ex>[rrr]_(.4)0 & & \ar@{--} @/^1pc/ @<1ex> [rr] \ar@{--} @/_1pc/ @<-1ex>[rr] & \langle x_2, x_4 \rangle \ar@<1ex> [rrr]^(.6){\begin{tiny}
						\begin{matrix}
							x_2 \longmapsto 0 \hfill \\
							x_4 \longmapsto x_3 \hfill
						\end{matrix}
				\end{tiny} } \ar@<-1ex>[rrr]_(.6){\begin{tiny}
						\begin{matrix}
							x_2 \longmapsto x_3 \hfill \\
							x_4 \longmapsto 0 \hfill
						\end{matrix}
				\end{tiny} } & & & \langle x_3 \rangle \\
				&&& \ccong &&& \\
				\mathbb{K} \ar@<1ex> [rrr]^{\begin{tiny}
						\left[ \begin{matrix}
							1 \\
							0
						\end{matrix} \right]
				\end{tiny}} \ar@<-1ex>[rrr]_0 & & \ar@{--} @/^1pc/ @<1ex> [rr] \ar@{--} @/_1pc/ @<-1ex>[rr] & \mathbb{K}^2 \ar@<1ex> [rrr]^{{\tiny [0 \ 1]}} \ar@<-1ex>[rrr]_{{\tiny [1 \ 0]}} & & & \mathbb{K}}$
		\end{center}
		\vspace*{0.2cm}
		More efficiently, we can prefer to replace each vertex by a copy of the field $\mathbb{K}$ and each arrow by an identity map. 
	\end{remark}
	\begin{definition} \label{banddef}
		A \new{band} of $\mathcal{Q}$ is a non-lazy string $\omega$ such that:
		\begin{itemize}
			\item[•] $s(\omega) = t(\omega)$;
			
			\item[•] for all $i \geqslant 1$, $\omega^i = \underset{ i \text{ times }}{\underbrace{\omega \ldots \omega}}$ is a string of $\mathcal{Q}$;
			
			\item[•] $\omega$ is primitive: $\omega$ is not a power of any shorter substring.
		\end{itemize}
		Two bands of $\mathcal{Q}$, $\omega$ and $\omega'$, are said to be \new{cyclically equivalent}, denoted $\omega \sim_{\text{cyc}} \omega'$ if we can obtain $\omega'$ from $\omega$ by applying a cyclic permutation to the arrows which define $\omega$.
	\end{definition}
	\begin{definition} \label{bandrepdef}
		Let $\omega = \alpha_k^{\varepsilon_k} \ldots \alpha_1^{\varepsilon_1}$ be a band of $\mathcal{Q}$. Let $0 \neq \lambda \in \mathbb{K}$ and an integer $d > 0$. A \new{band representation} $M(\omega, \lambda, d)$ of $\mathcal{Q}$ is a representation defined as follows:
		\begin{itemize}
			\item[•] let $v_0 = s(\alpha_1^{\varepsilon_1})$ and for $i \in \{ 1, \ldots, k-1 \}$, $\displaystyle v_i = t(\alpha_i^{\varepsilon_i})$;
			
			\item[•] for all vertices $q \in Q_0$, $M(\omega, \lambda, d)_q$ is the vector space having as basis $\{ x_i^{(1)}, \ldots, x_i^{(d)} \mid v_i = q \}$ where the $x_i^{(j)}$ are formal elements;
			
			\item[•] for all arrows $\gamma \in Q_1$, $M(\omega, \lambda, d)_\gamma$ is the linear transformation defined as follows:
			\begin{center}
				$\displaystyle M(\omega, \lambda, d)_\gamma(x_i^{(j)}) = \left\{ \begin{matrix}
					x_{i-1}^{(j)} & \text{ if } \gamma = \alpha_i \text{ for } 1 \leqslant i < k,\ 1 \leqslant j \leqslant d \text{ and } \varepsilon_i = -1 \hfill \\
					\lambda^{-1} x_{k-1}^{(j)} +  x_{k-1}^{(j+1)} & \text{ if } \gamma = \alpha_k \text{ for } i= 0,\  1 \leqslant j < d \text{ and } \varepsilon_k = -1 \hfill \\
					\lambda^{-1}  x_{k-1}^{(d)} & \text{ if } \gamma = \alpha_k \text{ for } i= 0,\ j = d \text{ and } \varepsilon_k = -1 \hfill \\
					x_{i+1}^{(j)} & \text{ if } \gamma = \alpha_{i+1} \text{ for } 1 \leqslant i < k-1,\ 1 \leqslant j \leqslant d \text{ and } \varepsilon_i = 1 \hfill \\
					\lambda  x_0^{(j)} +  x_0^{(j+1)} & \text{ if } \gamma = \alpha_k \text{ for } i= k-1,\ 1 \leqslant j < d \text{ and } \varepsilon_k = 1 \hfill \\
					\lambda x_0^{(d)} & \text{ if } \gamma = \alpha_k \text{ for } i= k-1,\ j = d \text{ and } \varepsilon_k = 1 \hfill \\
					0 & \text{ otherwise } \hfill
				\end{matrix} \right.$
			\end{center}
		\end{itemize}
	\end{definition}
	\begin{ex} \label{bandex} Consider the following gentle quiver.
		\vspace*{0.2cm}
		\begin{center}
			$\displaystyle \xymatrix @R=1em @C=1em{1 \ar@<1ex> [rrr]^a \ar@<-1ex>[rrr]_b & & \ar@{--} @/^1pc/ @<1ex> [rr] \ar@{--} @/_1pc/ @<-1ex>[rr] & 2 \ar@<1ex> [rrr]^c \ar@<-1ex>[rrr]_d & & & 3}$
		\end{center}
		\vspace*{0.2cm}
		The string $\rho = b^{-1}c^{-1}da$ is a band of $\mathcal{Q}$. For $0 \neq \lambda \in \mathbb{K}$ and $d > 0$, the representation $M(\rho,\lambda,d)$ is isomorphic to the following representation:
		\begin{center}
			$\displaystyle \xymatrix @R=1em @C=1em{\mathbb{K}^d \ar@<1ex> [rrr]^{\begin{tiny}
						\left[ \begin{matrix}
							I_d \\
							0
						\end{matrix} \right]
				\end{tiny}} \ar@<-1ex>[rrr]_{\begin{tiny}
						\left[ \begin{matrix}
							0 \\
							J_d(\lambda^{-1})
						\end{matrix} \right]
				\end{tiny}} & & \ar@{--} @/^1pc/ @<1ex> [rr] \ar@{--} @/_1pc/ @<-1ex>[rr] & \mathbb{K}^{2d} \ar@<1ex> [rrr]^{{\tiny [0 \ I_d]}} \ar@<-1ex>[rrr]_{{\tiny [I_d \ 0]}} & & & \mathbb{K}^d}$
		\end{center}
		where $J_d(\lambda^{-1})$ is the following Jordan block. \qedhere
		\begin{center}
			$\displaystyle \left( \begin{matrix}
				\lambda^{-1} & & &\\
				1 & \ddots & & \\
				& \ddots & \ddots & \\
				& & 1 & \lambda^{-1}
			\end{matrix} \right)$ 
		\end{center}
	\end{ex}
	\begin{remark} \label{bandrepcombconstruct}
		Like string representations, we can understand the construction of a band representation by replacing each vertex of the band with a copy of $\mathbb{K}^d$ and each arrow with an identity map or a linear transformation associated to a Jordan block $J_d(\lambda)$ or $J_d(\lambda^{-1})$. However, we identify the first and the last vector spaces.
		\begin{center}
			$\displaystyle 
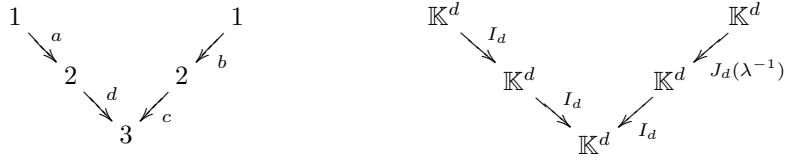
\xymatrix@R=1em @C=1em {
				1 \ar[rd]^a & & & & 1 \ar[ld]^b \\
				& 2 \ar[rd]^d & & 2 \ar[ld]^c & \\
				& & 3 &  &}$ \hspace*{2cm}  $\displaystyle \xymatrix@R=1em @C=1em {
				\mathbb{K}^d \ar[rd]^{I_d} & & &  &  \mathbb{K}^d \ar[ld]^{J_d(\lambda^{-1})} \\
				& \mathbb{K}^d \ar[rd]^{I_d} & & \mathbb{K}^d  \ar[ld]^{I_d}  &\\
				& & \mathbb{K}^d & & }$ 
		\end{center}
	\end{remark}
	Now we can state the following theorem that gives us the classification of all indecomposable representations of a gentle quiver.
	\begin{theorem}[\cite{BR87}]  \label{BR} 
		Let $\mathbb{K}$ be an algebraically closed field and $\mathcal{Q} = (Q,I)$ a gentle quiver. Any indecomposable representation of $\mathcal{Q}$ is a string or a band representation of $\mathcal{Q}$. Moreover :
		\begin{itemize}
			\item[•] A string representation is not isomorphic to a band representation ;
			
			\item[•] $M(\rho)$ is isomorphic to $M(\rho')$ if and only if $\rho' = \rho^{\pm 1}$ ;
			
			\item[•] $M(\omega,\lambda,d)$ is isomorphic to $M(\omega', \lambda', d')$ if and only if $d = d'$ and either $ \omega \sim_{\text{cyc}} \omega'$ and $\lambda=\lambda'$ or $\omega^{-1} \sim_{\text{cyc}} \omega'$ and $\lambda^{-1} = \lambda'$.	
		\end{itemize}
	\end{theorem}
	Remark that M.C.R. Butler and C.M. Ringel provide a classification of indecomposable modules over an arbitrary field. However, for simplicity, we restrict to the case that the field is algebraically closed which is the relevant case for us.
	
	Moreover, we also have a good understanding of the morphisms between two string representations of $\mathcal{Q}$. We just need the following notions.
	\begin{definition}\label{bottopdef}
		Let $\rho = \alpha_k^{\varepsilon_k} \ldots \alpha_1^{\varepsilon_1}$ a string of $\mathcal{Q}$. A substring $\sigma = \alpha_j^{\varepsilon_j} \ldots \alpha_i^{\varepsilon_i}$ of $\rho$ is said to be \new{on the top of} $\rho$ if $i =1$ or $\varepsilon_{i-1} = -1$, and $j=k$ or $\varepsilon_{j+1} = 1$. Analogically $\sigma$ is said to be \new{at the bottom of} $\rho$ if $i=1$ or $\varepsilon_{i-1} = 1$, and $j=k$ or $\varepsilon_{j+1} = -1$. 
	\end{definition}
	We can more easily understand these definitions thanks to the visual representation (mentioned in Example \ref{stringex}) of the string $\rho$.
	\begin{center}
		\scalebox{0.9}{\begin{tikzpicture}[yshift = 1cm, xshift = 5cm, ->,line width=0.2mm,>= angle 60,color=black, scale=0.6]
				\node (rho) at (-1,.5){$\rho =$};
				\node (1) at (0,0){$\bullet$};
				\node (2) at (2,0){$\bullet$};
				\node (3) at (3,1){$\bullet$};
				\node (4) at (6,1){$\bullet$};
				\node (5) at (7,0){$\bullet$};
				\node (6) at (9,0){$\bullet$};
				\draw[-,decorate, decoration={snake,amplitude=.4mm}] (1) -- (2);
				\draw (3) -- node[above left]{$\alpha_{i-1}$} (2);
				\draw[-,line width=0.7mm,decorate, decoration={snake,amplitude=.4mm}] (3) -- node[above]{$\pmb \sigma$} (4);
				\draw (4) -- node[above right]{$\alpha_{j+1}$} (5);
				\draw[-,decorate, decoration={snake,amplitude=.4mm}]  (5) -- (6);
				\node at (-.3,0.5){$\left(\vphantom{\begin{matrix}
							\\
							\\
							\\
					\end{matrix}}\right.$};
				\node at (2.8,0.5){$\left.\vphantom{\begin{matrix}
							\\
							\\
							\\
					\end{matrix}}\right)$};
				\node at (6.2,0.5){$\left(\vphantom{\begin{matrix}
							\\
							\\
							\\
					\end{matrix}}\right.$};
				\node at (9.3,0.5){$\left.\vphantom{\begin{matrix}
							\\
							\\
							\\
					\end{matrix}}\right)$};
				
				\begin{scope}[xshift = 12cm,yshift=1cm]
					\node (rho) at (-1,-.5){$\rho =$};
					\node (1) at (0,0){$\bullet$};
					\node (2) at (2,0){$\bullet$};
					\node (3) at (3,-1){$\bullet$};
					\node (4) at (6,-1){$\bullet$};
					\node (5) at (7,0){$\bullet$};
					\node (6) at (9,0){$\bullet$};
					\draw[-,decorate, decoration={snake,amplitude=.4mm}] (1) --  (2);
					\draw (2) -- node[below left]{$\alpha_{i-1}$} (3);
					\draw[-,line width=0.7mm,decorate, decoration={snake,amplitude=.4mm}] (3) -- node[below]{$\pmb \sigma$} (4);
					\draw (5) -- node[below right]{$\alpha_{j+1}$} (4);
					\draw[-,decorate, decoration={snake,amplitude=.4mm}]  (5) -- (6);
					\node at (-.3,-0.5){$\left(\vphantom{\begin{matrix}
								\\
								\\
								\\
						\end{matrix}}\right.$};
					\node at (2.8,-0.5){$\left.\vphantom{\begin{matrix}
								\\
								\\
								\\
						\end{matrix}}\right)$};
					\node at (6.2,-0.5){$\left(\vphantom{\begin{matrix}
								\\
								\\
								\\
						\end{matrix}}\right.$};
					\node at (9.3,-0.5){$\left.\vphantom{\begin{matrix}
								\\
								\\
								\\
						\end{matrix}}\right)$};
				\end{scope}
		\end{tikzpicture}}
	\end{center}
	On the left, we have a representation of a substring $\sigma$ on the top of $\rho$, and on the right a representation of a substring $\sigma$ at the bottom of $\rho$. Parts in parentheses are optional.
	\begin{prop}[\cite{CB89}] \label{morph}
		Let $\rho$ and $\rho'$ be two strings of $\mathcal{Q}$. Then \[\Hom(M(\rho),M(\rho')) \cong \mathbb{K}^{\#[\rho, \rho']}\] where $[\rho, \rho']$ is the set of couples $(\sigma, \sigma')$ such that the substring $\sigma$ is on the top of $\rho$, the substring $\sigma'$ is at the bottom of $\rho'$ and $\sigma' = \sigma^{\pm 1}$.
	\end{prop}
	\begin{ex} \label{morphex} Consider the following gentle quiver.	
		\begin{center}
			$\displaystyle \xymatrix@R=0.5em @C=0.5em{ \mathcal{Q} & = & 1 \ar[rrr]^a & & \ar@{--} @/^0.8pc/[rr] & 2 \ar[rrr]^b & & & 3 & & & 4 \ar[lll]_c}$
		\end{center}
		Let $X$ and $Y$ be the following representations.
		\begin{center}
			$\displaystyle X = \xymatrix@R=0.5em @C=0.5em{   \mathbb{K}^2 \ar[rrr]^{{\tiny[ 1 \ 0]}} & & \ar@{--} @/^0.8pc/[rr] & \mathbb{K} \ar[rrr]^{\begin{tiny}
						\left[ \begin{matrix}
							0 \\
							0
						\end{matrix} \right]
				\end{tiny}} & & & \mathbb{K}^2 & & & \mathbb{K}^2 \ar[lll]_{\begin{tiny}
						\left[ \begin{matrix}
							1 & 0\\
							0 & 1 
						\end{matrix} \right]
			\end{tiny}} } $
			
			$\displaystyle   Y = \xymatrix@R=0.5em @C=0.5em{   \mathbb{K} \ar[rrr]^{\begin{tiny}
						\left[ \begin{matrix}
							1 \\
							0
						\end{matrix} \right]
				\end{tiny}} & & \ar@{--} @/^0.8pc/[rr] & \mathbb{K}^2 \ar[rrr]^{{\tiny[ 0 \ 1]}} & & & \mathbb{K} & & & 0 \ar[lll] } $
		\end{center}
		We can decompose them as a sum of indecomposable representations using string representations of $\mathcal{Q}$.
		\begin{center}
			$\displaystyle X \cong M(e_1) \oplus M(a) \oplus M(c)^2 $ \qquad  $\displaystyle   Y \cong M(a) \oplus M(b) $
		\end{center}
		We describe $\mathsf{Hom}(X,Y)$ by:
		\begin{itemize}
			\item[•] $\mathsf{Hom}(M(e_1), M(a)) = 0$ because $e_1$ is the only substring of $e_1$ (as $e_1$ is lazy) and $e_1$ is not at the bottom of $a$;
			
			\item[•] $\mathsf{Hom}(M(e_1), M(b)) = 0$ because $e_1$ is not a substring of $b$;
			
			\item[•] $\mathsf{Hom}(M(a),M(a)) \cong \mathbb{K}a$ as $a$ is the only substring of $a$ which is both on the top and at the bottom of $a$;
			
			\item[•] $\mathsf{Hom}(M(a), M(b)) = 0$ because the only substring of both $a$ and $b$, $e_2$, is neither on the top of $a$ nor at the bottom of $b$;
			
			\item[•] $\mathsf{Hom}(M(c), M(a)) = 0$ because there is no substring of both $c$ and $a$;
			
			\item[•] $\mathsf{Hom}(M(c),M(b)) = 0$ because the only substring of both $b$ and $c$, $e_3$, is not on the top of $c$.
		\end{itemize}
		Thus $\mathsf{Hom}(X, Y) \cong \mathbb{K}$ and the only type of morphisms in this set is of the form
		\begin{center}
			$\displaystyle \xymatrix@R=0.5em @C=0.5em{  & & \mathbb{K}^2 \ar@{.>}[ddd]^{{\tiny[ k \ 0]}} \ar[rrr]^{\begin{tiny}
						[1 \ 0]
				\end{tiny}}   & & \ar@{--} @/^0.8pc/[rr] & \mathbb{K} \ar@{.>}[ddd]^{{\tiny \left[ \begin{matrix}
							k \\
							0
						\end{matrix} \right] }} \ar[rrr]^{\begin{tiny}
						\left[ \begin{matrix}
							0 \\
							0  
						\end{matrix} \right]
				\end{tiny}}  & & & \mathbb{K}^2 \ar@{.>}[ddd]^{{\tiny 0 }} & & & \mathbb{K}^2 \ar[lll]_{\begin{tiny}
						\left[ \begin{matrix}
							1 & 0\\
							0 & 1 
						\end{matrix} \right]
				\end{tiny}} \ar@{.>}[ddd]^{{\tiny 0 }} & =X  \\
				\phi = & & & &  &  & & & & & & &  \\
				& &   & &  &  & & & & & & & \\
				& &  \mathbb{K} \ar[rrr]_{\begin{tiny}
						\left[ \begin{matrix}
							1 \\
							0
						\end{matrix} \right]
				\end{tiny}} & & \ar@{--} @/_0.8pc/[rr] & \mathbb{K}^2 \ar[rrr]_{{\tiny[ 0 \ 1]}} & & & \mathbb{K} & & & 0 \ar[lll] & = Y}$
		\end{center}
		with $k \in \mathbb{K}$.
	\end{ex}
	These kinds of morphisms between representations are called \new{graph maps}. They provide a basis for the morphisms between string representations. However, morphisms exist between band representations that are not linear combinations of graph maps. We do not need to go into more precise statements to describe all the morphisms between band representations in this article, but if the reader is interested, we can refer to \cite{K88}. 
	
	\subsection{Jordan recoverability and canonical Jordan recoverability} Let $\mathcal{Q}=(Q,I)$ be a gentle quiver and $X \in \mathsf{rep}(\mathcal{Q})$.  A \new{nilpotent endomorphism} $N : X \longmapsto X$ is an endomorphism such that $N^k = 0$ for some integer $k > 0$. We can remark that $N$ is nilpotent if and only if $N_q$ is a nilpotent linear morphism for all $q \in Q_0$. Let $\mathsf{NEnd}(X)$ denote the set of nilpotent endomorphisms of $X$.
	
	Let $n > 0$ be an integer. An \new{integer partition of} $n$ is a finite weakly decreasing sequence $\lambda = (\lambda_1, \lambda_2, \ldots, \lambda_k)$ of strictly positive integers such that $\lambda_1 + \cdots + \lambda_k = n$. The \new{size} of an integer partition $\lambda$ is $|\lambda| = \lambda_1 + \cdots + \lambda_k$. We write $\lambda \vdash n$ when $\lambda$ is a partition of $n$. 
	
	Let $\underline{\dim}(X) = (d_q)_{q \in Q_0}$. For any $N \in \mathsf{NEnd}(X)$, we can consider the Jordan form of $N_q$ at each vertex $q$. It induces a sequence of partitions $\lambda^q \vdash d_q$. Hence $(\lambda^q)_{q\in Q_0}$ can be considered as the \new{Jordan form data of} $N$. We will denote it $\mathsf{JF}(N)$. 
	
	The \new{dominance order} over partitions of an integer $n$ is defined by: for any $\lambda$ and $\mu$ partitions of $n$, $\lambda \unlhd \mu$ if $\lambda_1 + \ldots + \lambda_k \leqslant \mu_1 + \ldots + \mu_k$ for each $k \geqslant 1$ where we add zero parts to $\lambda$ and $\mu$ if necessary. We extend this order to any $p$-tuple of partitions as follows. For $\underline{\lambda} = (\lambda^1, \ldots, \lambda^p)$ and $\underline{\mu} = (\mu^1, \ldots, \mu^p)$ such that $\lambda^i$ and $\mu^i$ are partitions of $m_i$, we say that $\underline{\lambda} \unlhd \underline{\mu}$ for all $i \in \{1,\ldots,p\}$ $\lambda^i \unlhd \mu^i$.
	
	A. Garver, R. Patrias, and H. Thomas proved the following statement.
	\begin{theorem}[\cite{GPT19}] \label{GPT3}
		Let $A$ be a finite-dimensional algebra over an algebraically closed field $\mathbb{K}$ and $X$ be a finite-dimensional (left) module over $A$. Then the set of nilpotent endomorphisms of $X$, $\mathsf{NEnd}(X)$, is an irreducible algebraic variety. Furthermore, there is a maximum value of $\mathsf{JF}$ on $\mathsf{NEnd}(X)$ and it is attained on a dense open set of $\mathsf{NEnd}(X)$
	\end{theorem}
	As we work with tools that are satisfying exactly the hypotheses of the previous theorem, we can define $\mathsf{GenJF}(X)$ the \new{generic Jordan form data of} $X$ as this maximal value of $\mathsf{JF}$ on $\mathsf{NEnd}(X)$.
	
	For more details about this result, see \cite[Section 2.2]{GPT19}.
	\begin{definition}\label{JRdef}
		A subcategory $\mathscr{C}$ of $\mathsf{rep}(\mathcal{Q})$ is called \new{Jordan recoverable} if, for any tuple of partitions $\underline{\lambda}$, there is at most a unique (up to isomorphism) $X \in \mathscr{C}$ such that $\mathsf{GenJF}(X) = \underline{\lambda}$.
	\end{definition}
	We can have another look at the Example \ref{A2ex1} to see in a simple case how it works. As we understand from this example, $\mathsf{rep}(\mathcal{Q})$ is not Jordan recoverable in general. A question that naturally arises is to determine which subcategories $\mathscr{C}$ of $\mathsf{rep}(\mathcal{Q})$ satisfy this property.
	
	In the first instance, any subcategory with the property that the dimension vectors of its indecomposable representations are linearly independent is Jordan recoverable. For such subcategories, we can recover $X$ from its dimension vector $\underline{\dim}(X)$.
	
	As we are interested in this kind of subcategory, let $\mathscr{C}_{\mathcal{Q},m}$ be the subcategory additively generated by all the indecomposable representations $X$ of $\mathcal{Q}$ which are supported at the vertex $m$ (that is to say such that $\dim X_m \geqslant 1$).
	\begin{ex} \label{JRex1}
		Let us take the following gentle quiver $\mathcal{Q}$:
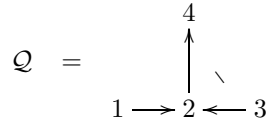
		\begin{center}
			$\displaystyle  \xymatrix @R=0.5em @C=0.5em{ 
				& & & & 4 & &   \\
				\mathcal{Q} & =& & & & & \\
				&  & 1 \ar[rr] & & 2 \ar[uu]  & \ar@{--} @/_/[lu]  & 3 \ar[ll] }$
		\end{center}
		The subcategory $\mathscr{C}_{\mathcal{Q},1}$ is generated by the following indecomposable representations.
		\begin{center}
			$\displaystyle  
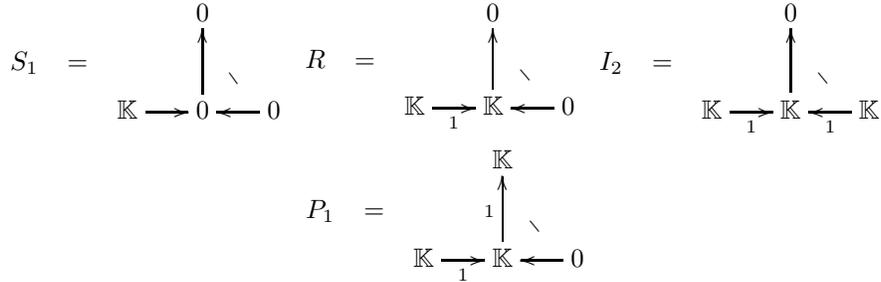
\xymatrix @R=0.5em @C=0.5em{ 
				& & & & 0 & &   \\
				S_1 & = & & & & & \\
				&  & \mathbb{K} \ar[rr] & & 0 \ar[uu]  & \ar@{--} @/_/[lu]  & 0 \ar[ll] }$ $\displaystyle  \xymatrix @R=0.5em @C=0.5em{ 
				& & & & 0 & &   \\
				R & = & & & & & \\
				&  & \mathbb{K} \ar[rr]_1 & & \mathbb{K} \ar[uu]  & \ar@{--} @/_/[lu]  & 0 \ar[ll] }$ $\displaystyle  \xymatrix @R=0.5em @C=0.5em{ 
				& & & & 0 & &   \\
				I_2 & = & & & & & \\
				&  & \mathbb{K} \ar[rr]_1 & & \mathbb{K} \ar[uu]  & \ar@{--} @/_/[lu]  & \mathbb{K} \ar[ll]^1 }$
			
			$\displaystyle  \xymatrix @R=0.5em @C=0.5em{ 
				& & & & \mathbb{K}& &   \\
				P_1 & = & & & & & \\
				&  & \mathbb{K} \ar[rr]_1 & & \mathbb{K} \ar[uu]^1  & \ar@{--} @/_/[lu]  & 0 \ar[ll] }$
		\end{center}
		Let $X = S_1^a \oplus I_2^b \oplus R^c \oplus P_1^d$. Let $N = (N_1,N_2,N_3,N_4) \in \mathsf{NEnd}(X)$.
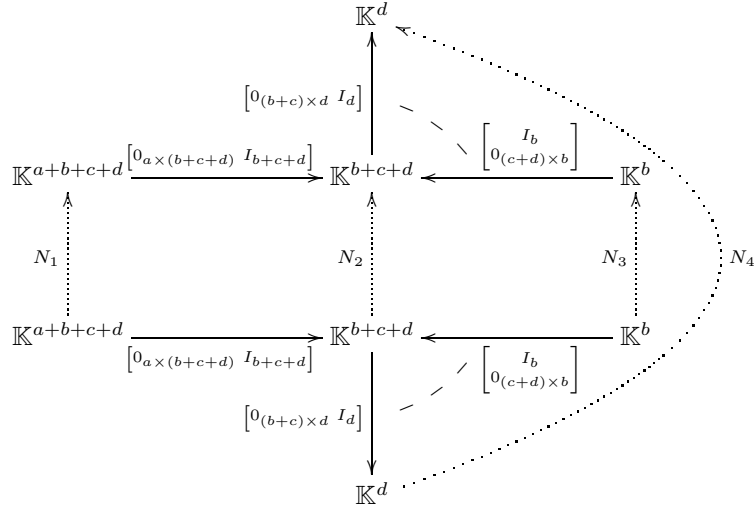
		\begin{center}
			$\displaystyle \xymatrix@R=2em @C=2em{
				& & & \mathbb{K}^d & & & \\
				& & & &  \\
				\mathbb{K}^{a+b+c+d} \ar[rrr]^{\begin{tiny}
						\left[ \begin{matrix}
							0_{a \times (b+c+d)} \ I_{b+c+d}
						\end{matrix} \right]
				\end{tiny}} & & & \mathbb{K}^{b+c+d} \ar[uu]^{\begin{tiny}
						\left[ \begin{matrix}
							0_{(b+c) \times d} \ I_{d}
						\end{matrix} \right]
				\end{tiny}} & \ar@{--} @/_/ [lu] & & \mathbb{K}^{b} \ar[lll]_(.4){\begin{tiny}
						\left[ \begin{matrix}
							I_{b} \\
							0_{(c+d)\times b}
						\end{matrix} \right]
				\end{tiny}} \\
				& & & & & &  \\
				\mathbb{K}^{a+b+c+d} \ar@{.>}[uu]^{N_1} \ar[rrr]_{\begin{tiny}
						\left[ \begin{matrix}
							0_{a \times (b+c+d)} \ I_{b+c+d}
						\end{matrix} \right]
				\end{tiny}} & & & \mathbb{K}^{b+c+d} \ar@{.>}[uu]^{N_2} \ar[dd]_{\begin{tiny}
						\left[ \begin{matrix}
							0_{(b+c) \times d} \ I_{d}
						\end{matrix} \right]
				\end{tiny}} & \ar@{--} @/^/ [ld] & & \mathbb{K}^{b} \ar@{.>}[uu]^{N_3} \ar[lll]^(.4){\begin{tiny}
						\left[ \begin{matrix}
							I_b \\
							0_{(c+d) \times b} 
						\end{matrix} \right]
				\end{tiny}} \\
				& & & & \\
				& & & \mathbb{K}^d \ar@{.>}@/_11pc/[uuuuuu]_{N_4}& & & } $
		\end{center}
		Let $(x_1, \ldots, x_{a+b+c+d})$ be a basis of $X_1 = \mathbb{K}^{a+b+c+d}$. By the linear tranformations defining $X$, we can consider $(x_{a+1}, \ldots x_{a+b+c+d})$ as a basis of $ X_2= \mathbb{K}^{b+c+d}$, $(x_{a+1}, \ldots, x_{a+b})$ as a basis of $X_3 =  \mathbb{K}^b$ and $(x_{a+b+c+1}, \ldots, x_{a+b+c+d})$ as a basis of $X_4 = \mathbb{K}^d$. Using the commutative square relations that $N_1$, $N_2$, $N_3$ and $N_4$ have to satisfy, we get:
		\begin{itemize}
			\item[•] $N_2(x_i) = \left[ \begin{matrix}
				0_{a \times (b+c+d)} \ I_{b+c+d}
			\end{matrix} \right] N_1(x_i) $ for all $i \in \{a+1, \ldots, a+b+c+d\}$;
			
			\item[•] $N_3(x_i) = N_2(x_i)$ for all $i \in \{a+1, \ldots, a+b\}$;
			
			\item[•] $N_4(x_i) =  \left[ \begin{matrix}
				0_{(b+c) \times d} \ I_{d}
			\end{matrix} \right] N_2(x_i)$ for all $i \in \{a+b+c+1, \ldots, a+b+c+d\}$.
		\end{itemize}
		We deduce that once we define $N_1$ we define the nilpotent endormorphism $N$. Let us take $N_1$ defined by:
		\begin{center}
			$N_1 (x_i) = \left\{ \begin{matrix}
				x_{i-1} & \text{for } 2 \leqslant i \leqslant a+b+c+d \hfill \\
				0 & \text{for } i=1 \hfill
			\end{matrix} \right.$
		\end{center}
		We can see that $N_1$ allows us to define $N \in \mathsf{NEnd}(X)$ and:
		\begin{center}
			$\mathsf{JF}(N)= ((a+b+c+d),(b+c+d),(c),(d))$
		\end{center}
		Moreover $N$ is a nilpotent endomorphism that attained the maximal element (with respect to $\unlhd$) of $\mathsf{NEnd}(X)$ that we can have. Hence, by \cite[Theorem 2.3]{GPT19}, we can conclude that
		\begin{center}
			$\mathsf{GenJF}(X) = ((a+b+c+d), (b+c+d),(c),(d))$
		\end{center}
		It is easy to see that we can recover $X$ from $\mathsf{GenJF}(X)$ just by the fact that from $\underline{\dim}(X)$ we can already recover the representation in $\mathscr{C}_{\mathcal{Q},1}$. Thus $\mathscr{C}_{\mathcal{Q},1}$ is Jordan recoverable.
	\end{ex}
	However we can ask ourselves in general how we can recover $X$ from its generic Jordan form data. 
	
	Suppose that $\underline{\lambda} = \mathsf{GenJF}(X) = (\lambda^q)_{q \in Q_0}$. From that we get a tuple $d_q = | \lambda^q |$, the dimension of the vector space $X_q$ at each vertex $q$. Let $Y_q$ be a vector space of dimension $d_q$ and $N_q$ a nilpotent endomorphism of $Y_q$ with Jordan blocks size given by $\lambda^q$.
	
	Let $\mathsf{rep}(\mathcal{Q},\underline{\lambda})$ be the representations $Y = ((Y_q)_{q \in Q_0}, (Y_\alpha)_{\alpha \in Q_1})$ such that the $Y_q$ are defined as above and $N = (N_q)_{q \in Q_0}$ define a nilpotent endomorphism of $Y$. \cite{GPT19} show that, under some conditions, $\mathsf{rep}(\mathcal{Q},\underline{\lambda})$ is an irreducible variety and then they deduce that there exists a dense open set $O \subset \mathsf{rep}(\mathcal{Q}, \underline{\lambda})$ such that for any representation in $O$, the dimension vectors of the indecomposable summands are well-defined. However, in our general case, this statement does not hold.
	\begin{ex} \label{Kro}
		Let $\mathcal{Q}$ be the Kronecker quiver.
		\vspace*{0.2cm}
		\begin{center}
			$\displaystyle \xymatrix{1 \ar@<1ex>[rr] \ar@<-1ex>[rr] & & 2 }$
		\end{center}
		\vspace*{0.2cm}
		Let $Y \in \mathsf{rep}(\mathcal{Q}, \underline{\lambda} = ((1),(1)))$. So $Y$ is of the following form
		\begin{center}
			$\displaystyle \xymatrix{\mathbb{K} \ar@<1ex>[rr]^{Y_a} \ar@<-1ex>[rr]_{Y_b} & & \mathbb{K}}$
		\end{center}
		with $Y_a : x \longmapsto k_1 x$ and $Y_b : x \longmapsto k_2 x $ and $k_1, k_2 \in \mathbb{K}$.
		
		Hence $Y$ is isomorphic to one of the following isomorphism classes:
		\begin{center}
			$\displaystyle \xymatrix{\mathbb{K} \ar@<1ex>[rr]^0 \ar@<-1ex>[rr]_0 & & \mathbb{K} }$ \qquad  $\displaystyle \xymatrix{\mathbb{K} \ar@<1ex>[rr]^0 \ar@<-1ex>[rr]_1 & & \mathbb{K} }$ \qquad $\displaystyle \xymatrix{\mathbb{K} \ar@<1ex>[rr]^1 \ar@<-1ex>[rr]_\alpha & & \mathbb{K} } \ (\alpha \in \mathbb{K})$
		\end{center}
		with each distinct value of $\alpha$ giving rise to a distinct isomorphism class of representations from Theorem \ref{BR}.
		
		More concretly:
		\begin{itemize}
			\item[•] one isomorphism class is given by taking $k_1 = 0$ and $k_2 = 0$;
			
			\item[•] another one by taking $k_1= 0$ and $k_2 \neq 0$: it corresponds to take $k_1 = 0$ and $k_2 = 1$ up to isomorphism;
			
			\item[•] for each specific value of $\alpha$, we have an isomorphism class by taking $k_2 = \alpha k_1$ (if $\alpha = 0$, we put $k_2 = 0$) and $k_1 \neq 0$: it corresponds to take $k_1 = 1$ and $k_2 = \alpha$ up to isomorphism.
		\end{itemize}
		None of those classes gives a dense set in $\mathsf{rep}(\mathcal{Q},\underline{\lambda})$. Thus there is no generic choice of a representation which have a generic Jordan form $((1),(1))$.
	\end{ex}
	All is not lost. Depending on the tuple of partitions $\underline{\lambda}$ and the gentle quiver we study, we can still have a generic choice of a representation $Y \in \mathsf{rep}(\mathcal{Q})$. One of the purposes of this article is to prove, by studying a specific kind of subcategories, and under some conditions on the quiver, we still have this property. 
	\begin{definition} \label{CJRdef1}
		A subcategory $\mathscr{C}$ of $\mathsf{rep}(\mathcal{Q})$ is said to be \new{canonically Jordan recoverable} if, for any  $X \in \mathscr{C}$, there exists a dense open set $O \subset \mathsf{rep}(\mathcal{Q},\mathsf{GenJF}(X))$ such that all the representations in $O$ are isomorphic to $X$.
	\end{definition}
	Note that to have canonical Jordan recoverability we need Jordan recoverability. However, canonical Jordan recoverability is a more restrictive condition (see Example \ref{JRnotCJRex})
	\begin{ex} \label{CJRex}
		Let $\mathcal{Q} = (Q,I)$ be the following gentle quiver.
		\begin{center}
			$\displaystyle \xymatrix @R=1em @C=1em{ 
				& & & & \\
				1 \ar[rr] & \ar@{--} @/^1pc/ [rr] & 2\ar[rr] & & 3}$
		\end{center}
		We have $\mathscr{C}_{\mathcal{Q},2} = \mathsf{add}(I_2, S_2, P_2)$ where :
		\vspace*{0.3cm}
		\begin{center}
			$\displaystyle I_2 = \xymatrix @R=1em @C=1em{ 
				\mathbb{K} \ar[rr]_1 & \ar@{--} @/^1pc/ [rr] & \mathbb{K}\ar[rr] & & 0}$, $\displaystyle S_2 = \xymatrix @R=1em @C=1em{ 
				0 \ar[rr] & \ar@{--} @/^1pc/ [rr] & \mathbb{K}\ar[rr] & & 0}$ ,  $\displaystyle P_2 = \xymatrix @R=1em @C=1em{ 
				0 \ar[rr] & \ar@{--} @/^1pc/ [rr] & \mathbb{K} \ar[rr]_1 & & \mathbb{K}}$
		\end{center}
		
		Let $X = I_2^a \oplus S_2^b \oplus P_2^c \in \mathscr{C}_{\mathcal{Q},2}$. 
		
		Let $N = (N_1, N_2,N_3)$ a generic nilpotent endomorphism. 
		\begin{center}
			$\displaystyle \xymatrix @R=1em @C=1em{ 
				& & & & \\
				\mathbb{K}^a \ar[rr]_(.4){X_\alpha} \ar@{.>}[dd]^{N_1} & \ar@{--} @/^1pc/ [rr] & \mathbb{K}^{a+b+c} \ar[rr]_(.6){X_\beta} \ar@{.>}[dd]^{N_2} & & \mathbb{K}^c \ar@{.>}[dd]^{N_3} \\
				& & & & \\
				\mathbb{K}^a \ar[rr]^(.4){X_\alpha} & \ar@{--} @/_1pc/ [rr] & \mathbb{K}^{a+b+c} \ar[rr]^(.6){X_\beta} & & \mathbb{K}^c \\
				& & & & }$
		\end{center}
		By direct calculation, inspired by Example \ref{1stGentlex}, we get 
		\begin{center}
			$\mathsf{GenJF}(X) = ((a),(a+b+c),(c))$.
		\end{center}
		Let $Y \in \mathscr{C}_{\mathcal{Q},2}$ such that $\mathsf{GenJF}(Y) = \mathsf{GenJF}(X)$. Clearly we can recover the vector spaces that we have at each vertex. By linear independence of dimension vectors of indecomposable representations that generate $\mathscr{C}_{\mathcal{Q},2}$, and by the following 
		\begin{center}
			$\underline{\dim}(Y) = \left( \begin{matrix}
				a \\
				a+b+c \\
				c
			\end{matrix} \right) = a \times \underline{\dim}(I_2) + b \times  \underline{\dim}(S_2) + c \times \underline{\dim}(P_2)$
		\end{center}
		we can conclude that $Y = I_2^a \oplus S_2^b \oplus P_2^c$. That is why $\mathscr{C}_{\mathcal{Q},2}$ is Jordan recoverable. 
		
		We will prove now that $\mathscr{C}_{\mathcal{Q},2}$ is canonically Jordan recoverable. Let us take a generic nilpotent endomorphism of the form $\lambda = ((a),(a+b+c),(c))$, $N = (N_1,N_2,N_3)$. By the dimension vector we can extract from $\lambda$, we can assure that $Y$ is of the form
		\begin{center}
			$\displaystyle Y \cong \xymatrix @R=1em @C=1em{ 
				\mathbb{K}^a \ar[rr]_(.4){Y_\alpha}& \ar@{--} @/^1pc/ [rr] & \mathbb{K}^{a+b+c} \ar[rr]_(.6){Y_\beta}  & & \mathbb{K}^c  }$
		\end{center}
		with $Y_\alpha$, $Y_\beta$ still to find.
		
		Let $(x_1, \ldots, x_a)$, $(y_1, \ldots, y_{a+b+c})$ and $(z_1, \ldots, z_c)$ be adapted bases for the respective nilpotent endormorphisms $N_1, N_2$ and $N_3$ associated to each vector space $Y_1, Y_2$ and $Y_3$.
		
		By commutativity relations, we have:
		\begin{itemize}
			\item[•] $Y_\alpha$ is completely defined by the image of  $x_1$;
			
			\item[•] $Y_\alpha(x_1) \in \mathsf{Ker}(N^a)$, which means
			
			\hspace*{-0.5cm} $Y_\alpha(x_1) = \alpha_1 y_{1+b+c} + \ldots + \alpha_a y_{a+b+c}$ for $\alpha_i \in \mathbb{K}$ for $1 \leqslant i \leqslant a$;
			
			\item[•] As $Y_\alpha$ has to be chosen generically, we have $\alpha_1 \neq 0$ and $Y_\alpha$ is injective;
			
			\item[•] $Y_\beta$ is completely defined by the image of $y_1$ ;
			
			\item[•] $Y_\beta(y_1) = \beta_1 z_1 + \ldots + \beta_{c} z_c$ with $\beta_j \in \mathbb{K}$ for $1 \leqslant j \leqslant c$;
			
			\item[•] As $Y_\beta$ has to be chosen generically, we have $\beta_1 \neq 0$ and $Y_\beta$ is surjective;
			
			\item[•] The relation $Y_\beta \circ Y_\alpha = 0$ is satisfied for free.
		\end{itemize}
		The following diagram sums up the previous points.
		\begin{center}
			$\displaystyle 
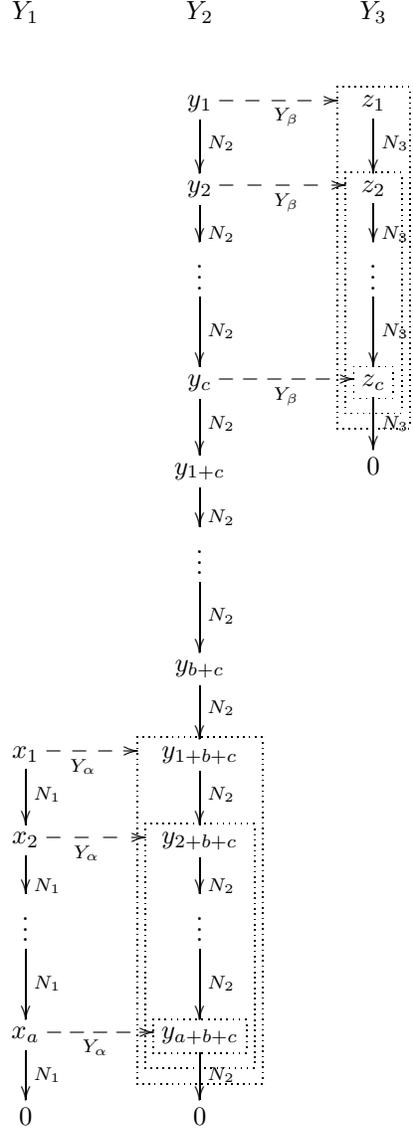
\xymatrix@C=4em @R=2em{
				Y_1 & Y_2 & Y_3 \\
				& y_1 \ar[d]^{N_2} & z_1 \ar[d]^{N_3}  \\
				& y_2 \ar[d]^{N_2}& z_2 \ar[d]^{N_3}  \\
				& \vdots \ar[d]^{N_2} & \vdots \ar[d]^{N_3}  \\
				& y_c  \ar[d]^{N_2} &  z_c \ar[d]^{N_3}  \\
				& y_{1+c} \ar[d]^{N_2} & 0 \\
				& \vdots \ar[d]^{N_2} & \\
				& y_{b+c} \ar[d]^{N_2} & \\
				x_1 \ar[d]^{N_1} & y_{1+b+c} \ar[d]^{N_2} &  \\
				x_2 \ar[d]^{N_1} & y_{2+b+c} \ar[d]^{N_2} & \\
				\vdots \ar[d]^{N_1} & \vdots \ar[d]^{N_2} &  \\
				x_a \ar[d]^{N_1} & y_{a+b+c} \ar[d]^{N_2} & \\
				0 & 0 & 
				\save "2,3"."5,3"*++[F.]\frm{}
				\ar@{<--}^{Y_\beta} "2,2"
				\restore
				\save "3,3"."5,3"*+[F.]\frm{}
				\ar@{<--}^{Y_\beta} "3,2"
				\restore  
				\save "5,3"."5,3"*[F.]\frm{}
				\ar@{<--}^{Y_\beta} "5,2"
				\restore 
				\save "9,2"."12,2"*++[F.]\frm{}
				\ar@{<--}^(.67){Y_\alpha} "9,1"
				\restore
				\save "10,2"."12,2"*+[F.]\frm{}
				\ar@{<--}^(.65){Y_\alpha} "10,1"
				\restore  
				\save "12,2"."12,2"*[F.]\frm{}
				\ar@{<--}^(.6){Y_\alpha} "12,1"
				\restore }$
		\end{center}
		Hence we conclude that $Y \cong  I_2^a \oplus S_2^b \oplus P_2^c = X $. 
		
		Thus $\mathscr{C}_{\mathcal{Q},2}$ is canonically Jordan recoverable.
	\end{ex}
	
	\begin{remark} \label{surprise} It is surprising to see that from any generic Jordan form we can get in the last example, we can recover the representation. Moreover, the fact that the relation ($Y_\beta Y_\alpha = 0$) is obtained automatically from generic choices of morphisms is noteworthy. The reader is invited to have another look at the remarks done below Example \ref{1stGentlex}.
	\end{remark}
	
	\begin{remark} \label{CJRneed1} We can also take notice that if we add an arrow $\gamma$ between $1$ and $3$ whatever its orientation and relations that we could add, then we lose the fact that $\mathscr{C}_{\mathcal{Q},2}$ is canonically Jordan recoverable. More explicitly, let us take the following quiver.
		\begin{center}
			$\displaystyle 
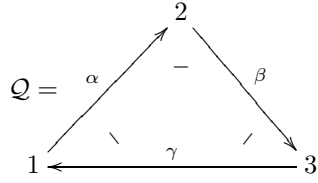
\xymatrix@R=0.5em @C=0.5em{
				& & & &  2 \ar[rrrrdddd]^\beta & & & & \\
				& & & \ar@{--}@/_/[rr] & & & & &  \\
				\mathcal{Q} =& & & & & & & &  \\
				& \ar@/^/@{--}[rd]& & & & & & &  \\
				1 \ar[rrrruuuu]^\alpha & &  & & &  & \ar@/^/@{--}[ur] & & 3 \ar[llllllll]_{\gamma}}$
		\end{center}
		The subcategory $\mathscr{C}_{\mathcal{Q},2}$ is Jordan recoverable as the dimension vectors of the indecomposable objects ($I_2, P_2$ and $S_2$) are linearly independant. Note that if we take $X = I_2^a \oplus S_2^b \oplus P_2^c$ then we have $\mathsf{GenJF}(X) = ((a),(a+b+c),(c))$.
		
		However, by taking $X = I_2 \oplus P_2$, there is no generic choice in $\mathsf{rep}(\mathcal{Q},((1),(2),(1))$. After some calculus, we note that we have  two isomorphism classes which come naturally from generic restrictions:
		\begin{center}
			$\displaystyle 
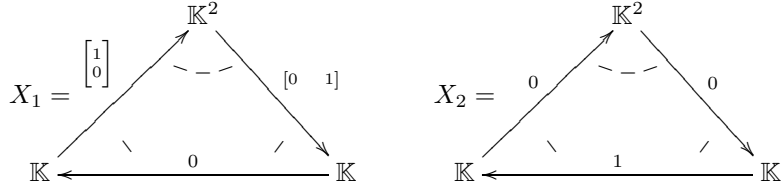
\xymatrix@R=0.5em @C=0.5em{
				& & & &  \mathbb{K}^2 \ar[rrrrdddd]^{\tiny \left[\begin{matrix}
						0 & 1
					\end{matrix} \right]} & & & & \\
				& & & \ar@{--}@/_/[rr] & & & & &  \\
				X_1 =& & & & & & & &  \\
				& & & & & & & &  \\
				\mathbb{K} \ar[rrrruuuu]^{\tiny \left[\begin{matrix}
						1 \\
						0
					\end{matrix}\right]} & & \ar@/_/@{--}[lu] & & & & \ar@/^/@{--}[ru] & & \mathbb{K}\ar[llllllll]_{0}}$ \qquad $\displaystyle \xymatrix@R=0.5em @C=0.5em{
				& & & &  \mathbb{K}^2 \ar[rrrrdddd]^{0} & & & & \\
				& & & \ar@{--}@/_/[rr] & & & & &  \\
				X_2 =& & & & & & & &  \\
				& & & & & & & &  \\
				\mathbb{K} \ar[rrrruuuu]^{0} & & \ar@/_/@{--}[lu] & & & & \ar@/^/@{--}[ru] & & \mathbb{K}\ar[llllllll]_{1}}$
		\end{center}
		We have $X_1 \cong X \ncong X_2$. So we conclude that there is no dense open set $O$ in $\mathsf{rep}(\mathcal{Q},((1),(2),(1))$ such that any representation $Y$ is isomorphic to $X$. Hence $\mathscr{C}_{\mathcal{Q},2}$ is not canonically Jordan recoverable.
	\end{remark}
	Now we introduce the following remarkable notion.
	\begin{definition} \label{mindef}
		A \new{minuscule} vertex $m \in Q_0$ of $\mathcal{Q}$ is a vertex such that, for any indecomposable representation $X \in \mathsf{rep}(Q)$, we have $\dim (X_m) \leqslant 1$.
	\end{definition}
	For gentle algebras, we can characterize minuscule vertices using combinatorial tools introduced at the end of the previous subsection.
	\begin{prop} \label{minstring}
		Let $\mathcal{Q} = (Q,I)$ be a gentle algebra. A vertex $m \in Q_0$ is minuscule if and only if each string of $\mathcal{Q}$ passes through $m$ at most once.
	\end{prop}
	\begin{proof}[Proof]
		This is a trivial consequence of the definition above and  Theorem \ref{BR}.
	\end{proof}
	\begin{ex} \label{minex}
		Consider the following gentle quiver.
		$$\displaystyle  \xymatrix@R=1em @C=1em{
			& & & 1 \ar[lldd]_\alpha \ar[rrdd]^\beta & & \\
			\mathcal{Q} = & & \ar@{--} @/^/[rd]& & & \\
			& 2\ar[rrrr]_\gamma & & & & 3
		}$$
		The string $\rho = \alpha \beta^{-1} \gamma$ is the only maximal string of $\mathcal{Q}$ and all the strings of $\mathcal{Q}$ are substrings of $\rho$. Hence $1$ and $3$ are minuscule vertices of $\mathcal{Q}$, but $2$ is not. 
	\end{ex}
	The fact that a vertex $m$ is minuscule seems to be a crucial criterion to get the subcategory $\mathscr{C}_{\mathcal{Q},m}$ to be canonically Jordan recoverable. The following recent result highlights this fact.
	\begin{theorem}[{\cite[Theorem 1.3]{GPT19}}] \label{GPT} If $Q$ is a Dynkin quiver and $m$ is a minuscule vertex, then $\mathscr{C}_{Q,m}$ is canonically Jordan recoverable.
	\end{theorem}
	It seems that $m$ has to be minuscule if we want to have canonical Jordan recoverability or Jordan recoverability of $\mathscr{C}_{\mathcal{Q},m}$; we will notice this point at the beginning of the proof (see Remark \ref{needmin}). However, in our gentle case, this condition is not sufficient.
	\begin{ex} \label{minnot1}
		Let us take again the gentle quiver from the previous example.
		
		For any vertex $m \in \{1,2,3\}$, we have the following indecomposable representations that are in $\mathscr{C}_{\mathcal{Q},m}$.
		\begin{center}
			$\displaystyle  \xymatrix@R=1em @C=1em{
				& & & \mathbb{K} \ar[lldd]_1 \ar[rrdd]^1 & & \\
				X = & & \ar@{--} @/^/[rd]& & & \\
				& \mathbb{K} \ar[rrrr]_0 & & & & \mathbb{K}
			}$ \qquad $\displaystyle  \xymatrix@R=1em @C=1em{
				& & & \mathbb{K} \ar[lldd]_0 \ar[rrdd]^1 & & \\
				Y = & & \ar@{--} @/^/[rd]& & & \\
				& \mathbb{K} \ar[rrrr]_1 & & & & \mathbb{K}
			}$
		\end{center}
		We can easily check that $\mathsf{GenJF}(X) = \mathsf{GenJF}(Y) = ((1),(1),(1))$ and consequently $\mathscr{C}_{\mathcal{Q},m}$ is not Jordan recoverable, even though $1$ and $3$ are minuscule.
	\end{ex}
	Let us see a less trivial example.
	\begin{ex} \label{minnot2} Consider the following gentle quiver.
		\vspace{0.2cm}
		\begin{center}
			$\displaystyle \xymatrix @R=1em @C=1em{1 \ar[rrr]_a & & \ar@{--} @/^1pc/ @<1ex> [rr]  & 2 \ar@<1ex> [rrr]^b \ar@<-1ex>[rrr]_c & & & 3}$
		\end{center}
		We can easily check that the only vertex which is minuscule is $1$. If we take the following representations
		\begin{center}
			$\displaystyle X =  \underset{=M(1)}{\underbrace{\xymatrix @R=1em @C=1em{\mathbb{K} \ar[rrr] & & \ar@{--} @/^1pc/ @<1ex> [rr]  & 0 \ar@<1ex> [rrr] \ar@<-1ex>[rrr] & & & 0}}} \oplus \underset{= M \left( \vcenter{
					\begin{small}
						\xymatrix @R=0.6em @C=0.6em{
							1 \ar[rd]^a & & & & \\
							& 2 \ar[rd]^c & & 2 \ar[rd]^c \ar[ld]^b & \\
							& & 3 & & 3 }
					\end{small}
				} \right)}{\underbrace{\xymatrix @R=1em @C=1em{\mathbb{K} \ar[rrr]_{\begin{tiny} \left[ 
								\begin{matrix}
									1 \\
									0
								\end{matrix} \right]
						\end{tiny}} & & \ar@{--} @/^1pc/ @<1ex> [rr]  & \mathbb{K}^2 \ar@<1ex> [rrr]^{\begin{tiny} \left[ 
								\begin{matrix}
									0 & 0 \\
									0 & 1
								\end{matrix} \right]
						\end{tiny}} \ar@<-1ex>[rrr]_{\begin{tiny} \left[ 
								\begin{matrix}
									1 & 0 \\
									0 & 1
								\end{matrix} \right]
						\end{tiny}} & & & \mathbb{K}^2}}}$
		\end{center}
		and
		\begin{center}
			$\displaystyle Y = ( \underset{ M \left( \vcenter{
					\begin{small}
						\xymatrix @R=0.6em @C=0.6em{
							1 \ar[rd]^a & & \\
							& 2 \ar[rd]^c & \\
							& & 3 }
					\end{small}
				} \right)}{\underbrace{\xymatrix @R=1em @C=1em{\mathbb{K} \ar[rrr]_1 & & \ar@{--} @/^1pc/ @<1ex> [rr]  & \mathbb{K} \ar@<1ex> [rrr]^0 \ar@<-1ex>[rrr]_1 & & & \mathbb{K}}}} )^2 $
		\end{center}
		\vspace*{0.2cm}
		then we get $\mathsf{GenJF}(X) = \mathsf{GenJF}(Y) = ((2),(2),(2))$.
	\end{ex}
	From these last examples, it could seem that we have to work with a special kind of quivers to get what we want, those which have only minuscule vertices. However it is more complicated than that. We can have gentle quivers with only minuscule vertices but with some $\mathscr{C}_{\mathcal{Q},m}$ which are still neither canonically Jordan recoverable (as we already saw in Remark \ref{CJRneed1}) nor Jordan recoverable.
	\begin{ex} \label{minnot3} Let us take again the quiver of Example \ref{JRex1}.
		\begin{center}
			$\displaystyle  \xymatrix @R=0.5em @C=0.5em{ 
				& & & & 4 & &   \\
				\mathcal{Q} & =& & & & & \\
				&  & 1 \ar[rr]^a & & 2 \ar[uu]^b  & \ar@{--} @/_/[lu]  & 3 \ar[ll]^c }$
		\end{center}
		We can check that all the vertices of $\mathcal{Q}$ are minuscule.
		
		We can easily see that $\mathscr{C}_{\mathcal{Q},3}$ and $ \mathscr{C}_{\mathcal{Q},4}$ are (canonically) Jordan recoverable for the same reasons that $\mathscr{C}_{\mathcal{Q},1}$ is.  However, surprisingly, $\mathscr{C}_{\mathcal{Q},2}$ is not Jordan recoverable, even if $2$ is a minuscule vertex and $\mathcal{Q}$ is derived equivalent to an $A_4$-Dynkin type quiver. Indeed with these two following representations in $\mathscr{C}_{\mathcal{Q},2}$:
		\begin{itemize}
			\item[•]$\displaystyle X = M \left( \vcenter{
				\begin{small}
					\xymatrix @R=0.6em @C=0.6em{
						2 \ar[rd]^b &  \\
						& 4  }
				\end{small}
			} \right) \oplus M \left( \vcenter{
				\begin{small}
					\xymatrix @R=0.6em @C=0.6em{
						1 \ar[rd]^a &\\
						& 2   }
				\end{small}
			} \right) \oplus M \left( \vcenter{
				\begin{small}
					\xymatrix @R=0.6em @C=0.6em{
						1 \ar[rd]^a & & 3 \ar[ld]^c \\
						& 2  &  }
				\end{small}
			} \right) \oplus M \left( \vcenter{
				\begin{small}
					\xymatrix @R=0.6em @C=0.6em{
						3 \ar[rd]^c & \\
						& 2  }
				\end{small}
			} \right)$ ;
			
			\item[•] $X' = M \left( \vcenter{
				\begin{small}
					\xymatrix @R=0.6em @C=0.6em{
						1 \ar[rd]^a & &\\
						& 2 \ar[rd]^b & \\
						& & 4 }
				\end{small}
			} \right) \oplus M(2) \oplus  M \left( \vcenter{
				\begin{small}
					\xymatrix @R=0.6em @C=0.6em{
						1 \ar[rd]^a & & 3 \ar[ld]^c \\
						& 2  &  }
				\end{small}
			} \right) \oplus M \left( \vcenter{
				\begin{small}
					\xymatrix @R=0.6em @C=0.6em{
						3 \ar[rd]^c & \\
						& 2}
				\end{small}
			} \right)$;
		\end{itemize} we get $\mathsf{GenJF}(X) = \mathsf{GenJF}(X') = ((2),(3,1),(2),(1))$.
	\end{ex}
	
	\section{Proof of the main results}
	
	We recall and reformulate the main results that we will prove.
	\begin{theorem} \label{main2}
		Let $\mathcal{Q} = (Q,I)$ be a finite connected gentle quiver and $m \in Q_0$. The subcategory $\mathscr{C}_{\mathcal{Q},m}$ is canonically Jordan recoverable if and only if $m$ satisfies the following conditions:
		\begin{itemize}
			\item[$(i)$] For any pair of strings $\rho$ and $\nu$ passing through $m$, there is no arrow $\alpha \in Q_1$ such that $\alpha \notin \mathsf{Supp}_1(\rho) \cup \mathsf{Supp}_1(\nu)$, $s(\alpha) \in \mathsf{Supp}_0(\rho)$ and $t(\alpha) \in \mathsf{Supp}_0(\nu)$;
			
			\item[$(ii)$] At least one of the two following conditions:
			\begin{itemize}
				\item[$(a)$] There is at most one arrow $\alpha \in Q_1$ such that $s(\alpha) = m$, and there is at most one arrow $\beta \in Q_1$ such that $t(\beta) = m$ ; if $\alpha$ and $\beta$ both exist then $\alpha \beta \in I$;
				
				\item[$(b)$] There is at most one string of $\mathcal{Q}$ maximal by inclusion and passing through $m$. 
			\end{itemize}
		\end{itemize}
	\end{theorem}
	
	\begin{theorem}\label{2ndmain2}
		Let $\mathcal{Q} = (Q,I)$ be a finite connected gentle quiver and $m \in Q_0$. The subcategory $\mathscr{C}_{\mathcal{Q},m}$ is Jordan recoverable if and only if $m$ satisfies the point $(ii)$ of Theorem \ref{main2} and:
		\begin{itemize}
			\item[$(i*)$] The strings that start at $m$ are uniquely determined by their endpoint.
		\end{itemize}
	\end{theorem}
	
	\subsection{Proof of Theorem \ref{main}}
	\label{sub:proof1}
	
	To begin with, we will prove the following little lemma.
	\begin{lemma} \label{1imply0} Let $m$ be a vertex of the quiver $\mathcal{Q}$. We have the following assertions:
		\begin{itemize}
			\item[•] If $m$ satisfies $(i)$ then $m$ satisfies $(i*)$;
			
			\item[•] If $m$ satisfies $(i*)$ then $m$ satisfies the following:
			\begin{itemize}
				\item[$(o)$] Any string passing through $m$ passes through any vertex of $\mathcal{Q}$ at most once.
			\end{itemize}
			\item[•] If $m$ satisfies $(o)$ then $m$ is minuscule.
		\end{itemize}
	\end{lemma}
	\begin{remark}
		The condition $(o)$ prevents, in particular, the existence of bands passing through $m$.
	\end{remark}
	\begin{proof}[Proof]
		Let us prove first that $(i)$ implies $(i*)$ by proving the contrapositive. Let us assume that $m$ does not satisfy $(i*)$.
		
		By hypothesis, there exist $\rho$ and $\nu$ two distinct strings such that $s(\nu) = s(\rho) = m$ and $t(\nu) = t(\rho)$. In the set $S$ of such pairs of strings $(\rho, \nu)$, let us consider a pair $(\rho_0, \nu_0) \in S$ minimal with respect to the sum of the lengths of the strings.  Up to exchanging the role of $\rho_0$ and $\nu_0$, we can assume that $\ell(\rho_0) \leqslant \ell(\nu_0)$. 
		
		First we can note that $\nu_0$ cannot be a lazy path. Next, by construction of our strings, there exist an arrow $\alpha$, $\varepsilon \in \{\pm 1\}$ and a substring $\nu'$ such that $\nu_0 = \alpha^\varepsilon \nu'$ and $\alpha \notin \mathsf{Supp}_1(\nu') \cup \mathsf{Supp}_1(\rho_0)$; else we could construct a substring $\sigma$ of $\rho_0$ such that $(\sigma, \nu') \in S$ $-$ which leads us to a contradiction.
		
		Thus $\rho_0$ and $\nu'$ are two strings such that there is an arrow $\alpha$ which satisfies:
		\begin{itemize}
			\item[•] $s(\alpha^\varepsilon) = t(\nu')$, and so $s(\alpha^\varepsilon) \in \mathsf{Supp}_0(\nu')$;
			
			\item[•] $t(\alpha^\varepsilon) = t(\rho_0)$ and consequently $t(\alpha^\varepsilon) \in \mathsf{Supp}_0(\rho_0)$;
			
			\item[•] $\alpha \notin \mathsf{Supp}_1(\nu') \cup \mathsf{Supp}_1(\rho_0)$.
		\end{itemize}
		We can conclude that $m$ does not satisfy $(i)$.
		
		Secondly we show that $(i*)$ implies $(o)$ by proving the contrapositive. Let us assume that $m$ does not satisfy $(o)$.
		
		Then there exists a string $\mu$ passing through $m$ and passing through a vertex at least twice. Let $T$ be the set of such strings. We can consider a string $\mu_0$ minimal by inclusion in $T$.  By hypothesis, $\mu_0$ is not lazy and there is only one vertex in $Q_0$ which $\mu_0$ is passing through at least twice. Up to replacing $\mu_0$ by its inverse, we can affirm this vertex is $t(\mu_0)$ and $\mu_0$ is passing through it exactly twice. 
		
		Let us construct two distinct strings $\rho$ and $\nu$ such that $s(\rho) = s(\nu) = m$ and $t(\rho) = t(\nu)$ from $\mu_0$. We have to distinguish the following cases:
		\begin{itemize}
			\item[•] if $s(\mu_0) = m = t(\mu_0)$ : we can take $\rho = e_m$ and $\nu = \mu_0$;
			
			\item[•] if $s(\mu_0) = m \neq t(\mu_0)$: we can consider $\rho$ to be the substring of $\mu_0$ minimal by inclusion such that $s(\rho) = m$ and $t(\rho) = t(\mu_0)$, and $\nu = \mu_0$;
			
			\item[•] if $s(\mu_0) \neq m \neq t(\mu_0)$: hence $s(\mu_0) = t(\mu_0)$! That is why we can write $\mu_0 = \rho \nu^{-1}$ where $\rho$ and $\nu$ are two strings satisfying what we want.
		\end{itemize}
		Thus $m$ does not satisfy $(i*)$.
		
		Finally $(o)$ implies minusculeness by using Proposition \ref{minstring}.
	\end{proof}
	
	\begin{cor} \label{needmin}
		Theorem \ref{2ndmain2} implies that $m$ must be minuscule in order to get $\mathscr{C}_{\mathcal{Q},m}$  Jordan recoverable.
	\end{cor}
	
	\begin{remark} \label{notsuff}
		We already saw that is not sufficient (see Remark \ref{CJRneed1} or Example \ref{minnot3}).
	\end{remark}
	Now let us prove the main theorem. First we show that these conditions are sufficient.
	\begin{prop} \label{1et2b}
		Let $m$ be a vertex of $\mathcal{Q}$ satisfying $(i)$ and $(ii)(b)$. Then $\mathscr{C}_{\mathcal{Q},m}$ is canonically Jordan recoverable.
	\end{prop}
	
	\begin{proof}[Proof]
		Let $m$ be as supposed. By $(i)$, which implies $(o)$ by the previous lemma, and the finiteness of $\mathcal{Q}$, there exists at least one string maximal by inclusion in the set of strings that pass through $m$. By $(ii)(b)$, it is unique. 
		
		Let $\rho$ be the obtained string. If $\rho$ is $e_m$, it means that $\mathcal{Q} \cong A_1$ and we conclude easily. If $\ell(\rho) > 0$, we can consider the (gentle) quiver $\mathcal{R} = (R,J)$ with:
		\begin{itemize}
			\item[•] $R_0 = \mathsf{Supp}_0(\rho)$ and $R_1 =  \mathsf{Supp}_1(\rho)$;
			
			\item[•] $\sigma, \tau : R_1 \longrightarrow R_0$ the source and target functions which coincide with $s$ and $t$ on $R_1 \subset Q_1$;
			
			\item[•] $J = \{0\}$.
		\end{itemize}
		By construction $\mathcal{R}$ is a $A_n$-Dynkin type quiver (where $n = \ell(\rho)+1$) and by $(i)$, for any arrow $\alpha \in Q_1 \setminus \mathsf{Supp}_1(\rho)$, we get $s(\alpha) \notin \mathsf{Supp}_0(\rho)$ or $t(\alpha) \notin \mathsf{Supp}_0(\rho)$, and therefore, for any $X \in \mathscr{C}_{\mathcal{Q},m }$,  $X_{s(\alpha)} = 0$ or $X_{t(\alpha)} = 0$. So $X_\alpha$ is a zero-morphism. Thus any $Y \in \mathsf{rep}(\mathcal{Q})$ such that $\vdim(Y) = \vdim(X)$ has to satisfy $Y_\alpha = 0$. Such an observation allows us to say that we only have to focus on what is going on in $\mathcal{R}$.
		
		Then to know if $\mathscr{C}_{\mathcal{Q},m}$ is canonically Jordan recoverable amounts to knowing if $\mathscr{C}_{\mathcal{R},m}$ is. By Theorem \ref{GPT}, we can conclude that $\mathscr{C}_{\mathcal{Q},m}$ is canonically Jordan recoverable.
	\end{proof}
	\begin{lemma} \label{stringstruct1*et2a}
		Let $m$ be a vertex of $\mathcal{Q}$ satisfying $(ii)(a)$. Let $\Sigma_\mathcal{Q}(m)$ be the set of all the strings of $\mathcal{Q}$ that pass through $m$. Then for any string $\rho \in \Sigma_\mathcal{Q}(m)$, $s(\rho) = m$ or $s(\rho^{-1}) = m$.
	\end{lemma}
	\begin{proof}[Proof]
		This follows immediately from the hypotheses.
	\end{proof}
	\begin{prop} \label{1*et2aGenJF}
		Let $m$ be a vertex of $\mathcal{Q}$ satisfying $(i*)$ and $(ii)(a)$. Then for any $X \in \mathscr{C}_{\mathcal{Q},m},\ \mathsf{GenJF}(X) = ((\dim X_q))_{q \in Q_0}$. That is to say, the Jordan form of a generic nilpotent endomorphism of $X$ at each vertex consists of a single block.
	\end{prop}
	\begin{proof}[Proof]
		Let $m$ be as supposed. By the previous lemma, we get that any string of $\Sigma_{\mathcal{Q}}(m)$, up to inverse, starts at $m$ and can be identified by its ending vertex.
		
		Now we will define a total order on $\Sigma_\mathcal{Q}(m)$. This order was first introduced by Sheila Brenner \cite{B86} and then reformulated and used in particular by Jan Schroër in his thesis \cite{S98}. Let $\mu = \alpha_k^{\varepsilon_k} \ldots \alpha_1^{\varepsilon_1} $ and $\nu = \beta_p^{\xi_p} \ldots \beta_1^{\xi_1}$ be two strings of $\Sigma_\mathcal{Q}(m)$. If $\alpha_1 = \beta_1$, let $j$ be the maximal integer such that $\alpha_i = \beta_i$ for all $1 \leqslant i \leqslant j$; otherwise put $j=0$. We will say that $\mu \leqslant  \nu$ if and only if:
		\begin{itemize}
			\item[•] if $j < \min(k,p)$ then $\varepsilon_{j+1} = - 1$ and $\xi_{j+1} = 1$;
			
			\item[•] if $j = k < p$ then $\xi_{j+1} = 1$;
			
			\item[•] if $j = p < k$, then $\varepsilon_{j+1} = -1$;
			
			\item[•] or if $j = p = k$.
		\end{itemize}
		
		We can have in mind the following drawing.
		\begin{center}
			$\displaystyle 
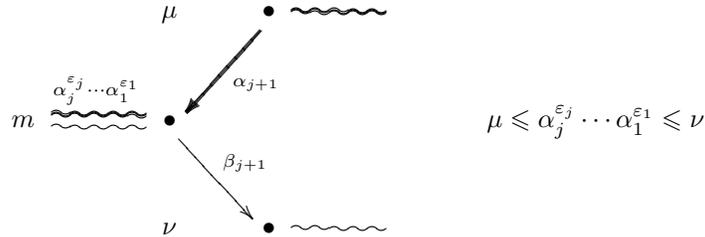
\xymatrix@R=1em @C=1em{
				& & & \mathbf{\mu} & & \ar@<0.7ex> [lldd]^{\alpha_{j+1}} \ar@<0.6ex> [lldd]  \ar@<0.5ex> [lldd] \bullet \ar@<0.1ex>@{~}[rrr] \ar@<-0.1ex>@{~}[rrr] \ar@{~}[rrr]  & & & & &\\
				& & & & & & & & & & \\
				m \ar@<0.6ex> @{~}[rrr] \ar@<0.7ex> @{~}[rrr]^{\alpha_j^{\varepsilon_j} \cdots \alpha_1^{\varepsilon_1}}  \ar@<0.5ex> @{~}[rrr] \ar@<-0.5ex> @{~} [rrr] & & & \bullet  \ar@<-0.5ex> [rrdd]^{\beta_{j+1}} & & & & & & & \mathbf{\mu} \leqslant \alpha_j^{\varepsilon_j} \cdots \alpha_1^{\varepsilon_1} \leqslant \nu \\
				& & & & & & & & & &\\
				& & & \nu & & \bullet \ar@{~}[rrr] & & & & & } $
		\end{center}
		We can easily check that we have a total order relation. 
		
		By Theorem \ref{BR}, any string of $\Sigma_{\mathcal{Q}}(m)$ corresponds to an indecomposable representation of $\mathscr{C}_{\mathcal{Q},m}$, and vice-versa. Moreover, by $(i*)$, for any pair of strings $\mu, \nu \in \Sigma_{\mathcal{Q}}(m)$ such that $\mu \leqslant \nu$, there is a unique common substring $\alpha_j^{\varepsilon_j} \cdots \alpha_1^{\varepsilon_1}$ which is at the bottom of $\mu$ and on the top of $\nu$. 
		
		Such observations allow us to define a relation $\leqslant$ on indecomposable representations of $\mathscr{C}_{\mathcal{Q},m}$ as follows: for any pair of indecomposable representations $X,Y \in \mathscr{C}_{\mathcal{Q},m}$, we have $X \leqslant Y$ if and only if there exists a non-zero morphism from $Y$ to $X$. Then $\mu \leqslant \nu $ if and only if $M(\mu) \leqslant M(\nu)$ as the only non-zero morphism (up to a scalar) from $M(\nu)$ to $M(\mu)$ corresponds to the substring $\alpha_j^{\varepsilon_j} \cdots \alpha_1^{\varepsilon_1}$, by Proposition \ref{morph}. Indeed we note that $\alpha_j^{\varepsilon_j} \cdots \alpha_1^{\varepsilon_1}$ is the unique common substring, maximal by inclusion, of $\mu$ and $\nu$ as a direct consequence of Lemma \ref{stringstruct1*et2a}. As a result, we translate the total order on strings of $\Sigma_\mathcal{Q}(m)$ to a total order on indecomposable representations of $\mathscr{C}_{\mathcal{Q},m}$.
		
		By $(i*)$, $\Sigma_\mathcal{Q}(m)$ is a finite set and so is the isomorphism classes of indecomposable representations in $\mathscr{C}_{\mathcal{Q},m}$.  Let $X(d) \leqslant \ldots \leqslant X(1)$ be the indecomposable representations contained in $\mathscr{C}_{\mathcal{Q},m}$. Consider a decomposition of $X = ((X_q)_{q \in Q_0}, (X_\alpha)_{\alpha \in Q_1}) \in \mathscr{C}_{\mathcal{Q},m}$ as a sum of string modules 
		\begin{center}
			$\displaystyle X = \bigoplus_{i=1}^d X(i)^{\delta_i}$
		\end{center}
		with $\delta_i \in \mathbb{Z}_{\geqslant 0}$ for $1 \leqslant i \leqslant d$. We choose a basis for each $X_q$ adapted to this decomposition, so that the maps $X_\alpha$ are given by matrices with at most one non-zero entry in each row or column, and all the non-zero entries are equal to $1$. Note that this implies that the chosen basis of each $X_q$ can be identified with a subset of the basis of $X_m$.
		
		We now define an order on the basis of $X_m$ (and thus by restriction on the basis of each $X_q$). Each basis element belongs to a particular string in the string decomposition. Then we fix a total order extending the total order $\leqslant$ we already defined on strings. In other words, the basis elements of $X_m$  corresponding to $X(i)$ are $x_{\delta_1 + \ldots + \delta_{i-1} + 1}, \ldots, x_{\delta_1 + \ldots + \delta_{i-1} + \delta_i}$. By our definition of the order $\leqslant$, we observe that the basis of $X_q$ consists of  $x_{\delta_1+\ldots+\delta_{i-1}+ 1}, \ldots, x_{\delta_1 + \ldots + \delta_j}$ for some $1 \leqslant i \leqslant j \leqslant d$. 
		
		Let us consider the following linear transformations at each vertex $q$, inspired by Example \ref{JRex1}:
		\begin{center}
			$N_q(x_k) = \left\{ 
			\begin{matrix}
				x_{k-1} & \text{ if } \delta_1 + \cdots + \delta_{i-1} + 2 \leqslant k \leqslant \delta_1 + \cdots + \delta_{j} \hfill \\
				0 & \text{otherwise.} \hfill
			\end{matrix}
			\right.$
		\end{center}
		and we define $N_q = 0$ if $X_q = 0$.
		
		By construction, if $N = (N_q)$ is a well-defined endomorphism, then it is nilpotent. Let us check that $N$ is an endomorphism of $X$. So we want to prove that for any $\alpha \in Q_1$,
		\begin{center}
			$(\diamond): X_\alpha N_{s(\alpha)} = N_{t(\alpha)} X_\alpha$ 
		\end{center}
		Let $\alpha \in Q_1$. We can note that if $X_{s(\alpha)} = 0$ or $X_{t(\alpha)} = 0$ then $(\diamond)$ is automatically satisfied. Suppose that $X_{s(\alpha)} \neq 0$ and $X_{t(\alpha)} \neq 0$. So there is a pair of strings $\mu,\nu \in \Sigma_{\mathcal{Q}}(m)$ such that $t(\mu) = s(\alpha)$ and $t(\nu) = t(\alpha)$. As any string of $\Sigma_{\mathcal{Q}}(m)$ is characterized by its ending vertex, then we have either $(A) : \mu = \alpha^{-1} \nu$  or $(B) : \nu = \alpha \mu$.
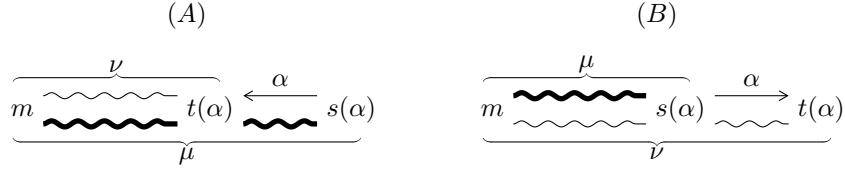
		\begin{center}
			\begin{tikzpicture}[>= angle 60,scale=1.25]
				\node (1) at (0,0){$m$};
				\node (2) at (2,0){$t(\alpha)$};
				\node (3) at (3.5,0){$s(\alpha)$};
				\draw[<-] ([yshift=1.5mm]2.east)-- node[above]{$\alpha$}([yshift=1.5mm]3.west);
				\draw[-,decorate, decoration={snake,amplitude=.4mm}] ([yshift=1.5mm]1.east)--([yshift=1.5mm]2.west);
				\draw[-,line width=0.7mm,decorate, decoration={snake,amplitude=.4mm}] ([yshift=-1.5mm]1.east)--([yshift=-1.5mm]2.west);
				\draw[-,line width=0.7mm,decorate,decoration={snake,amplitude=.4mm}] ([yshift=-1.5mm]2.east)--([yshift=-1.5mm]3.west);
				\draw[decorate,decoration={calligraphic brace, amplitude=3pt}]
				(-.1, .3) -- (2.1,.3) node[midway,auto]{$\nu$}; 
				\draw[decorate,decoration={calligraphic brace, amplitude=3pt,mirror}]
				(-.1, -.3) -- (3.6,-.3) node[midway,auto,swap]{$\mu$};
				\node at (1.75,1){$(A)$};
				\begin{scope}[xshift = 5cm]
					\node (1) at (0,0){$m$};
					\node (2) at (2,0){$s(\alpha)$};
					\node (3) at (3.5,0){$t(\alpha)$};
					\draw[->] ([yshift=1.5mm]2.east)-- node[above]{$\alpha$}([yshift=1.5mm]3.west);
					\draw[-,line width=0.7mm,decorate, decoration={snake,amplitude=.4mm}] ([yshift=1.5mm]1.east)--([yshift=1.5mm]2.west);
					\draw[-,decorate, decoration={snake,amplitude=.4mm}] ([yshift=-1.5mm]1.east)--([yshift=-1.5mm]2.west);
					\draw[-,decorate,decoration={snake,amplitude=.4mm}] ([yshift=-1.5mm]2.east)--([yshift=-1.5mm]3.west);
					\draw[decorate,decoration={calligraphic brace, amplitude=3pt}]
					(-.1, .3) -- (2.1,.3) node[midway,auto]{$\mu$}; 
					\draw[decorate,decoration={calligraphic brace, amplitude=3pt,mirror}]
					(-.1, -.3) -- (3.6,-.3) node[midway,auto,swap]{$\nu$};
					\node at (1.75,1){$(B)$};
				\end{scope}
			\end{tikzpicture}
		\end{center}
		In either configuration we get $\mu \leqslant \nu$. Hence, as strings correspond to indecomposable representations, if $X_{s(\alpha)}$ is generated by $(x_{\delta_1 + \cdots + \delta_{i_1} + 1}, \ldots, x_{\delta_1 + \cdots + \delta_{j_1}})$ and if $X_{t(\alpha)}$ is generated by $(x_{\delta_1 + \cdots + \delta_{i_2} + 1}, \ldots, x_{\delta_1 + \cdots + \delta_{j_2}})$, then :
		\begin{itemize}
			\item[•] if we are in the configuration $(A)$: $i_1 = i_2 < j_1 < j_2$;
			
			\item[•] otherwise in the configuration $(B)$: $i_1 < i_2 < j_1 = j_2$.
		\end{itemize}
		We easily check in each configuration, the square corresponding to the arrow $\alpha$ is commutative. Consequently we get $(\diamond)$ for any $\alpha$, and we conclude $N$ is a nilpotent endomorphism.
		
		By construction, we easily get that $\mathsf{JF}(N) = ((\dim(X_q)))_{q \in Q_0}$. This $\#Q_0$-tuple of partitions is the maximal one (in the sense of $\unlhd$ order) that we can attain among all the Jordan forms we can have from any $N' \in \mathsf{NEnd}(X)$. Hence $\mathsf{GenJF}(X) = ((\dim(X_q)))_{q \in Q_0}$ as claimed (see Theorem \ref{GPT3} or \cite[Section 2.2]{GPT19}).
	\end{proof}
	The corollary below will be useful to prove Theorem \ref{2ndmain}.
	\begin{cor} \label{JFin1*and2}
		Let $m$ be a vertex of $\mathcal{Q}$ satisfying $(i*)$ and $(ii)(a)$. Then $\mathscr{C}_{\mathcal{Q},m}$ is Jordan recoverable.
	\end{cor}
	\begin{proof}[Proof]
		By previous proposition, the dimension vectors of the indecomposable representations of $\mathscr{C}_{\mathcal{Q},m}$ are linearly independent. So we get for free that $\mathscr{C}_{\mathcal{Q},m}$ is Jordan recoverable.
	\end{proof}
	Here we want to establish a stronger result.
	\begin{lemma} \label{distancetom}
		Let $m$ be a vertex of $\mathcal{Q}$ satisfying $(i*)$ and $(ii)(a)$. Define $\Delta : Q_0 \longrightarrow \mathbb{Z}_{\geqslant 0} \cup \{\infty\}$ such that for any $q \in Q_0$:
		\begin{center}
			$\displaystyle \Delta(q) = \left\{ \begin{matrix}
				k & \text{if there exists a string } \rho \text{ such that } \ell(\rho) = k,\ s(\rho) = m \text{ and } t(\rho) = q \\
				\infty & \text{otherwise } \hfill
			\end{matrix} \right.$
		\end{center}
		Then $\Delta$ is well-defined. Moreover for any pair of arrows $\alpha, \beta$ such that $\beta \alpha \in I$ and 
		\begin{equation} 
			\Delta(s(\beta)) \leqslant \min(\Delta(t(\beta)), \Delta(s(\alpha))), \tag{$\otimes$}
		\end{equation} and for any $X \in \mathscr{C}_{\mathcal{Q},m}$, we have: 
		\begin{equation}
			\dim X_{s(\beta)}  \geqslant \dim X_{t(\beta)} + \dim X_{s(\alpha)}. \tag{$\circledast$}
		\end{equation}
	\end{lemma}
	\begin{proof}[Proof]
		Let $m$ be as supposed. As $m$ satisfies $(i*)$ by Lemma \ref{1imply0}, then any string of $\Sigma_\mathcal{Q}(m)$, up to inverse, starts at $m$ and is characterized by its ending vertex by Lemma \ref{stringstruct1*et2a}. So $\Delta$ is well-defined.
		
		Let $\alpha, \beta \in Q_1$ such that $\beta \alpha \in I$ and $(\otimes)$. Put $q = s(\beta) = t(\alpha)$. Obviously if $\Delta(q) = \infty$ then for any representation $X$ in $\mathscr{C}_{\mathcal{Q},m}$ we have $X_q = X_{s(\alpha)} = X_{t(\beta)} = 0$ and we conclude $\circledast$.
		
		Assume that $\Delta(q) < \infty$. There exists a string $\rho_q \in \Sigma_\mathcal{Q}(m)$ unique by $(i*)$ such that $t(\rho_q) = q$. Note that $s(\alpha)$ and $t(\beta)$ are not in $\mathsf{Supp}_0(\rho_q)$ by $(\otimes)$. Then $\rho_q$ does not end by $\alpha$ or $\beta^{-1}$.
		
		By gentleness of $\mathcal{Q}$, $\beta \rho_q$ and $\alpha^{-1} \rho_q$ are strings of $\Sigma_\mathcal{Q}(m)$. Thus $\Delta(s(\alpha)) = \Delta(t(\beta)) = \Delta(q) + 1$. 
		\begin{center}
			$\displaystyle 
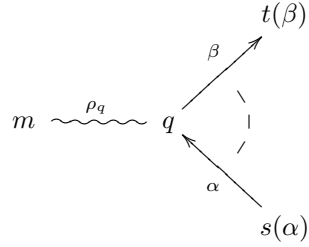
\xymatrix@R=1em @C=1em{
				& & & & & t(\beta) \\
				& & & & & \\
				m \ar@{~}[rrr]^{\rho_q} & & & q \ar[rruu]^\beta \ar@{<-}[rrdd]_\alpha & & \\
				& & & & \ar@/_1pc/@{--}[uu] & \\
				& & & & & s(\alpha)}$
		\end{center}
		Moreover, by $(i*)$ we can assert that:
		\begin{itemize}
			\item[•] there is no string $\nu \in \Sigma_\mathcal{Q}(m)$ such that $s(\alpha)$ and  $t(\beta)$ are both in $\mathsf{Supp}_0(\nu)$: if such a string exists, the minimal substring in $\Sigma_{\mathcal{Q}}(m)$ passing through $s(\alpha)$ and $t(\beta)$ must end at  $s(\alpha)$, or $t(\beta)$, and this raises a contradiction with $(i*)$;
			
			\item[•] For any string $\mu$ such that  $t(\beta) \in \mathsf{Supp}_0(\mu)$ or $s(\alpha) \in \mathsf{Supp}_0(\mu)$, then $q \in \mathsf{Supp}_0(\mu)$.
		\end{itemize} 
		Let us consider $\pi_m(r) = \{ \rho \in \Sigma_\mathcal{Q}(m) \mid r \in \mathsf{Supp}_0(\rho)\}$ for $r \in Q_0$. Following our previous assertions, we have $\pi_m(s(\alpha)) \cap \pi_m(t(\beta)) = \varnothing$ and $\pi_m(s(\alpha)) \cup \pi_m(t(\beta)) \subseteq \pi_m(q)$.
		
		Let $X \in \mathscr{C}_{\mathcal{Q},m}$. Then we can decompose $X$ and write
		\begin{center}
			$\displaystyle X = \bigoplus_{\rho \in \Sigma_\mathcal{Q}(m)} M(\rho)^{d_\rho}$
		\end{center}
		for some $d_\rho \in \mathbb{Z}_{\geqslant 0}$ for any $\rho \in \Sigma_\mathcal{Q}(m)$. Again thanks to $(i*)$, for any $\rho \in \Sigma_\mathcal{Q}(m)$ and for any $r \in \mathsf{Supp}_0(\rho)$, $\dim M(\rho)_r = 1$. Thus:
		\begin{center}
			$\displaystyle \dim X_q = \sum_{\rho \in \pi_m(q)} d_\rho \geqslant \sum_{\rho \in \pi_m(s(\alpha))} d_\rho  + \sum_{\rho \in \pi_m(t(\beta))} d_\rho  = \dim X_{s(\alpha)} + \dim X_{t(\beta)}$
		\end{center}
		\hfill
	\end{proof}
	\begin{cor}
		Let $m$ be a vertex of $\mathcal{Q}$ satisfying $(i*)$ and $(ii)(a)$. For any $X \in \mathscr{C}_{\mathcal{Q}, m}$, and for all $\beta \in Q_1$ such that $\Delta(s(\beta)) \leqslant \Delta(t(\beta))$ then $\dim X_{s(\beta)} \geqslant \dim X_{t(\beta)}$. Moreover, for all $\alpha \in Q_1$ such that $\Delta(t(\alpha)) \leqslant \Delta(s(\alpha))$ then $\dim X_{t(\alpha)} \geqslant \dim X_{s(\alpha)}$.
	\end{cor}
	\begin{proof}[Proof]
		This is an obvious consequence of Lemma \ref{distancetom}.
	\end{proof}
	\begin{lemma} \label{injsurj0}
		Let $m$ be a vertex of $\mathcal{Q}$ satisfying $(i)$ and $(ii)(a)$. For any $X \in \mathscr{C}_{\mathcal{Q},m}$ and for any $\alpha \in Q_1$, we have $X_\alpha$ is injective or surjective.
	\end{lemma}
	\begin{proof}[Proof]
		Let $m$ be as assumed. Let $X \in \mathscr{C}_{\mathcal{Q},m}$ and $\alpha \in Q_1$. We have to treat two cases.
		
		First assume that there is no string $\rho \in \Sigma_\mathcal{Q}(m)$ such that $\alpha \in \mathsf{Supp}_1(\rho)$. Hence, by $(i*)$, $s(\alpha)$ is not in the vertex support of any string of $\Sigma_\mathcal{Q}(m)$, or $t(\alpha)$ is not in the vertex support of any string of $\Sigma_\mathcal{Q}(m)$. Thus $X_{s(\alpha)} = 0$ or $X_{t(\alpha)} = 0$ which implies either way that $X_\alpha$ is injective or surjective respectively.
		
		Now assume that there exists a string $\rho \in \Sigma_\mathcal{Q}(m)$ such that $\alpha \in \mathsf{Supp}_1(\rho)$. Using bases adapted to the decomposition of $X$ by string modules of $\Sigma_\mathcal{Q}(m)$ as described in the proof of Proposition \ref{1*et2aGenJF}, we can note that $X_\alpha$ is a linear transformation of maximal rank. Hence $X_\alpha$ is injective or surjective.
	\end{proof} 
	\begin{lemma} \label{injsurj}
		Let $m$ be a vertex of $\mathcal{Q}$ satisfying $(i)$ and $(ii)(a)$. Let $X \in \mathscr{C}_{\mathcal{Q},m}$ and $\underline{\lambda} = \mathsf{GenJF}(X)$. Then there exists a dense open set $\Omega \subseteq \mathsf{rep}(\mathcal{Q}, \underline{\lambda})$ such that any representation $Y \in \Omega$ has all its maps $Y_\alpha$ to be injective or surjective for any $\alpha \in Q_1$.
	\end{lemma}
	\begin{proof}[Proof]
		Let $X$ and $\underline{\lambda}$ be as assumed. By Proposition \ref{1*et2aGenJF}, we obtain that $\underline{\lambda} = ((\dim X_q))_{q \in Q_0}$. For $q \in Q_0$, consider $Z_q = \mathbb{K}^{\dim X_q}$ endowed with a fixed nilpotent endomorphism $N_q$ of type $\lambda^q = (\dim X_q)$ for all $q \in Q_0$. We also fix for each $Z_q$ an adapted basis to $N_q$. 
		
		To study the set $\mathsf{rep}(\mathcal{Q}, \underline{\lambda})$, we consider representations $Y = ((Z_q)_{q \in Q_0}, (Y_\alpha)_{\alpha \in Q_1})$ such that $N= (N_q)_{q \in Q_0}$ is a well-defined endomorphism. We can show that there exists a dense open set $\Omega$ in $\mathsf{rep}(\mathcal{Q}, \underline{\lambda})$ in which any representation $Y_{\alpha}$ is injective or surjective. 
		
		For any $\alpha \in Q_1$:
		\begin{itemize}
			\item[•] if $\Delta(s(\alpha)) = \infty$ or $\Delta(t(\alpha)) = \infty$, then $Y_\alpha = 0$ and either way $Y_\alpha$ is injective in the first case or surjective in the second one. In this case, we define $O_\alpha = \mathsf{rep}(\mathcal{Q}, \lambda)$;
			
			\item[•] if $\Delta(s(\alpha)) = \Delta(t(\alpha)) + 1$: by $(\circledast)$ which allows us to assume that $Z_{s(\alpha)} \cong \mathbb{K}^a$ and $Z_{t(\beta)} \cong \mathbb{K}^{a+b}$ with $a > 0$ and $b \geqslant 0$, and by taking $(x_1, \ldots, x_a)$ and $(y_1, \ldots, y_{a+b})$ as adapted bases of respectively $Z_{s(\alpha)}$ and $Z_{t(\alpha)}$ for $N_{s(\alpha)}$ and $N_{t(\alpha)}$, we have $Y_{\alpha}(x_1) = k_1 y_{1+b} + \cdots + k_a y_{a+b}$ with $k_1, \ldots, k_a \in \mathbb{K}$. By commutative square relations $(\diamond) : Y_\alpha N_{s(\alpha)} = N_{t(\alpha)} Y_\alpha$ that $N$ have to satisfy, we can deduce a total description of $Y_\alpha$ from $Y_\alpha(x_1)$. Then, by a generic restriction $k_1 \neq 0$, there exists a dense open set $O_\alpha$ in $\mathsf{rep}(\mathcal{Q},\lambda)$ such that $Y_\alpha$ is injective.
			
			\item[•] if $\Delta(t(\alpha)) = \Delta(s(\alpha)) + 1$: by $(\circledast)$ which allows us to assume that $Z_{s(\alpha)} \cong \mathbb{K}^{b+c}$ and $Z_{t(\beta)} \cong \mathbb{K}^{c}$ with $b \geqslant 0$ and $c > 0$, and by taking $(y_1, \ldots, y_{b+c})$ and $(z_1, \ldots, z_c)$ as adapted bases of respectively $Z_{s(\alpha)}$ and $Z_{t(\alpha)}$ for $N_{s(\alpha)}$ and $N_{t(\alpha)}$, we have $Y_{\alpha}(y_1) = k'_1 z_1 + \cdots + k'_c z_c$ with $k'_1, \ldots, k'_c \in \mathbb{K}$. Again by commutative square relations $(\diamond)$ that $N$ and $Y_\alpha$ have to satisfy, we can deduce a total description of $Y_\alpha$ from $Y_\alpha(y_1)$. Then, by a generic restriction $k'_1 \neq 0$, there exists a dense open set $O_\alpha$ in $\mathsf{rep}(\mathcal{Q},\lambda)$ such that $Y_\alpha$ is surjective.
		\end{itemize}
		Moreover, for any $\alpha, \beta \in Q_1$ such that $\alpha \beta \in I$, we have for free that $Y_\alpha Y_\beta = 0$ thanks to $(\circledast)$. Thus by taking $\Omega = \bigcap_{\alpha \in Q_1} O_\alpha$,and by noting from Lemma \ref{injsurj0} that $X \in \Omega \neq \varnothing$, we get our wanted dense open set.
	\end{proof}
	\begin{lemma} \label{decomp}
		Let $m$ be a vertex of $\mathcal{Q}$ satisfying $(i)$ and $(ii)(a)$, and an integer $l > 0$. Let $Y \in \mathsf{rep}(\mathcal{Q})$ such that for any $q \in Q_0$ with $\Delta(q) > l$, we have $\dim(Y_q) = 0$, and let $\rho \in \Sigma_{\mathcal{Q}}(m)$ such that $\ell(\rho) = l$. If there exists an injective map $\iota : M(\rho) \hookrightarrow Y$, then $M(\rho)$ is a summand of $Y$. 
	\end{lemma}
	\begin{proof}[Proof]
		Let $m$, $Y$, $\rho$ and $l$ be as assumed. Let us denote $\Sigma_\mathcal{Q}(m)_{\leqslant l} = \{ \rho \in \Sigma_\mathcal{Q}(m) \mid \ell(\rho) \leqslant l \}$. By Theorem \ref{BR} and by hypothesis $(i)$ (which allows us to say that $Y$ has no band representation as a summand), we can write 
		\begin{center}
			$\displaystyle Y \cong \bigoplus_{\nu \in \Sigma_{\mathcal{Q}}(m)_{\leqslant l}} M(\nu)^{d_\nu} $
		\end{center}
		where $d_\nu \geqslant 0$ for any $\nu \in \Sigma_{\mathcal{Q}}(m)_{\leqslant l}$.
		
		First, if $l = 0$, then $Y \cong M(e_m)^d$ and the lemma becomes trivial. From now until the end of the proof, we will assume that $l > 0$.
		
		Let us suppose that there exists an injective map $\iota : M(\rho) \hookrightarrow Y$. Then there exists a string $\nu \in \Sigma_{\mathcal{Q}}(m)_{\leqslant l}$ such that $\rho$ is at the bottom of $\nu$. By definition $M(\nu)$ is a direct summand of $Y$. Let us assume that $\rho$ is a strict substring of $\nu$, i.e. there exists an arrow $\alpha$ such that $\alpha \rho$, $\alpha^{-1} \rho$, $\rho \alpha$ or $\rho \alpha^{-1}$ is a substring of $\nu$.
		
		On one hand, as $s(\rho) = m$, and by $(ii)(a)$, there is no arrow $\alpha$ such that $\rho \alpha$ or $\rho \alpha^{-1}$ is a string. On the other hand, if $\alpha^\varepsilon \rho$ is a string, then $\ell(\nu) \geqslant \ell(\rho) + 1 > l$ and so $\Delta(t(\nu)) > l$. Hence we get a contradiction between the fact that $\dim Y_{t(\nu)} = 0$ and the fact that $M(\nu)$ is a summand of $Y$.
		
		We conclude that $M(\rho)$ is a summand of $Y$.
	\end{proof}
	\begin{remark}
		Under the same conditions over $Y$, and by dual arguments, we can prove that if there exists a surjective map $\pi: Y \twoheadrightarrow M(\rho)$, then $M(\rho)$ is a summand of $Y$.
	\end{remark}
	\begin{prop} \label{1et2a}
		Let $m$ be a vertex of $\mathcal{Q}$ satisfying $(i)$ and $(ii)(a)$. Then $\mathscr{C}_{\mathcal{Q},m}$ is canonically Jordan recoverable.
	\end{prop}
	\begin{proof}[Proof]
		Let $\mathscr{C}_{\mathcal{Q},m}(l)$ for any integer $l \geqslant 0$ be the subcategory of $\mathscr{C}_{\mathcal{Q},m}$ additively generated by representations $M(\rho)$ such that $\ell(\rho) \leqslant l$. We will prove by induction on $l$ that $\mathscr{C}_{\mathcal{Q},m}(l)$ is canonically Jordan recoverable. As $\mathscr{C}_{\mathcal{Q},m} = \mathscr{C}_{\mathcal{Q},m}(l_0)$ for some $l_0 \geqslant 0$ by finiteness of the set of strings passing through $m$, we will be able to conclude.
		
		It is easy to check that $\mathscr{C}_{\mathcal{Q},m}(0) = \mathsf{add}(S_m)$ is canonically Jordan recoverable: it suffices to note that for all $Y \in \mathsf{rep}(Q)$ such that $Y_q \neq 0$ for $q \neq m$, we have $Y_\alpha = 0$ for all $\alpha \in Q_1$, as a consequence of $(i)$, and thus $Y \cong S_m^a$ for some $a \in \mathbb{N}$.
		
		Now let us assume that for a fixed $l \geqslant 0$ then $\mathscr{C}_{\mathcal{Q},m}(l)$ is canonically Jordan recoverable. Let $X \in \mathscr{C}_{\mathcal{Q},m}(l+1)$. First, if $X \in \mathscr{C}_{\mathcal{Q},m}(l)$, by induction, there exists a dense open set $O$ in $\mathsf{rep}(\mathcal{Q}, \mathsf{GenJF}(X))$ such that any representation in $O$ is isomorphic to $X$. Let us assume that $X \notin \mathscr{C}_{\mathcal{Q},m}(l)$. This implies in particular that there exists a vertex $w \in Q_0$ such that $\Delta(w) = l+1$ and $\dim X_{w} > 0$. 
		
		By the fact that $\mathscr{C}_{\mathcal{Q},m}(l+1) \subset \mathscr{C}_{\mathcal{Q},m}$, we saw that $\mathsf{GenJF}(X) = ((\dim X_q))_{q \in Q_0}$ by Lemma \ref{1*et2aGenJF} and there exists a dense open set $\Omega$ of $\mathsf{rep}(\mathcal{Q}, ((\dim X_q))_{q \in Q_0})$ in which any representation $Y \in \Omega$ we have $Y_\alpha$ injective or surjective for any $\alpha \in Q_1$ by Lemma \ref{injsurj}.
		
		Let $\rho = \alpha_{l+1}^{\varepsilon_{l+1}} \cdots \alpha_1^{\varepsilon_1}$ be the unique string of $\Sigma_\mathcal{Q}(m)$ such that $s(\rho) =m$ and $t(\rho) = w$. Let $v_0 = m$ and $v_i = t(\alpha_i^{\varepsilon_i})$ for $i \in \{1, \ldots, l+1\}$. Write $\delta = \dim X_w$. For any $i \in \{0, \ldots, l+1\}$, we consider a sequence $z_1^{(i)}, \ldots, z_\delta^{(i)}$ of vectors such that:
		\begin{itemize}
			\item[•] we choose $z_1^{(l+1)}, \ldots, z_\delta^{(l+1)}$ such that it forms a basis of $Y_w$ adapted to $N_w$.
			
			\item[•] recursively, once we have defined $z_1^{(i+1)}, \ldots, z_\delta^{(i+1)}$:
			
			\begin{itemize}
				\item[•] if $\varepsilon_{i+1} = -1$, we define $z_j^{(i)} = Y_{\alpha_{i+1}}(z_j^{(i+1)})$ for all $j \in \{1, \ldots, \delta\}$; note that for $Y \in \Omega$, $Y_{\alpha_{i+1}}$ is injective and so $z_1^{(i)}, \ldots z_{\delta}^{(i)}$ are linearly independent;
				
				\item[•] else if $\varepsilon_{i+1} = 1$, we know that for $Y \in \Omega$, $Y_{\alpha_{i+1}}$ is surjective; so there exists a $z_1^{(i)}$ such that $Y_{\alpha_{i+1}}(z_1^{(i)}) = z_1^{(i+1)}$. Hence once we choose such a $z_1^{(i)}$, we define $z_j^{(i)} = N_{v_i}^{j-1}(z_1^{(i)})$ for $j \in \{2, \ldots, \delta\}$. Again, we can note that $z_1^{(i)}, \ldots z_{\delta}^{(i)}$ are linearly independent.
			\end{itemize}
		\end{itemize} 
		We can define, for $i \in \{0, \ldots, l+1\}$, $W_{v_i} = \langle z_j^{(i)} \mid 1 \leqslant j \leqslant \delta \rangle$ and for $q \notin \mathsf{Supp}_0(\rho)$, $W_q = 0$. By construction, we can easily check that for any $\alpha \in Q_1$, $Y_\alpha(W_{s(\alpha)}) \subseteq W_{t(\alpha)}$ and more precisely, for $\alpha \in \mathsf{Supp}_1(\rho)$, $Y_\alpha (W_{s(\alpha)}) = W_{t(\alpha)}$. 
		
		Let us consider the following representation $W=((W_q)_{q \in Q_0}, (W_\alpha)_{\alpha \in Q_1})$ where $W_q$ are vector spaces as we defined previously and $W_\alpha = Y_\alpha|_{W_\alpha}$. Then, we can easily verify that $W \cong M(\rho)^\delta$ following our construction.
		
		Thanks to the $z_j^{(i)}$, we can construct an injective map $\iota : M(\rho)^\delta \hookrightarrow Y$. By applying the previous lemma, then we can write that for any $Y \in \Omega$, then $Y \cong M(\rho)^\delta \oplus Y'$ where $\dim Y'_{w} = 0$.
		By reproducing the same construction for any vertex $w$ such that $\Delta(w) = l+1$ and $\dim X_w > 0$, we have for any $Y \in \Omega$,
		\begin{center}
			$\displaystyle Y \cong Z \oplus \left( \bigoplus_{\rho \in \Sigma_Q(m) \mid \ell(\rho) = l+1} M(\rho)^{\dim X_{t(\rho)}} \right)$ 
		\end{center} 
		where $Z_q = 0$ for any $q \in Q_0$ such that $\Delta(q) > l$. In addition, we can note that $\mathsf{GenJF}(Z) = ((\dim Z_q))_{q \in Q_0}$.
		On the other side, we have:
		\begin{center}
			$\displaystyle X \cong X' \oplus \left( \bigoplus_{\rho \in \Sigma_Q(m) \mid \ell(\rho) = l+1} M(\rho)^{\dim X_{t(\rho)}} \right)$
		\end{center} 
		with $X' \in \mathscr{C}_{\mathcal{Q},m}(l)$ and thus $\mathsf{GenJF}(X') = ((\dim X_q'))_{q \in Q_0}$. We can check easily that $\mathsf{GenJF}(X') = \mathsf{GenJF}(Z)$, by the fact that $\underline{\dim}(X') = \underline{\dim}(Z)$. 
		
		By induction, as $\mathscr{C}_{\mathcal{Q},m}(l)$ is canonically Jordan recoverable, there exists a dense open set $\Omega' \subset \mathsf{rep}(\mathcal{Q}, \mathsf{GenJF}(X'))$ such that any $Z \in \Omega'$ is isomorphic to $X'$.
		
		So there exists a dense open set $O \subset \mathsf{rep}(\mathcal{Q}, \mathsf{GenJF}(X))$ such that $Y \cong X$.
		
		Hence we conclude that $\mathscr{C}_{\mathcal{Q},m}(l+1)$ is canonically Jordan recoverable, and finally $\mathscr{C}_{\mathcal{Q},m}(l)$ is canonically Jordan recoverable for any integer $l \geqslant 0$. In particular, $\mathscr{C}_{\mathcal{Q},m}$ is canonically Jordan recoverable.
	\end{proof}
	Now we will prove that these conditions are necessary.
	
	We start by showing that if $m$ does not satisfy $(i)$, then $\mathscr{C}_{\mathcal{Q},m}$ is not canonically Jordan recoverable.
	\begin{prop} \label{not1} Let $m$ be a vertex which does not satisfy $(i)$. Then $\mathscr{C}_{\mathcal{Q},m}$ is not canonically Jordan recoverable.
	\end{prop}
	\begin{proof}[Proof]
		Let $m$ be as supposed. It means we have $\rho$ and $\nu$ two distinct strings passing through $m$ such that there exists an arrow $\gamma$ satisfying $\gamma \notin \mathsf{Supp}_1(\rho) \cup \mathsf{Supp}_1(\nu)$, $s(\gamma) \in \mathsf{Supp}_0(\rho)$ and $t(\gamma) \in \mathsf{Supp}_0(\nu)$.
		
		In such a configuration, if we take the representation $X = M(\rho) \oplus M(\nu) \in \mathscr{C}_{\mathcal{Q},m}$, then $X_{s(\gamma)}$ and $X_{t(\gamma)}$ are non-zero vector spaces, and $X_\gamma$ is a zero map. 
		
		Now let us assume that $\mathscr{C}_{\mathcal{Q},m}$ is canonically Jordan recoverable. It implies that there exists an open dense set $O \subset \mathsf{rep}(\mathcal{Q}, \mathsf{GenJF}(X))$ such that any representation $Z$ in $O$ is isomorphic to $X$. Thus any representation $Y$ of $\mathsf{rep}(\mathcal{Q},\mathsf{GenJF}(X))$ satisfies $Y_\gamma = 0$.
		
		By seeking a contradiction, we will construct $W \in \mathsf{rep}(\mathcal{Q}, \mathsf{GenJF}(X))$ such that $W_\gamma \neq 0$. Let $(W_q)_{q \in Q_0}$ be a collection of vector spaces such that $W_q \cong Y_q$. For all $q \in Q_0$, consider $N_q$ be a nilpotent endomorphism of $W_q$ admitting a Jordan form given by the partition $\mathsf{GenJF}(X)^q$.
		
		Choose $u \notin \mathsf{Im}(N_{s(\gamma)})$ and $v \notin \mathsf{Im}(N_{t(\gamma)})$. Consider $\eta$ the index such that $N_{s(\gamma)}^{\eta-1}(u) \neq 0$ and $N_{t(\gamma)}^\eta(u) = 0$, and define $u_i = N^i(u)$ for all $i \in \{0, \ldots, \eta-1\}$. By noting $(u_0, \ldots, u_{\eta-1})$ is linearly independant, we can complete the family into a basis $(u_0, \ldots, u_p)$ of $W_{s(\gamma)}$. Similarly, let $k$ be the index such that $N_{t(\gamma)}^{k-1}(v) \neq 0 = N_{t(\gamma)}^k(v)$, and define $v_i = N^i(v)$ for all $i \in \{0, \ldots, \eta-1\}$. We complete the family $(v_0, \ldots,v_{k-1})$ into a basis $(v_0, \ldots, v_r)$ of $W_{t(\gamma)}$.  We can define a linear transformation $W_\gamma = W_{s(\gamma)} \longrightarrow W_{t(\gamma)}$ defined by 
		\begin{itemize}
			\item[$\bullet$] in the case where $k \geqslant i$, we put $W_\gamma(u_i) = v_{k-\eta+i}$ if $i \in \{0,\ldots,\eta-1\}$, and $W_\gamma(u_i) = 0$ otherwise;
			
			\item[$\bullet$] in the case where $k \leqslant \eta$, we put $W_\gamma(u_i) = v_i$ if $i \in \{0, \ldots, k-1\}$, and $W_\gamma(u_i) = 0$ otherwise.
		\end{itemize}
		We can define a representation $W$ as it follows:
		\begin{itemize}
			\item[$\bullet$] $W_q = Y_q$ for $q \in Q_0$;
			
			\item[$\bullet$]  $W_\alpha = 0$ for $\gamma \neq \alpha \in Q_1$, and $W_\gamma$ be as we defined it above.
		\end{itemize}
		By construction, we get that $W \in \mathsf{rep}(\mathcal{Q}, \mathsf{GenJF}(X))$.
		
		We showed above, from our assumption that $X$ was canonically Jordan recoverable, that every $Y \in \mathsf{rep}(\mathcal{Q}, \mathsf{GenJF}(X))$ has $Y_\gamma = 0$. So we have reached a contradiction.
	\end{proof}
	\begin{remark}
		We recall that if $m$ does not satisfy $(i)$ then we could still have $\mathscr{C}_{\mathcal{Q},m}$ Jordan recoverable (see Remark \ref{CJRneed1}).
	\end{remark}
	Finally, we have to prove that if $m$ does not satisfy $(ii)$, then $\mathscr{C}_{\mathcal{Q},m}$ is not canonically Jordan recoverable.
	
	Here we first need to know the existence of a certain kind of string.
	\begin{lemma} \label{stringstruct0}
		Let $m$ be a vertex of $\mathcal{Q}$ satisfying $(o)$ that  does not satisfy $(ii)$. Then there exists a string $\nu$ with $s(\nu) = m$ and three distinct arrows $\alpha, \gamma$ and $\delta $ such that $\delta \nu$, $\gamma^{-1}\nu$ and $\nu \alpha^\varepsilon$ are strings of $\mathcal{Q}$ and $\delta \gamma \in I$.
	\end{lemma}
	\begin{proof}[Proof]
		Let $m$ be as supposed. Note that none of the arrows incident to $m$ are loops by $(o)$. Since $(ii)(a)$ is not satisfied, there is no maximal string $\mu$ such that $s(\mu) = m$. By $(o)$, there must be at least one maximal string passing through $m$, and since $(ii)(b)$ is not satisfied, there must be more than one. 
		
		Let $\mu$ and $\rho$ be two such strings. As they both pass through $m$ at least once, we can divide each of them into two strings $-$ let say $\mu_1, \mu_2$ and $\rho_1, \rho_2$ $-$ such that:
		\begin{itemize}
			\item[•] $\mu = \mu_2 \mu_1^{-1}$  and $\rho = \rho_2 \rho_1^{-1}$;
			
			\item[•] $s(\rho_i) = s(\mu_i) = m$ for $i = 1,2$.
		\end{itemize}
		Note that none of these strings are lazy.
		If three of these strings start with different arrows, then by gentleness of $\mathcal{Q}$, there exists $\alpha, \delta, \gamma \in Q_1$ such that $t(\gamma) = s(\delta) = s(\alpha^\varepsilon) = m$ and $\delta \gamma \in I$. Hence, by taking $\nu = e_m$, we conclude.
		\begin{center}
			$\displaystyle 
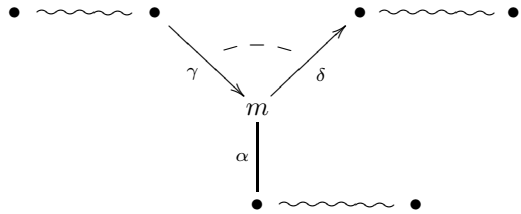
\xymatrix@R=1em @C=1em{
				\bullet \ar@{~}[rrr] & & & \bullet \ar[rrdd]_\gamma & & &  & \bullet \ar@{~}[rrr] & & & \bullet \\
				& & & & \ar@{--}@/^/[rr] & & & & & & \\
				& & & & &  m  \ar[rruu]_\delta& & & & &  \\
				& & & & & & & & & & \\
				& & & & & \bullet \ar@{-}[uu]^\alpha \ar@{~}[rrr] & & &\bullet & &
			}$
		\end{center}
		Otherwise, among those strings, there exist at least two strings, say $\mu_1$ and $\rho_1$, that leave $m$ by the same arrow. We can assume that $\mu_1 \neq \rho_1$; otherwise, we could replace $\mu_1$ and $\rho_1$ by $\mu_2$ and $\rho_2$ which must be distinct as $\rho \neq \mu$. Let $\nu$ be the common substrings of $\rho_1$ and $\mu_1$ maximal by inclusion such that $s(\nu) = m$. We can write $\rho_1 = \rho' \nu$ and $\mu_1 = \mu' \nu$. Remark that $\rho'$ and $\mu'$ are not lazy: both of them cannot be lazy as $\mu_1 \neq \rho_1$ and if for instance $\rho'$ is lazy then $\rho$ is not maximal by inclusion, as $\mu_1 \rho_2^{-1}$ is a string passing through $m$ strictly containing $\rho$ as a substring. 
		
		By gentleness of $\mathcal{Q}$, up to exchanging the roles of $\mu_1$ and $\rho_1$, we have $\rho'$ starting with $\delta \in Q_1$ and $\mu'$ starting with $\gamma \in Q_1$ such that $s(\delta) = t(\gamma) = t(\nu)$ and $\delta \gamma \in I$. So $\delta \nu$ and $\gamma^{-1} \nu$ are strings.
		\begin{center}
			$\displaystyle 
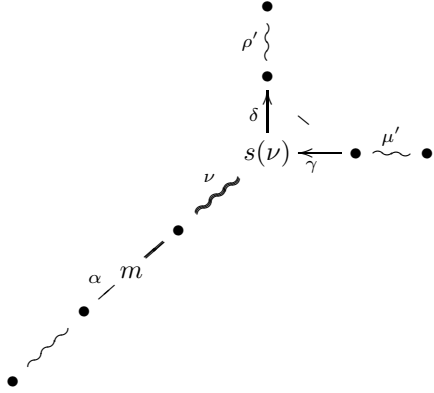
\xymatrix@R=0.5em @C=0.5em{
				& & & & & & \bullet & & & & \\
				& & & & & & &  & & & \\
				& & & & & & \bullet \ar@{~}[uu]^{\rho'} & & & &  \\
				& & & & & & \ar@{--}@/^/[rd] & & & & \\
				& & & & & & s(\nu) \ar[uu]^{\delta} & & \bullet \ar[ll]^{\gamma} \ar@{~}[rr]^{\mu'} & & \bullet \\
				& & & & & & & & & &  \\
				& & & & \bullet \ar@{~}[rruu]\ar@{~}@<0.1ex>[rruu] \ar@{~}@<0.2ex>[rruu]^\nu & & & & & & \\
				& & &  m \ar@<0.1ex>@{-}[ru] \ar@<0.2ex>@{-}[ru] \ar@{-}[ru]& & & & & & &  \\
				& & \bullet \ar@{-}[ru]^\alpha & & & & & & & & \\
				& & & & & & & & & & \\
				\bullet  \ar@{~}[rruu] & & & & & & & & & & }$
		\end{center}
		Moreover by using the fact that there is no maximal string starting at $m$, then there exists $\alpha \in Q_1$ such that $\mu_1 \alpha^\varepsilon$ is a string. So $\nu \alpha^\varepsilon$ is a string. We conclude.
	\end{proof}
	\begin{prop} \label{not2}
		Let $m$ be a vertex of $\mathcal{Q}$ that does not satisfy $(ii)$. Then $\mathscr{C}_{\mathcal{Q},m}$ is not canonically Jordan recoverable.
	\end{prop}
	\begin{proof}[Proof]
		Thanks to the Proposition \ref{not1}, we know that if $m$ does not satisfy $(i)$, then $\mathscr{C}_{\mathcal{Q},m}$ is not canonically Jordan recoverable.
		
		So suppose that $m$ satisfies $(i)$. Hence we have, from properties of $m$ and the previous lemma, there exist a string $\nu$ with $s(\nu) = m$ and three distinct arrows $\alpha, \gamma, \delta \in Q_1$ such that $\delta \nu$, $\gamma^{-1} \nu$ and $\nu \alpha^\varepsilon$ are strings and $\delta \gamma \in I$.
		
		From that fact, inspired by Example \ref{minnot3}, we can consider in $\mathscr{C}_{\mathcal{Q},m}$:
		\begin{itemize}
			\item[•] $X = M(\delta \nu) \oplus M(\nu \alpha^{\varepsilon}) \oplus M ( \gamma^{-1} \nu \alpha^{\varepsilon}) \oplus M(\gamma^{-1} \nu)$;
			
			\item[•] $Y = M(\delta \nu \alpha^{\varepsilon}) \oplus M(\nu) \oplus  M ( \gamma^{-1} \nu \alpha^{\varepsilon}) \oplus M(\gamma^{-1} \nu)$.
		\end{itemize}
		Let us suppose that $\alpha$ is an incoming arrow of $m$ ($\varepsilon = -1$). By hypothesis, we can note that the only substring of $\nu$ that is both on the top and at the bottom of $\nu$ is itself. In the following figures, we highlight the significant substrings that give us the endomorphisms of $X$ and of $Y$ that can be defined from the morphisms between their summands. We specify the morphisms by their substrings (via Proposition \ref{morph}). In addition, we recall we read a string like a sequence of compositions of functions.
		
		Here is the form of the indecomposable summands of $X$
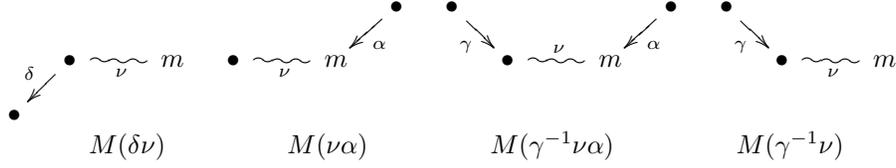
		\begin{center}
			$\displaystyle \xymatrix@R=1em @C=1em{
				& & &  &  & &  &\bullet \ar[ld]^\alpha & \bullet \ar[rd]_\gamma &   & &  & \ar[ld]^\alpha \bullet & \bullet \ar[rd]_\gamma & &    \\
				& \bullet \ar[ld]_\delta \ar@{~}[rr]_\nu & & m   & \bullet \ar@{~}[rr]_\nu & & m & & & \bullet \ar@{~}[rr]^\nu & & m  & & & \bullet \ar@{~}[rr]_\nu & & m   \\
				\bullet & & &   & & & & & & & & & & & & & 
			}$

			\hfill \hfill \hfill $M(\delta \nu)$ \hfill \hfill \hfill $M(\nu \alpha)$ \hfill \hfill \hfill $M(\gamma^{-1} \nu \alpha)$ \hfill \hfill  \hfill  $M(\gamma^{-1} \nu)$ \hfill \hfill \hfill \hfill
		\end{center}
		and here is a description of the morphisms between summands of $X$.
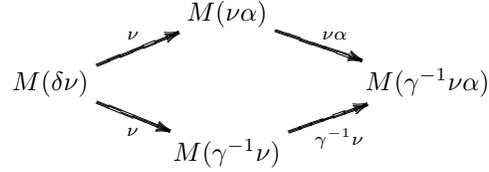
		\begin{center}
			$\displaystyle \xymatrix@R=1em @C=1em{
				& & M(\nu \alpha)  \ar@<-0.1ex>[rrd] \ar[rrd] \ar@<0.1ex>[rrd]^{\nu \alpha} & & \\
				M(\delta \nu) \ar@<-0.1ex>[rru] \ar[rru] \ar@<0.1ex>[rru]^\nu \ar@<0.1ex>[rrd] \ar[rrd] \ar@<-0.1ex>[rrd]_\nu & & & & M(\gamma^{-1} \nu \alpha)\\
				& & M(\gamma^{-1} \nu) \ar@<0.1ex>[rru] \ar[rru] \ar@<-0.1ex>[rru]_{\gamma^{-1} \nu} & &}$
		\end{center}
		Now let us see the form of the indecomposable summands of $Y$
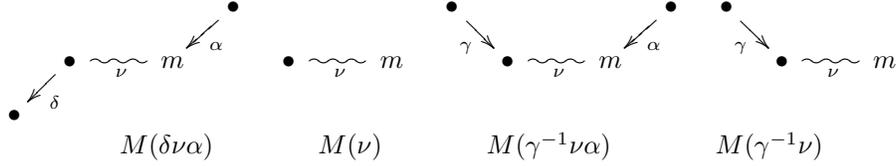
		\begin{center}
			$\displaystyle \xymatrix@R=1em @C=1em{
				& & & &\bullet \ar[ld]^\alpha &  & & & \bullet \ar[rd]_\gamma &   & &  & \ar[ld]^\alpha \bullet & \bullet \ar[rd]_\gamma & & &   \\
				& \bullet \ar[ld]^\delta \ar@{~}[rr]_\nu & & m  &   & \bullet \ar@{~}[rr]_\nu & & m &  & \bullet \ar@{~}[rr]_\nu & & m  & &  & \bullet \ar@{~}[rr]_\nu & & m  \\
				\bullet & & & &  & & & & & & & & & & & &
			}$\\
			\hfill \hfill \hfill $M(\delta \nu \alpha)$ \hfill  \hfill $M(\nu)$ \hfill \hfill  $M(\gamma^{-1} \nu \alpha)$ \hfill \hfill   $M(\gamma^{-1} \nu)$ \hfill \hfill \hfill \hfill
		\end{center}
		and a description of the morphisms between summands of $Y$.
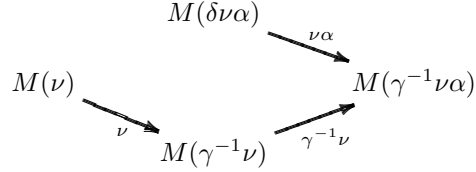
		\begin{center}
			$\displaystyle \xymatrix@R=1em @C=1em{
				& & M(\delta \nu \alpha)  \ar@<-0.1ex>[rrd] \ar[rrd] \ar@<0.1ex>[rrd]^{\nu \alpha} & & \\
				M(\nu)  \ar@<0.1ex>[rrd] \ar[rrd] \ar@<-0.1ex>[rrd]_\nu & & & & M(\gamma^{-1} \nu \alpha)\\
				& & M(\gamma^{-1} \nu) \ar@<0.1ex>[rru] \ar[rru] \ar@<-0.1ex>[rru]_{\gamma^{-1} \nu} & &}$
		\end{center}
		Thanks to these descriptions, after calculations, we get $\mathsf{GenJF}(X) = \mathsf{GenJF}(Y) = (\lambda^q)$ with:
		\begin{itemize}
			\item[•] $\lambda^q = (3,1)$ for $q \in \mathsf{Supp}(\nu)$;
			
			\item[•] $\lambda^q = (2)$ for $q = s(\beta)$ or $q = s(\gamma)$;
			
			\item[•] $\lambda^q = (1)$ for $q = t(\delta)$;
			
			\item[•] $\lambda^q = (0)$ otherwise.
		\end{itemize}
		If $\alpha$ was an outgoing arrow ($\varepsilon = 1$), a similar observation can be made and we can construct $X$ and $Y$ with the same behavior.
		
		That is why $\mathscr{C}_{\mathcal{Q},m}$ is not canonically Jordan recoverable.
	\end{proof}
	\begin{remark}\label{0andnot2}
		If we read carefully the previous proof, we can highlight we proved that if $m$ satisfies $(o)$ but not $(ii)$, then $\mathscr{C}_{\mathcal{Q},m}$ is not Jordan recoverable.
	\end{remark}
	Let us recap the proof of the main result.
	\begin{proof}[Proof of Theorem \ref{main2}]
		Thanks to Proposition \ref{1et2a} and Proposition \ref{1et2b}, we highlighted that the conditions $(i)$ and $(ii)$ are sufficient. Moreover we proved with Proposition \ref{not1} and Proposition \ref{not2} these conditions are necessary. We conclude we have the desired equivalence.
	\end{proof}
	
	\subsection{Proof of Theorem \ref{2ndmain}}
	\label{sub:proof2}
	
	From the previous work, we can deduce the following proposition.
	\begin{prop} \label{1*and2}
		Let $m$ be a vertex of $\mathcal{Q}$ such that $m$ satisfies $(i*)$ and $(ii)$. Then $\mathscr{C}_{\mathcal{Q},m}$ is Jordan recoverable.
	\end{prop}
	\begin{proof}[Proof]
		If $m$ satisfies $(i*)$ and $(ii)(b)$, we can note that the proof of Proposition \ref{1et2b} still works because we do not care about all the other arrows that could interact with the quiver $\mathcal{R}$ as these arrows cannot be reached by any string passing through $m$. As Theorem \ref{GPT} gives us the Jordan recoverability of $\mathscr{C}_{\mathcal{R},m}$, we conclude.
		
		If $m$ satisfies $(i*)$ and $(ii)(a)$, we can apply Proposition \ref{1*et2aGenJF} and recall that we get Jordan recoverability of $\mathscr{C}_{\mathcal{Q},m}$ by Corollary \ref{JFin1*and2}. Hence we conclude that $\mathscr{C}_{\mathcal{Q},m}$ is Jordan recoverable in either way.
	\end{proof}
	Like we did to prove the first theorem, now we have to show that if $m$ does not satisfy $(i*)$ or $(ii)$ then $\mathscr{C}_{\mathcal{Q},m}$ is not Jordan recoverable. 
	
	To begin with, we will show that if $m$ does not satisfy $(i*)$ then $\mathscr{C}_{\mathcal{Q},m}$ is not Jordan recoverable. To do it properly, we need to know the existence of certain sorts of strings when $(i*)$ is not satisfied by first making stronger assumptions about $m$. Then, by relaxing hypotheses over $m$, we will conclude what we want.
	
	With these ideas in mind, let first assume that $m$ is not minuscule.
	\begin{lemma} \label{stringstruct1}
		Let $m$ be a non minuscule vertex of $\mathcal{Q}$. Then there exists a string $\rho = \phi^{-1} \Gamma \phi$ with:
		\begin{itemize}
			\item[$(\alpha)$] $\phi$ is a string such that:
			\begin{itemize}
				\item[$(\alpha 1)$] $s(\phi) = m$;
				
				\item[$(\alpha 2)$] $\phi$ is passing through any vertex of $\mathcal{Q}$ at most once;
			\end{itemize}
			\item[$(\beta)$] $\Gamma$ is a non-lazy  string such that:
			\begin{itemize}
				\item[$(\beta 1)$] $s(\Gamma) = t(\Gamma) = t(\phi)$;
				
				\item[$(\beta 2)$] if $\phi$ is not lazy, then the first and the last arrow of $\Gamma$ are in relation;
				
				\item[$(\beta 3)$] $\Gamma$ is passing through $t(\phi)$ exactly twice; 
				
				\item[$(\beta 4)$] $\Gamma$ is passing through any vertex $q \neq t(\phi)$ at most once;
				
				\item[$(\beta 5)$] $\mathsf{Supp}_0(\phi) \cap \mathsf{Supp}_0(\Gamma) = \{t(\phi)\}$.
			\end{itemize}
		\end{itemize}
	\end{lemma}
	\begin{proof}[Proof]
		Let $m$ be as assumed. By Proposition \ref{minstring}, there exists a string $\Xi$ passing through $m$ at least twice. Among all the substrings of $\Xi$, there exists a string $\mu$ such that $s(\mu) = t(\mu) = m$ and $\mu$ is passing exactly twice through $m$. Thus the set $S(m)$ of strings starting and ending at $m$, and passing through $m$ exactly twice is not empty. Let us consider $\chi \in S(m)$ to be a string minimal with respect to the ordering by length. Note that $\chi$ is not lazy. We will prove that from $\chi$ we can extract a string $\rho$ as claimed.
		
		Let us write $\chi = \alpha_k^{\varepsilon_k} \cdots \alpha_1^{\varepsilon_1}$. Consider $v_0 = m$ and $v_i = t(\alpha_i^{\varepsilon_i})$ for any $i \in \{1, \ldots, k\}$. Let:
		\begin{itemize}
			\item[•] $p = \min(i \in \{1, \ldots, k\} \mid v_i \in \{v_0, \ldots, v_{i-1} \})$;
			
			\item[•] $j$ be the only integer in $\{0, \ldots, p-1\}$ such that $v_j = v_p$.
		\end{itemize}
		Note that $p$ is well-defined as $\chi$ passes at least twice through $m$. Consider $\phi = \alpha_j^{\varepsilon_j} \cdots \alpha_1^{\varepsilon_1} $ and $\Gamma = \alpha_p^{\varepsilon_p} \cdots \alpha_{j+1}^{\varepsilon_{j+1}}$. 
		\begin{center}
			$\displaystyle \xymatrix@R=1em @C=2em{
				& & & & & &  \bullet \ar@{~}@/^2pc/[dddd] \\
				& & \phi & & & &  \\
				m & & \ar@{~}@/_/ '[rr] '[ll]  & & \bullet \ar@{-}[rruu]^{\alpha_{p}^{\varepsilon_{p}}} \ar@{-}[rrdd]_{\alpha_{j+1}^{\varepsilon_{j+1}}} & & \Gamma  \\
				& & & & & &  \\
				& & & & & & \bullet }$
		\end{center}
		Now we will check that $\phi$ and $\Gamma$ satisfy the claimed conditions. 
		
		First let us show that $\rho = \phi^{-1} \Gamma \phi$ is a string. If $\phi$ is lazy, then we can conclude clearly. Assume that $\phi$ is not lazy. Note that $p < k$. At the vertex $v_j= v_p$, if $p > j+1$, we have at least three distinct incident arrows $\alpha_j$, $\alpha_{j+1}$ and $\alpha_p$. By gentleness of $\mathcal{Q}$, this means that exactly one among those following words $\alpha_{j+1}^{\varepsilon_{j+1}} \alpha_j^{\varepsilon_j}$, $\alpha_p^{\varepsilon_p} \alpha_j^{\varepsilon_j}$ or $\alpha_{j+1}^{\varepsilon_{j+1}} \alpha_p^{\varepsilon_p}$ is not a string. 
		
		Obviously $\alpha_{j+1}^{\varepsilon_{j+1}} \alpha_j^{\varepsilon_j}$ is a string by construction, as $\chi$ is a string. If $\alpha_p$ and $ \alpha_j$ are in relation at $v_j$, then, by gentleness of $\mathcal{Q}$, $\alpha_{p+1}$ and $\alpha_j$ are not in relation at $v_j$ and so $\chi' = \alpha_k^{\varepsilon_k} \cdots \alpha_{p+1}^{\varepsilon_{p+1}} \alpha_j^{\varepsilon_j} \cdots \alpha_1^{\varepsilon_1}$ is a string of $S(m)$ and $\ell(\chi') < \ell(\chi)$. This is a contradiction with the minimality of $\chi$.
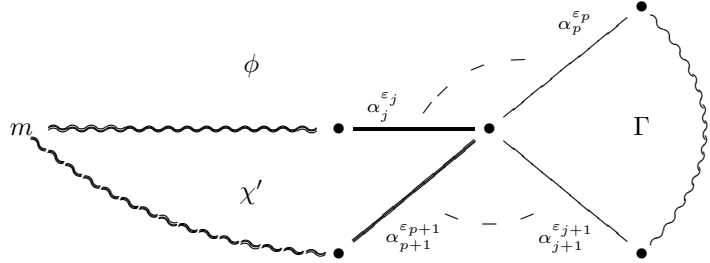
		\begin{center}
			$\displaystyle \xymatrix@R=1em @C=2em{
				& & & & & & & &  \bullet \ar@{~}@/^2pc/[dddd] \\
				& & & \phi & & & & &  \\
				m \ar@{~}[rrrr] \ar@{~}@<0.1ex>[rrrr] \ar@{~}@<-0.1ex>[rrrr] & &  & & \bullet \ar@{-}@<0.1ex>[rr] \ar@{-}@<-0.1ex>[rr] \ar@{-}[rr]^(.3){\alpha_j^{\varepsilon_j}} & \ar@{--}@/^1pc/[rru] & \bullet \ar@{-}@<0.1ex>[ddll] \ar@{-}@<-0.1ex>[ddll] \ar@{-}[ddll]^(.7){\alpha_{p+1}^{\varepsilon_{p+1}}} \ar@{-}[rruu]^(.7){\alpha_{p}^{\varepsilon_{p}}} \ar@{-}[rrdd]_(.7){\alpha_{j+1}^{\varepsilon_{j+1}}} & & \Gamma \\
				& & & \mathbf{\chi'} & & \ar@{--}@/_1pc/[rr] & & & \\
				& & & &  \bullet \ar@{~}@<0.1ex>@/^1pc/[lllluu] \ar@{~}@<-0.1ex>@/^1pc/[lllluu] \ar@{~}@/^1pc/[lllluu]& & & & \bullet }$
		\end{center}
		Hence $\alpha_{j+1}$ and $\alpha_p$ are in relation at $v_j$ and $\rho = \phi^{-1} \Gamma \phi$ is a string. Note that we come to the same conclusion if $p = j+1$. Finally we can easily check by our construction that $\phi$ and $\Gamma$ satisfy the conditions we claimed.
		\begin{center}
			$\displaystyle 
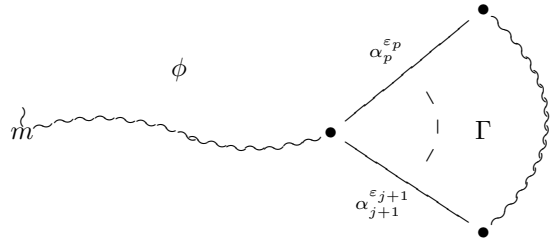
\xymatrix@R=1em @C=2em{
				& & & & & &  \bullet \ar@{~}@/^2pc/[dddd] \\
				& & \phi & & & &  \\
				m & & \ar@{~}@/_/ '[rr] '[ll]  & & \bullet \ar@{-}[rruu]^{\alpha_{p}^{\varepsilon_{p}}} \ar@{-}[rrdd]_{\alpha_{j+1}^{\varepsilon_{j+1}}} & & \Gamma  \\
				& & & & & \ar@{--}@/_1pc/[uu] &  \\
				& & & & & & \bullet }$
		\end{center}
		\hfill
	\end{proof}
	\begin{prop} \label{nonminresult}
		Let $m$ be a vertex of $\mathcal{Q}$ which is not minuscule. Then $\mathscr{C}_{\mathcal{Q},m}$ is not Jordan recoverable.
	\end{prop}
	\begin{proof}[Proof]
		Let $m$ be as supposed. Thanks to the previous lemma, there exists a string $\rho = \phi^{-1} \Gamma \phi$ satisfying the conditions stated in the previous lemma.
		
		If $\ell(\Gamma) = 1$, let $X = M(\phi)^2$ and $Y = M(\rho)$ two representations of $\mathscr{C}_{\mathcal{Q},m}$. By the fact that $Y$ admits an endomorphism defined by the substring $\phi$ (Proposition \ref{morph}), we can clearly get $\mathsf{GenJF}(X) = \mathsf{GenJF}(Y) = (\lambda^q)_{q \in Q_0}$ such that:
		\begin{itemize}
			\item[•] $\lambda^q = (2)$ for $q \in \mathsf{Supp}(\rho)$;
			
			\item[•] $\lambda^q = (0)$ otherwise.
		\end{itemize}
		Thus $\mathscr{C}_{\mathcal{Q},m}$ is not Jordan recoverable.
		
		Otherwise, inspired by Example \ref{minnot1}, recalling that we wrote $\Gamma = \beta_p^{\xi_p} \cdots \beta_1^{\xi_1}$, let $X = M(\phi^{-1} \beta_p^{\xi_p} \cdots \beta_2^{\xi_2})$ and $Y = M(\beta_{p-1}^{\xi_{p-1}} \cdots \beta_1^{\xi_1} \phi)$ be two representations of $\mathscr{C}_{\mathcal{Q},m}$. In this case, we get $\mathsf{GenJF}(X) = \mathsf{GenJF}(Y) = (\lambda^q)_{q \in Q_0}$ such that:
		\begin{itemize}
			\item[•] $\lambda^q = (1)$ for $q \in \mathsf{Supp}(\rho)$;
			
			\item[•] $\lambda^q = (0)$ otherwise.
		\end{itemize}
		Again, we conclude that $\mathscr{C}_{\mathcal{Q},m}$ is not Jordan recoverable.
		
		Hence in each case, we get what we want.
	\end{proof}
	Now we can prove that for any vertex $m$ which does not satisfy $(o)$, $\mathscr{C}_{\mathcal{Q},m}$ is not Jordan recoverable.
	
	We need to prove beforehand the following lemma.
	\begin{lemma}\label{stringstruct2}
		Let $m$ be a minuscule vertex of $\mathcal{Q}$ which is not satisfying $(o)$. Then at least one of the following assertions is true:
		\begin{itemize}
			\item[$(\alpha)$] There exist two non-lazy strings $\nu$ and $\mu$ such that:
			\begin{itemize}
				\item[$(\alpha 1)$] $s(\nu) = s(\mu) = m$ and $t(\nu) = t(\mu) \neq m$;
				
				\item[$(\alpha 2)$] $\mu \nu^{-1}$ is a string;
				
				\item[$(\alpha 3)$] The last arrows of $\nu$ and $\mu$ are in relation;
				
				\item[$(\alpha 4)$] $\nu$ and $\mu$ are strings that are passing through any vertex of $\mathcal{Q}$ at most once;
				
				\item[$(\alpha 5)$] $\mathsf{Supp}_0(\nu) \cap \mathsf{Supp}_0(\mu) = \{m, t(\nu)\}$.
			\end{itemize}
			
			\item[$(\beta)$] There exist a non-lazy string $\phi$ and a band $\Gamma$ such that:
			\begin{itemize}
				\item[$(\beta 1)$] $s(\phi) = m$ and $t(\phi) = s(\Gamma) = t(\Gamma) \neq m$;
				
				\item[$(\beta 2)$] $\Gamma \phi$ is a string;
				
				\item[$(\beta 3)$] the last arrows of $\Gamma$ and $\phi$ are in relation;
				
				\item[$(\beta 4)$] $\phi$ is a string that is passing through any vertex of $\mathcal{Q}$ at most once;
				
				\item[$(\beta 5)$] $\Gamma$ is a string which is passing through any vertex $q \neq t(\phi)$ of $\mathcal{Q}$ at most once and is passing through $t(\phi)$ exactly twice;
				
				\item[$(\beta 6)$] $\mathsf{Supp}_0(\phi) \cap \mathsf{Supp}_0
				(\Gamma) = \{t(\phi)\}$.
			\end{itemize}
		\end{itemize}
	\end{lemma}
	\begin{proof}[Proof]
		Let $m$ be as supposed. Then there exists a string $\Xi$ passing through $m$ and passing through another vertex at least twice. Let us consider $\chi$ a minimal substring of $\Xi$ by inclusion such that $\chi$ is passing through $m$ and there exists another vertex $-$ let us say $q \neq m$ $-$ such that $\chi$ is passing through it at least twice.
		
		First, it is easy to check that $\chi$ passes through $q$ exactly twice, and $s(\chi) = q$ or $t(\chi) = q$. Then we have two cases.
		\begin{itemize}
			\item[$(\alpha)$] If $s(\chi) = t(\chi) = q$, then we can write $\chi = \mu \nu^{-1}$ where $\nu$ and $\mu$ are the two distinct non-lazy strings such that $s(\nu) = s(\mu) = m$ and $t(\nu) = t(\mu) = q$. We claim that $\nu$ and $\mu$ are satisfying what we want.
		\end{itemize}
		\begin{center}
			$\displaystyle 
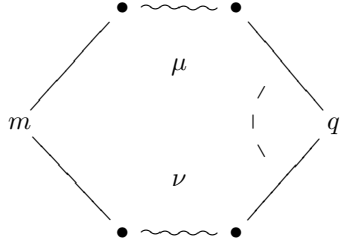
\xymatrix@R=1em @C=1em{
				& & \bullet \ar@{~}[rr] & & \bullet \ar@{-}[rrdd]& & \\
				& & & \mu & & & \\
				m \ar@{-}[rruu] \ar@{-}[rrdd] & & & & & & q \\
				& & & \nu & & \ar@{--}@/^1pc/[uu] & \\
				& & \bullet \ar@{~}[rr]& & \bullet \ar@{-}[rruu] & & }$
		\end{center}
		\begin{itemize}
			\item[]
			\begin{itemize}
				\item[$(\alpha 1)$] $s(\nu) = s(\mu) = m$ and $t(\nu) = t(\mu) = q \neq m$ for free;
				
				\item[$(\alpha 2)$] $\chi = \mu \nu^{-1}$ is a string for free;
				
				\item[$(\alpha 3)$] if the last arrows of $\nu$ and $\mu$ are not in relation, then $\nu^{-1} \mu$ is a non-lazy string starting and ending at $m$; this is a contradiction with the fact that $m$ is minuscule. So the last arrows of $\nu$ and $\mu$ are in relation;
				
				\item[$(\alpha 4)$] if $\nu$ is passing through a vertex $v$ at least twice, as $\nu$ is a substring of $\chi$, we get a contradiction with the minimality of $\chi$. As the same argument works for $\mu$, then $\nu$ and $\mu$ are passing through any vertex at most once;
				
				\item[$(\alpha 5)$] if there exists $v \in \mathsf{Supp}_0(\nu) \cap \mathsf{Supp}_0(\mu)$ and $v \neq q,m$, we could consider $\nu_0$ a substring of $\nu$ and $\mu_0$ a substring of $\mu$ such that $s(\nu_0)= s(\mu_0) = m$ and $t(\nu_0) = t(\mu_0) = v$. Thus $\mu_0 \nu_0^{-1}$ is a strict substring of $\chi$ and that is a contradiction with the minimality of $\chi$. 
			\end{itemize}
			
			\item[$(\beta)$] Otherwise, up to inversing $\chi$, we can consider that $s(\chi) = m$ and $t(\chi) = q$. Let $\phi$ be the minimal substring of $\chi$ such that $s(\phi) = m$ and which is passing through $q$. Let $\Gamma$ be the string such that $\chi = \Gamma \phi$. We claim that $\phi$ and $\Gamma$ are satisfying what we want.
		\end{itemize}
		\begin{center}
			$\displaystyle 
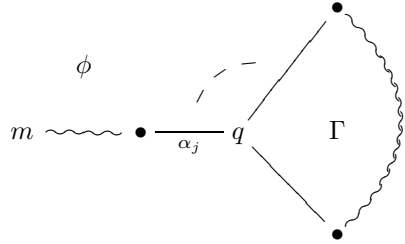
\xymatrix@R=1em @C=1em {
				&  & & & & & \bullet\\
				&  \phi & & & & & \\
				m \ar@{~}[rr] & &  \bullet \ar@{-}[rr]_{\alpha_j} & \ar@{--}@/^1pc/[rru] & q \ar@{-}[rruu] \ar@{-}[rrdd] & & \Gamma \\
				& & & & & & \\
				& &  & & &&   \bullet \ar@{~}@/_2pc/[uuuu]}$
		\end{center}
		\begin{itemize}
			\item[]
			\begin{itemize}
				\item[$(\beta 1)$] For free, we have $s(\phi) = m$ and $t(\phi) = s(\Gamma) = t(\Gamma) = q \neq m$;
				
				\item[$(\beta 2)$] $\mu' = \Gamma \phi$ for free;
				
				\item[$(\beta 3)$] if the last arrows of $\Gamma$ and $\phi$ are not in relation at $q$, then $\phi^{-1} \Gamma \phi$ is a string which starts and ends at $m$. This is a contradiction with the fact that $m$ is minuscule. Thus the last arrows of $\Gamma$ and $\phi$ are in relation. Furthermore we can prove that $\Gamma$ is a band:
				\begin{itemize}
					\item[•] $s(\Gamma) = t(\Gamma)$;
					
					\item[•] As the last arrows of $\Gamma$ and $\phi$ are in relation at $t(\Gamma)$ and the first arrow of $\Gamma$ is not in relation with the last arrow of $\phi$ at $t(\Gamma)$, by the gentleness of $\mathcal{Q}$, the last and the first arrow of $\Gamma$ are not in relation. Then $\Gamma^j$ is a string;
					
					\item[•] If $\Gamma$ is not primitive, we can write $\Gamma = \Upsilon^j$ for some $j > 1$. So $\Upsilon \phi$ is a substring of $\chi$ which passes through a vertex $t(\Gamma) = s(\Upsilon) = t(\Upsilon)$ at least twice. This is a contradiction with the minimality of $\chi$. So $\Gamma$ is primitive.
				\end{itemize}
				
				\item[$(\beta 4)$] $\phi$ cannot pass through a vertex at least twice as $\phi$ is a strict substring of $\chi$;
				
				\item[$(\beta 5)$] if $\Gamma$ passes through another vertex $v \neq q$ at least twice, let $\Gamma_0$ be the substring of $\Gamma$ obtained by deleting its last arrow. Then $\Gamma_0 \phi$ is a substring of $\chi$ which is passing through $v$ at least twice. This is a contradiction.
				
				If $\Gamma$ passes through $q$ at least thrice, then $\Gamma_0$ passes through $q$ at least twice and we conclude again with a contradiction. Consequently $\Gamma$ is a string passing through any vertex $v \neq t(\phi)$ at most once and passing through $t(\phi)$ exactly twice;
				
				\item[$(\beta 6)$] Trivially, $\{ t(\phi) \} \subseteq \mathsf{Supp}_0(\phi) \cap \mathsf{Supp}_0(\Gamma)$. Let us assume that there exists a vertex $v \neq t(\phi)$ such that $v \in \mathsf{Supp}_0(\phi) \cap \mathsf{Supp}_0(\Gamma)$. Let $\Gamma_1$ be a substring of $\Gamma$ starting with the first arrow of $\Gamma$ and such that $t(\Gamma_1) = v$. Then $\Gamma_1 \phi$ is a substring of $\chi$ which is passing through a vertex at least twice. This is a contradiction again with the minimality of $\chi$.
			\end{itemize}
		\end{itemize}
		Either way, we conclude.
	\end{proof}
	\begin{prop} \label{not0}
		Let $m$ be a vertex of $\mathcal{Q}$ which does not satisfy $(o)$. Then $\mathscr{C}_{\mathcal{Q},m}$ is not Jordan recoverable.
	\end{prop}
	\begin{proof}[Proof] If $m$ is not minuscule, we can conclude from the Proposition \ref{nonminresult}.
		
		Let us suppose that $m$ is minuscule. By the previous lemma, we can construct a certain kind of string, and we have to distinguish the two cases.
		
		If $(\alpha)$ is true, then let us write $\mu = \alpha_k^{\varepsilon_k} \cdots \alpha_1^{\varepsilon_1}$ and $\nu = \beta_p^{\xi p} \cdots \beta_1^{\xi_1}$. Consider $X = M(\alpha_{k-1}^{\varepsilon_{k-1}} \cdots \alpha_1^{\varepsilon_1} \nu^{-1})$ and $Y = M(\mu (\beta_{p-1}^{\xi_{p-1}} \cdots \beta_1^{\xi_1})^{-1})$. We can easily check that $\mathsf{GenJF}(X) = \mathsf{GenJF}(Y) = (\lambda^q)_{q \in Q_0}$ with:
		\begin{itemize}
			\item[•] $\lambda^q = (1)$ for any $q \in \mathsf{Supp}_0(\mu) \cup \mathsf{Supp}_0(\nu)$;
			
			\item[•] $\lambda^q = (0)$ otherwise.
		\end{itemize}
		Thus $\mathscr{C}_{\mathcal{Q},m}$ is not Jordan recoverable.
		If $(\beta)$ is true, then let $\Gamma'$ be the string obtained from $\Gamma$ by deleting the last arrow of $\Gamma$ $-$ called $\delta$ in the drawing below $-$ and $\phi'$ be the string obtained for $\phi$ by deleting the last arrow $-$ called $\beta$. Note that by $(\beta 3)$ $\beta$ and $\delta$ are in relation at $t(\phi)$. If $\beta$ is ingoing to $t(\phi)$ then $\delta$ is outgoing from $t(\phi)$ and vice-versa. Inspired by Example \ref{minnot2}, we can consider $X = M(\Gamma' \phi)^2$ and $Y = M(\phi') \oplus M(\Gamma'\Gamma\phi)$. 
		It is clear enough how the endomorphisms of $X$ go. Let us see the morphisms between the summands of $Y$. We label the morphisms by substrings as an application of Theorem \ref{morph}.
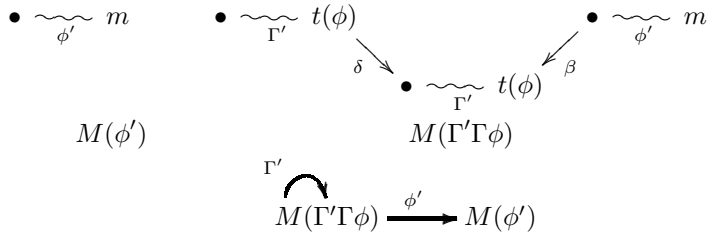
		\begin{center}
			$\displaystyle \xymatrix@R=1em @C=1em{
				\bullet \ar@{~}[rr]_{\phi'} & & m & & \bullet \ar@{~}[rr]_{\Gamma'} & & t(\phi) \ar[rd]_\delta & & & & \bullet \ar[ld]^\beta \ar@{~}[rr]_{\phi'} & & m \\
				& & & & & & &  \bullet \ar@{~}[rr]_{\Gamma'} & & t(\phi) & & & & & }$\\
			\hfill \hfill \hfill  $\displaystyle M(\phi')$ \hfill \hfill \hfill \hfill \hfill $\displaystyle M(\Gamma' \Gamma \phi)$ \hfill \hfill \hfill \hfill \hfill \hfill \hfill \hfill \hfill\\
			$\displaystyle \xymatrix@R=1em @C=1em{
				M(\Gamma' \Gamma \phi) \ar@<-0.1ex>@(l,u) \ar@(l,u) \ar@<0.1ex>@(l,u)^{\Gamma'} \ar@<-0.1ex>[rr] \ar[rr] \ar@<0.1ex>[rr]^{\phi'} & & M(\phi')
			}$
		\end{center}
		From that, by simple calculation, we get $\displaystyle \mathsf{GenJF}(X) = \mathsf{GenJF}(Y) = (\lambda^q)_{q \in Q_0}$ where:
		\begin{itemize}
			\item[•] $\lambda^q = (2)$ for all $q \in \mathsf{Supp}(\Gamma' \phi)$;
			
			\item[•] $\lambda^q = (0)$ otherwise.
		\end{itemize} 
		In other cases (for example if $\beta$ is an outgoing arrow of $m'$, or if $\delta$ is an incoming arrow of $m'$), we can make similar observations, construct representations $X$ and $Y$ analogously, and we will obtain the same result.
		
		That is why $\mathscr{C}_{\mathcal{Q},m}$ is not Jordan recoverable.
	\end{proof}
	Now let us show another translation of the condition $(i*)$. By the previous proposition, we can suppose that $m$ satisfies $(o)$.
	\begin{lemma} \label{stringstruct3}
		Let $m$ be a vertex of $\mathcal{Q}$ satisfying $(o)$ but not $(i*)$. Then there exist three strings $\phi$, $\nu$ and $\mu$ such that:
		\begin{itemize}
			\item[$(\gamma 1)$] $\nu$ and $\mu$ are non-lazy;
			
			\item[$(\gamma 2)$] $s(\phi) = m$, $t(\phi) = s(\nu) = s(\mu)$ and $t(\nu) = t(\mu)$;
			
			\item[$(\gamma 3)$] $\nu \phi$ and $\mu \phi$ are strings of $\mathcal{Q}$;
			
			\item[$(\gamma 4)$] $\phi$, $\nu$ and $\mu$ are passing through any vertex of $\mathcal{Q}$ at most once;
			
			\item[$(\gamma 5)$] the first arrows and the last arrows of $\nu$ and $\mu$ are in relation at $s(\nu)$ and $t(\mu)$ respectively;
			
			\item[$(\gamma 6)$] $\mathsf{Supp}_0(\phi) \cap \mathsf{Supp}_0(\nu) = \mathsf{Supp}_0(\phi) \cap \mathsf{Supp}_0(\mu) = \{s(\nu)\}$ and $\mathsf{Supp}_0(\nu) \cap \mathsf{Supp}_0(\mu) = \{s(\nu), t(\nu)\}$;
		\end{itemize}
	\end{lemma}
	\begin{proof}[Proof]
		Let $m$ be as supposed. By the fact that $m$ does not satisfy $(i*)$, there exist two distinct strings $\Xi$ and $\Psi$ such that $s(\Xi) =s(\Psi)=m$ and $t(\Psi) = t(\Xi)$. Let us consider such a pair $(\chi, \psi)$ to be minimal in the following way: for any pair of distinct strings $(\sigma, \tau)$ such that $s(\sigma) = s(\tau) = m$ and $t(\sigma) = t(\tau)$, we have $\ell(\sigma) \geqslant \ell(\chi)$ and if $\ell(\sigma) = \ell(\chi)$ then $\ell(\tau) \geqslant \ell(\psi)$. Consequently $\ell(\chi) \leqslant \ell(\psi)$.
		
		Let $\phi$ be the common string maximal by inclusion of $\chi$ and $\psi$ such that $s(\phi) = m$. Let $\nu$ and $\mu$ be the strings such that $\chi = \nu \phi$ and $\psi = \mu \phi$. We will show that $\phi$, $\nu$ and $\mu$ are three strings as claimed.
		\begin{center}
			$\displaystyle 
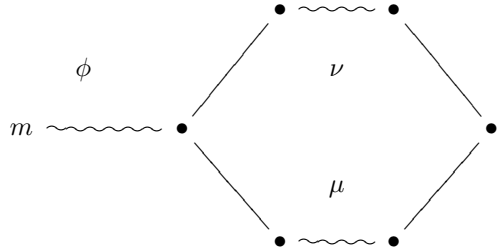
\xymatrix@R=1em @C=1em{
				& & & & & \bullet \ar@{~}[rr]& &  \bullet \ar@{-}[rrdd] & & \\
				& \phi &  & & & & \nu & & & \\
				m \ar@{~}[rrr] & & & \bullet \ar@{-}[rruu] \ar@{-}[rrdd] & & & & & & \bullet \\
				& & & & & & \mu & & &\\
				& & & & & \bullet \ar@{~}[rr] & & \bullet \ar@{-}[rruu]& & \\ }$
		\end{center}
		\begin{itemize}
			\item[$(\gamma 1)$] As $\ell(\psi) \geqslant \ell(\chi)$, if $\mu$ is lazy then $\psi = \chi$ which is a contradiction with the fact that $\psi$ and $\chi$ are distinct. If $\nu$ is lazy, then $\psi = \mu \phi$ is a string passing through $s(\mu)$ at least twice. This is a contradiction with the fact that $m$ satisfies $(o)$. So $\nu$ and $\mu$ are non-lazy.
			
			\item[$(\gamma 2)$] By construction we have $s(\phi) = m$, $t(\phi) = s(\nu) = s(\mu)$ and $t(\nu) = t(\mu)$ for free;
			
			\item[$(\gamma 3)$] By construction $\chi = \nu \phi$ and $\psi = \mu \phi$ are strings of $\mathcal{Q}$;
			
			\item[$(\gamma 4)$] If there exists a vertex $q$ of $\mathcal{Q}$ such that one of the three strings $\phi$, $\nu$ or $\mu$ is passing through $q$ at least twice, then $\chi$ or $\psi$ is passing through $q$ at least twice. This is a contradiction with the fact that $m$ satisfies $(o)$. So $\phi$, $\nu$, and $\mu$ are passing through any vertex of $\mathcal{Q}$ at most once;
			
			\item[$(\gamma 5)$] If the last arrows of $\nu$ and $\mu$ are not in relation at $t(\nu)$, then $\mu^{-1} \nu \phi$ is a string passing through $m$ and through $s(\nu)$ at least twice. This is a contradiction with the fact that $m$ satisfies $(o)$. By noting that $\nu \phi$ and  $\mu \phi$ are strings of $\mathcal{Q}$, and the first arrows of $\nu$ and $\mu$ have to be distinct by maximality of $\phi$, if $\phi$ is non-lazy then the last arrow of $\phi$ is in relation with neither the first arrow of $\nu$ nor the first arrow of $\mu$ at $t(\phi)$; by gentleness of $\mathcal{Q}$ the first arrows of $\nu$ and $\mu$ have to be in relation at $s(\nu)$. 
			
			If $\phi$ is lazy and if the first arrows are not in relation at $s(\nu) = m$, then $\mu \nu^{-1}$ is a well-defined string of $\mathcal{Q}$ passing through $m$ and through $t(\nu)$ at least twice. This is a contradiction with the fact that $m$ satisfies $(o)$. Hence we conclude that the first arrows and the last arrows of $\nu$ and $\mu$ are in relation at $s(\nu)$ and $t(\nu)$ respectively; 
			
			\item[$(\gamma 6)$] It is clear that $\{s(\nu)\} \subseteq \mathsf{Supp}_0(\phi) \cap \mathsf{Supp}_0(\nu) $. If there exists a vertex $q \neq s(\nu)$ in $\mathsf{Supp}_0(\phi) \cap \mathsf{Supp}_0(\nu)$, then $\chi = \nu \phi$ is a string passing through $q$ at least twice. This is a contradiction with the fact that $m$ satisfies $(o)$. So $\mathsf{Supp}_0(\phi) \cap \mathsf{Supp}_0(\nu)  = \{s(\nu)\}$. Analogously we have $\mathsf{Supp}_0(\phi) \cap \mathsf{Supp}_0(\mu)  = \{s(\nu)\}$. 
			
			It is also clear that $\{s(\nu), t(\nu)\} \subseteq \mathsf{Supp}_0(\nu) \cap \mathsf{Supp}_0(\mu)$. Assume that there exists a vertex $t(\nu) \neq q \neq s(\nu)$ in $\mathsf{Supp}_0(\nu) \cap \mathsf{Supp}_0(\mu)$. Thanks to $(\gamma 4)$, we can consider $\nu_q$ (respectively $\mu_q$) the substring of $\nu$ (respectively $\mu$) starting at $s(\nu)$ and ending at $q$. Then $(\chi_q = \nu_q \phi, \psi_q = \mu_q \phi)$ is a pair of distinct strings such that $s(\chi_q) = s(\psi_q)$ and $t(\chi_q) = t(\psi_q) = q$. But $\ell(\chi_q) < \ell(\chi)$ by construction, which is a contradiction with the minimality of $(\chi, \psi)$.
			
			Hence we conclude that $\mathsf{Supp}_0(\nu) \cap \mathsf{Supp}_0(\mu) = \{s(\nu), t(\nu)\}$.\qedhere
		\end{itemize}
		\vspace*{-0.43cm}
	\end{proof}
	\begin{prop} \label{not1*}
		Let $m$ be a vertex of $\mathcal{Q}$ which does not satisfy $(i*)$. Then $\mathscr{C}_{\mathcal{Q},m}$ is not Jordan recoverable.
	\end{prop}
	\begin{proof}[Proof]
		Let $m$ be as supposed. We already know that if $m$ does not satisfy $(o)$ then $\mathscr{C}_{\mathcal{Q},m}$ is not Jordan recoverable thanks to Proposition \ref{not0}.
		
		Let us assume that $m$ does satisfy $(o)$. By Lemma \ref{stringstruct3}, there exist three strings $\phi$, $\nu$, and $\mu$ satisfying the conditions stated in the previous lemma.
		
		Without loss of generality, we can assert that the first arrow $\alpha$ of $\nu$ is ingoing to $s(\nu)$ and the first arrow $\beta$ of $\mu$ is therefore outgoing from $s(\nu)$. Define $\nu'$ (respectively $\mu'$) the string obtained from $\nu$ (respectively $\mu$) by deleting its last arrow.
		
		Consider $X = M(\nu' \phi) \oplus M(\mu \phi)$ and $Y =  M(\nu \phi) \oplus M(\mu' \phi)$. We can see that the only significant substring that is on the top of $\mu \phi$ (respectively $\mu' \phi$) and at the bottom of $\nu' \phi$ (respectively $\nu \phi$) is $\phi$. In the following figures, we highlighted the morphisms between summands of $X$ and $Y$ like previously. 
		
		Here is what we have for $X$.
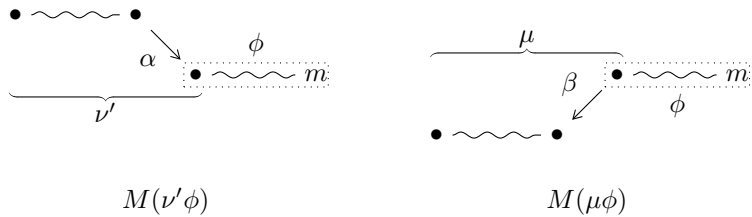
		\begin{center}
			\begin{tikzpicture}[>= angle 60, scale=0.8]
				\node (1) at (0,1){$\bullet$};
				\node (2) at (2,1){$\bullet$};
				\node (3) at (3,0){$\bullet$};
				\node (4) at (5,0){$m$};
				\draw[->] (2) -- node[below left]{$\alpha$} (3);
				\draw[-,decorate, decoration={snake,amplitude=.4mm}] (1) -- (2);
				\draw[-,decorate,decoration={snake,amplitude=.4mm}] (3) -- (4);
				\draw[decorate,decoration={calligraphic brace, amplitude=3pt,mirror}]
				(-.1, -.3) -- (3.1,-.3) node[midway,auto,swap]{$\nu'$};
				\draw[dotted] (2.8,-.2) rectangle (5.2,.2); 
				\node at (4,0.5){$\phi$};
				\node at (2.5,-2.1){$M(\nu' \phi)$};
				\begin{scope}[xshift = 7cm]
					\node (1) at (0,-1){$\bullet$};
					\node (2) at (2,-1){$\bullet$};
					\node (3) at (3,0){$\bullet$};
					\node (4) at (5,0){$m$};
					\draw[->] (3) -- node[above left]{$\beta$} (2);
					\draw[-,decorate, decoration={snake,amplitude=.4mm}] (1) -- (2);
					\draw[-,decorate,decoration={snake,amplitude=.4mm}] (3) -- (4);
					\draw[decorate,decoration={calligraphic brace, amplitude=3pt}]
					(-.1, .3) -- (3.1,.3) node[midway,auto]{$\mu$};
					\draw[dotted] (2.8,-.2) rectangle (5.2,.2); 
					\node at (4,-0.5){$\phi$};
					\node at (2.5,-2.1){$M(\mu \phi)$};
				\end{scope}
			\end{tikzpicture}
			
			$\displaystyle \xymatrix@R=1em @C=1em{
				M(\mu \phi) \ar@<0.1ex>[rr]^\phi \ar[rr] \ar@<-0.1ex>[rr]& &M(\nu' \phi)}$
		\end{center}
		And here is what we have for $Y$.
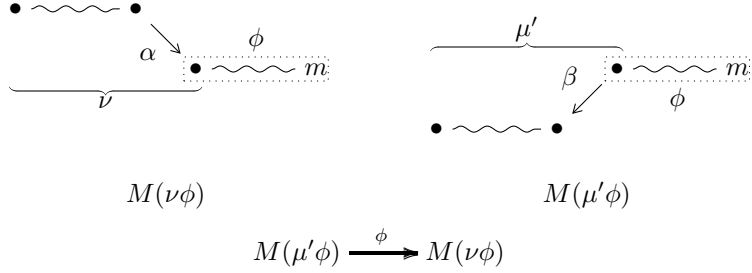
		\begin{center}
			\begin{tikzpicture}[>= angle 60,scale=0.8]
				\node (1) at (0,1){$\bullet$};
				\node (2) at (2,1){$\bullet$};
				\node (3) at (3,0){$\bullet$};
				\node (4) at (5,0){$m$};
				\draw[->] (2) -- node[below left]{$\alpha$} (3);
				\draw[-,decorate, decoration={snake,amplitude=.4mm}] (1) -- (2);
				\draw[-,decorate,decoration={snake,amplitude=.4mm}] (3) -- (4);
				\draw[decorate,decoration={calligraphic brace, amplitude=3pt,mirror}]
				(-.1, -.3) -- (3.1,-.3) node[midway,auto,swap]{$\nu$};
				\draw[dotted] (2.8,-.2) rectangle (5.2,.2); 
				\node at (4,0.5){$\phi$};
				\node at (2.5,-2.1){$M(\nu \phi)$};
				\begin{scope}[xshift = 7cm]
					\node (1) at (0,-1){$\bullet$};
					\node (2) at (2,-1){$\bullet$};
					\node (3) at (3,0){$\bullet$};
					\node (4) at (5,0){$m$};
					\draw[->] (3) -- node[above left]{$\beta$} (2);
					\draw[-,decorate, decoration={snake,amplitude=.4mm}] (1) -- (2);
					\draw[-,decorate,decoration={snake,amplitude=.4mm}] (3) -- (4);
					\draw[decorate,decoration={calligraphic brace, amplitude=3pt}]
					(-.1, .3) -- (3.1,.3) node[midway,auto]{$\mu'$};
					\draw[dotted] (2.8,-.2) rectangle (5.2,.2); 
					\node at (4,-0.5){$\phi$};
					\node at (2.5,-2.1){$M(\mu' \phi)$};
				\end{scope}
			\end{tikzpicture}
			
			$\displaystyle \xymatrix@R=1em @C=1em{
				M(\mu' \phi) \ar@<0.1ex>[rr]^\phi  \ar[rr] \ar@<-0.1ex>[rr]& & M(\nu \phi)}$
		\end{center}
		Hence, after calculations, we have $\mathsf{GenJF}(X) = \mathsf{GenJF}(Y) = (\lambda^q)_{q \in Q_0}$ such that:
		\begin{itemize}
			\item[•] $\lambda^q = (2)$ if $q \in \mathsf{Supp}_0(\phi)$;
			
			\item[•] $\lambda^q = (1)$ if $q \in \mathsf{Supp}_0(\mu) \cup \mathsf{Supp}_0(\nu) \setminus \mathsf{Supp}_0(\phi)$; 
			
			\item[•] $\lambda^q = (0)$ otherwise.
		\end{itemize}
		
		Thus we conclude that $\mathscr{C}_{\mathcal{Q},m}$ is not Jordan recoverable.
	\end{proof}
	
	Finally we still have to prove that if $m$ does not satisfy $(ii)$, then $\mathscr{C}_{\mathcal{Q},m}$ is not Jordan recoverable. 
	
	\begin{prop} \label{not2*}
		Let $m$ be a vertex which does not satisfy $(ii)$. Then $\mathscr{C}_{\mathcal{Q},m}$ is not Jordan recoverable.
	\end{prop}
	
	\begin{proof}[Proof]
		Let $m$ be as supposed. We already know that if $m$ does not satisfy $(i*)$ then $\mathscr{C}_{\mathcal{Q},m}$ is not Jordan recoverable. Let us assume that $m$ satisfies $(i*)$.
		
		From Lemma \ref{1imply0}, we know that $(i*)$ implies $(o)$. Futhermore if we pay attention to the proof of Proposition \ref{not2}, thanks to Lemma \ref{stringstruct0}, we proved in fact that if $m$ satisfies $(o)$ but not $(ii)$, then $\mathscr{C}_{\mathcal{Q},m}$ is not Jordan recoverable (as we said in Remark \ref{0andnot2}). Thus we conclude.
	\end{proof}
	
	Let us recap the proof of the second main result.
	
	\begin{proof}[Proof of Theorem \ref{2ndmain2}]
		Thanks to Proposition \ref{1*and2}, we showed that $(i*)$ and $(ii)$ are sufficient. Then we showed with Proposition \ref{not1*} and Proposition \ref{not2*} that these conditions are necessary. As a result, we have the equivalence we wished for.
	\end{proof}
	
	\section{To go further}
	
	Here are some ideas or questions that we could ask ourselves based on our work.
	\begin{itemize}
		\item[•] \textit{As the proof of the main result uses  string combinatorics, can we extend that result to string algebras or skew-gentle algebras ?}
	\end{itemize}
	We can also think about string algebras and skew-gentle algebras that are algebras which have a similar description of their indecomposable representations. However working with them does not seem as easy as the gentle case. We can at least expect that the description of vertices $m$ such that $\mathscr{C}_{\mathcal{Q},m}$ is canonically Jordan recoverable could be something of the same flavor, with a few more restrictions.
	
	Nonetheless, we can notice that the results can be easily extended to the case of locally gentle algebras.
	\begin{itemize}
		\item[•] \textit{What are exactly all the canonically Jordan recoverable subcategories of modules over a gentle algebra given by a gentle quiver $\mathcal{Q}$ ?}
	\end{itemize}
	From a bit more work over $A_n$-type quivers, we have a conjecture that could characterize all the subcategories that are canonically Jordan recoverable.
	\begin{conj} \label{AnType}
		Let $\mathcal{Q}$ be an $A_n$-type quiver. A subcategory $\mathscr{C}$ of $\mathsf{rep}(\mathcal{Q})$ is canonically Jordan recoverable if and only if for any pair of strings $(\rho,\nu)$ which are associated to indecomposable representations of $\mathscr{C}$, we do not have an arrow $\alpha \in Q_1$ such that $s(\alpha) \in \mathsf{Supp}_0
		(\rho)$, $t(\alpha) \in \mathsf{Supp}_1(\nu)$ and $\alpha \notin \mathsf{Supp}_1(\rho) \cup \mathsf{Supp}_1(\nu)$.\end{conj}
	Thanks to ideas coming from the Example \ref{JRnotCJRex} and the familiarity with the hypothesis $(i)$ of the Theorem \ref{main2}, we can highlight the fact that it is a necessary condition. Actually, we can use the proof of Proposition \ref{not1} in this particular case to conclude. 
	
	However it is still difficult to explain why it is sufficient. Hopefully, if this statement is true, we could expect to extend that result to characterize all the canonically Jordan recoverable subcategories of representations over a gentle quiver.
	\begin{itemize}
		\item[•] \textit{Can we translate our combinatorial condition on the vertex $m$ into conditions in terms of quiver representation theory on the category $\mathscr{C}_{\mathcal{Q},m}$ ?}
	\end{itemize}
	This is not an easy question. At this point, we could expect to get a translation of the conditions from Conjecture \ref{AnType} in terms of not having an extension between two indecomposable representations which do not have a non-zero morphism in either way.
	\begin{itemize}
		\item[•] \textit{Is there a combinatorial map extending the Robinson--Schensted--Knuth correspondance ?}
	\end{itemize}
	Let $Q$ be a Dynkin quiver with $Q_0 = \{1, \ldots, n\}$. In the article \cite{GPT19}, they study the structure of Jordan form data in the minuscule case. For any minuscule vertex $m$, there is a minuscule poset $\mathscr{P}_{\mathcal{Q},m}$ for whose definition we refer to \cite[Section 4.1]{GPT19}. This minuscule poset is equipped with a map $\pi$ to the vertices of $Q$. In particular, $\pi$ satisfies that each fiber $\pi^{-1}(j)$ is totally ordered. 
	\begin{theorem}[\cite{GPT19}] $ $
		\begin{itemize}
			\item[$(i)$] For any $X \in \mathscr{C}_{Q,m}$ and $j \in Q_0$, the number of parts in the partition $\mathsf{GenJF}(X)^j$ is less or equal to the size of the fibre $\pi^{-1}(j)$.
			
			\item[$(ii)$] For $X \in \mathscr{C}_{Q,m}$, define a map $\rho_{Q,m}:\mathscr{P}_{Q,m} \longrightarrow \mathbb{N}$ as follows: the values of $\rho_{Q,m}(X)$ restricted to $\pi^{-1}(j)$ are the entries of $\mathsf{GenJF}(X)^j$, padded with extra zeros if needed, and ordered so that, restricted to $\pi^{-1}(j)$, the function is order-reversing. Then $\rho_{Q,m}$ is order-reversing as a map from $\mathscr{P}_{Q,m}$ to $\mathbb{N}$.
			
			\item[$(iii)$] The map from isomorphism classes of $\mathscr{C}_{Q,m}$ to order-reversing maps from $\mathscr{P}_{Q,m}$ to $\mathbb{N}$, sending $X$ to $\rho_{Q,m}(X)$ is a bijection.
		\end{itemize}
	\end{theorem}
	As an enumerative corollary of the previous result, thanks to an interesting link to reverse plane partitions, the isomorphism classes of representations in $\mathscr{C}_{Q,m}$ have the following generating function.
	\begin{theorem}[\cite{P84}, \cite{GPT19}]
		For $Q$ a Dynkin quiver and $m$ a minuscule vertex, we have
		\begin{center}
			$\displaystyle \sum_{X \in \mathscr{C}_{Q,m}} \prod_{i=1}^n q_i^{\dim(X_i)} = \prod_{u \in \mathscr{P}_{Q,m}} \dfrac{1}{1 - \prod_{i=1}^n q_i^{\dim((M_u)_i)}}$
		\end{center}
		where:
		\begin{itemize}
			\item[•] the sum is over isomorphism classes of representations in $\mathscr{C}_{Q,m}$;
			
			\item[•] in the product, $M_u \in \mathscr{C}_{\mathcal{Q},m}$ is the indecomposable representation of $Q$ corresponding to $u \in \mathscr{P}_{Q,m}$
		\end{itemize}
	\end{theorem}
	One of its combinatorial applications is that we can consider the map $\rho_{Q,m}$ from $\mathscr{C}_{Q,m}$ to reverse plane partitions over $\mathscr{P}_{Q,m}$ to be a generalization of the classical Robinson--Schensted--Knuth (RSK) correspondence. Indeed this map has the same Green-Kleitman invariants \cite{GP17}.
	
	From this perspective, we could expect to generalize the RSK correspondence in the case of gentle algebras, string algebras, or skew-gentle algebras, and even in the case, we work with other Jordan recoverable categories that have more interesting behavior. One can note that our Jordan recoverable categories $\mathscr{C}_{\mathcal{Q},m}$ are similar to the $A_n$ quiver case. So it does not seem unreasonable to say that we can have the same results, without any more interesting things to add. However it could come to be more impactful to get that result after having characterized what are all the Jordan recoverable categories in these new cases.
	
	The reader is invited to have a look into these different problems or things related to them.
	
	\section*{Acknowlegments}
	
	I want first to thank Yann Palu for all I learned about gentle algebras and their combinatorial behavior, and for all his support. 
	
	Then I want to acknowledge the Université du Québec à Montréal and the Institut des Sciences Mathématiques for allowing me to complete my work.
	
	In addition, I would like to thank Claire Amiot, François Bergeron, and Frédéric Chapoton for the couple of advice and help they gave me in their reading of this article as a part of my Ph.D. thesis. 
	
	Last but not least, I owe a great deal of thanks to my Ph.D. supervisor Hugh Thomas who shares with me his time, his advice, his support, his help, to get to this main result and to write this article even if we had to work with some thousands of kilometers between us on the grounds of the sad pandemic situation.

	\nocite{*}
	\bibliography{Jordan_recoverability_of_some_subcategories_of_modules_over_gentle_algebras_v3}

\newcommand{\etalchar}[1]{$^{#1}$}
\begin{thebibliography}{BDM{\etalchar{+}}19}

\bibitem[APS19]{APS19}
Claire Amiot, Pierre-Guy Plamondon, and Sibylle Schroll.
\newblock A complete derived invariant for gentle algebras via winding numbers
  and {A}rf invariants.
\newblock {\em arXiv:1904.02555}, 2019.

\bibitem[AS87]{AS87}
Ibrahim Assem and Andrzej Skowro{\'{n}}ski.
\newblock Iterated tilted algebras of type $\tilde{A}_n$.
\newblock {\em Mathematische Zeitschrift}, page 269–290, 1987.

\bibitem[ASS06]{ASS06}
Ibrahim Assem, Andrzej Skowronski, and Daniel Simson.
\newblock {\em Elements of the Representation Theory of Associative Algebras:
  Techniques of Representation Theory}, volume~1 of {\em London Mathematical
  Society Student Texts}.
\newblock Cambridge University Press, 2006.

\bibitem[BCS19]{BS18}
Karin Baur and Raquel Coelho~Simões.
\newblock {A Geometric Model for the Module Category of a Gentle Algebra}.
\newblock {\em International Mathematics Research Notices},
  2021(15):11357--11392, 07 2019.

\bibitem[BDM{\etalchar{+}}19]{BDMT19}
Thomas Brüstle, Guillaume Douville, Kaveh Mousavand, Hugh Thomas, and Emine
  Yıldırım.
\newblock On the combinatorics of gentle algebras.
\newblock {\em Canadian Journal of Mathematics}, 72(6):1551–1580, Jul 2019.

\bibitem[BR87]{BR87}
Michael Charles~Richard Butler and Claus~Michael Ringel.
\newblock Auslander-{R}eiten sequences with few middle terms and applications
  to string algebrass.
\newblock {\em Communications in Algebra}, 15(1-2):145--179, 1987.

\bibitem[Bre86]{B86}
Sheila Brenner.
\newblock A combinatorial characterisation of finite auslander-reiten quivers.
\newblock In Vlastimil Dlab, Peter Gabriel, and Gerhard Michler, editors, {\em
  Representation Theory I Finite Dimensional Algebras}, pages 13--49, Berlin,
  Heidelberg, 1986. Springer Berlin Heidelberg.

\bibitem[CB89]{CB89}
W.W Crawley-Boevey.
\newblock Maps between representations of zero-relation algebras.
\newblock {\em Journal of Algebra}, 126(2):259--263, 1989.

\bibitem[FGLZ21]{FGLZ21}
Changjian Fu, Shengfei Geng, Pin Liu, and Yu~Zhou.
\newblock On support $\tau$-tilting graphs of gentle algebras, 2021.

\bibitem[GP17]{GP17}
Alexander Garver and Rebecca Patrias.
\newblock Greene--{K}leitman invariants for {S}ulzgruber insertion.
\newblock {\em Electron. J. Comb.}, 26:3, 2017.

\bibitem[GPT23]{GPT19}
Alexander Garver, Rebecca Patrias, and Hugh Thomas.
\newblock Minuscule reverse plane partitions via quiver representations.
\newblock {\em Selecta Mathematica}, 29(3):37, Apr 2023.

\bibitem[Kra88]{K88}
Henning Krause.
\newblock Maps between tree and band modules.
\newblock {\em Journal of Algebra}, 137(1):186--194, 1988.

\bibitem[OPS18]{OPS18}
Sebastian Opper, Pierre-Guy Plamondon, and Sibylle Schroll.
\newblock A geometric model for the derived category of gentle algebras.
\newblock {\em arXiv:1801.09659}, 2018.

\bibitem[PPP17]{PPP182}
Yann Palu, Vincent Pilaud, and Pierre-Guy Plamondon.
\newblock Non-kissing complexes and tau-tilting for gentle algebras.
\newblock {\em Memoirs of the American Mathematical Society}, 274, 07 2017.

\bibitem[PPP18]{PPP18}
Yann Palu, Vincent Pilaud, and Pierre-Guy Plamondon.
\newblock Non-kissing and non-crossing complexes for locally gentle algebras.
\newblock {\em arXiv:1807.04730}, 2018.

\bibitem[Pro84]{P84}
Robert~A. Proctor.
\newblock Bruhat lattices, plane partition generating functions, and minuscule
  representations.
\newblock {\em European Journal of Combinatorics}, 5(4):331--350, 1984.

\bibitem[Sch98]{S98}
Jan Schroër.
\newblock {\em Hammocks for string algebras}.
\newblock PhD thesis, 1998.

\end{thebibliography}
	\bibliographystyle{alpha}
	
\end{document}